\documentclass[reqno]{amsart}

\allowdisplaybreaks
\usepackage{graphicx}
\usepackage{float}
\usepackage{subfigure}
\usepackage{amsfonts}
\usepackage{mathrsfs}
\usepackage{bbold}
\usepackage{fancyhdr}
\usepackage{amsthm}
\usepackage{cases}
\usepackage{amsmath,amssymb,bm}
\usepackage{tikz}
\usepackage{tkz-euclide}

\usepackage{enumitem}
\usepackage{xcolor}
\usepackage{upgreek}
\usepackage{empheq}
\usepackage[
bookmarks=true,  
bookmarksnumbered=true, 
colorlinks=true, pdfstartview=FitV, linkcolor=red, citecolor=blue,
urlcolor=blue]{hyperref}

\def\bb0{\mathbb{0}}

\def\cC{\mathcal{C}}

\def\cD{\mathcal{D}}

\def\EE{\mathbb{E}}
\def\PP{\mathbb{P}}

\def\cF{\mathcal{F}}
\def\cH{\mathcal{H}}
\def\cI{\mathcal{I}}

\def\fI{\mathfrak{I}}
\def\cJ{\mathcal{J}}

\def\bfr{\mathbf{r}}

\def\RR{\mathbb{R}}
\def\cS{\mathcal{S}}
\def\bfs{\mathbf{s}}
\def\vbfs{\vec{\mathbf{s}}}
\def\vbfr{\vec{\mathbf{r}}}
\def\vrho{\vec{\bm{\rho}}}

\def\cS{\mathcal{S}}
\def\bT{\mathbb{T}}

\def\bfx{\mathbf{x}}

\def\Om{\Omega}

\def\fB{\mathfrak{B}}

\newcommand\HH{\mathfrak H}

\def\th{{\theta}}

\def\al{{\alpha}}
\def\ls{{\lesssim}}

\def\be{{\beta}}

\def\la{{\lambda}}
\def\si{{\sigma}}

\def\De{{\Delta}}

\def\Om{{\Omega}}
\def\al{{\alpha}}

\def\be{{\beta}}

\def\ga{{\gamma}}

\def\De{{\Delta}}

\def\si{{\sigma}}

\def\la{{\lambda}}

\def\vare{{\varepsilon}}

\def\rank{{\rm rank }}
\def\beq{\begin{eqnarray}}
\def\eeq{\end{eqnarray}}
\def\bsp{\begin{equation}
\begin{split}}
\def\esp{\end{split}
\end{equation}}
\def\ba{\begin{align}}
\def\ea{\end{align}}
\def\add{{\hbox{\tiny  add}} }

\newcommand{\ep}{\varepsilon}

\newcommand{\1}{{\bf 1}}

\newcommand{\lc}{\left(}
\newcommand{\rc}{\right)}

\newcommand{\lt}{\left }
\newcommand{\rt}{\right}

\newtheorem{Thm}{Theorem}

\newtheorem{lemma}[Thm]{Lemma}
\newtheorem{proposition}[Thm]{Proposition}

\newtheorem{corollary}[Thm]{Corollary}

\newtheorem{definition}[Thm]{Definition}
\theoremstyle{Def}

\newtheorem{example}[Thm]{Example}

\newtheorem{remark}[Thm]{Remark}

\numberwithin{Thm}{section}

\numberwithin{equation}{section}

\begin{document}

\title{In search of necessary   and  sufficient  conditions to  solve  the parabolic Anderson  model with fractional Gaussian   noises}



\author{Shuhui Liu}
\address{School of Mathematics, Shandong University, Jinan, Shandong 250100, China}
\address{Department of Applied Mathematics, The Hong Kong Polytechnic University, HungHom, Kowloon, Hong Kong}
\email{shuhuiliusdu@gmail.com}
\thanks{SL was supported by the China Scholarship Council.}

\author{Yaozhong Hu}
\address{Department of Mathematical and Statistical Sciences, University of Alberta, Edmonton, AB T6G 2G1, Canada}
\email{yaozhong@ualberta.ca}
\thanks{YH was supported by the NSERC discovery fund and a startup fund of University of Alberta.}

\author{Xiong Wang*}
\address{Department of Mathematics, Johns Hopkins University, Baltimore, 21218, USA}
\email{xiongwang@jhu.edu}
\thanks{*Corresponding author: xiongwang@jhu.edu}


\date{}



\begin{abstract}
This paper attempts  to obtain necessary and sufficient conditions  to solve   the parabolic Anderson model with fractional Gaussian noises:  $\frac{\partial}{\partial t}u(t,x)=\frac{1}{2}\Delta u(t,x)+u(t,x)\dot{W}(t,x)$, where $ {W}(t,x)$ is the  fractional Brownian field with temporal Hurst parameter  $H_0\in [1/2, 1) $  and spatial Hurst parameters $H$ $ =(H_1, \cdots, H_d)$ $ \in (0, 1)^d$,  and $\dot{W}(t,x)=\frac{\partial ^{d+1}}{\partial t \partial x_1 \cdots \partial x_d}W(t,x)$. When $d=1$  and     when  $(H_0,H)\in(\frac 12,1)\times(\frac 1{20},\frac 12)$ we show that  the condition  $2H_0+H>5/2$  is necessary and sufficient to ensure the existence of a unique solution for the parabolic Anderson Model.  When $d\ge 2$,   we find the necessary and sufficient condition on the Hurst parameters so that each chaos of the solution candidate is square integrable.
\end{abstract}

\maketitle

\section{Introduction  and main results}

In this work, we study the solvability (i.e., existence and uniqueness) problem of the following stochastic heat equation on 
the Euclidean space $\RR^d$, also known as the parabolic Anderson model (PAM):
	\begin{equation}\label{PAM}
		\left\{\begin{aligned}
			&\frac{\partial}{\partial t}u(t,x) =\frac{1}{2}\Delta u(t,x)+u(t,x)\dot{W}(t,x)\,, ~t>0, {x\in\RR^d} \,;\\
			& u(0,x) =u_0(x)\,,
		\end{aligned}\right.
	\end{equation}
	where $\De=\sum_{i=1}^d \frac{\partial ^2}{\partial x_i^2}$ is the Laplacian on $\RR^d$, the process $\{W(t,x)\,, t\ge 0\,, x\in \RR^d\}$ is a centered fractional Gaussian field of Hurst parameters $(H_0, H_1,\cdots, H_d)$
	and $\dot W(t,x)=\frac{\partial ^{1+d}}{\partial t \partial x_1 \cdots \partial x_d} W(t,x)$.  This means that formally the mean of $\dot W(t,x)$   is zero  and its   covariance function  is    given    by
	\begin{equation}\label{Def.Cov_W}
		\text{Cov} (\dot W(t,x),\dot W(s,y))=\gamma_0(s-t)\gamma(x-y) 
	\end{equation}
with 
	\begin{equation}\label{Def.Cov_frac}
		\gamma_0(t)=c_{H_0}|t|^{2H_0-2}\,,\quad \gamma(x)=\prod_{j=1}^d c_{H_j }|x_j|^{2H_j-2} \,,\quad t\in \RR,
		x=(x_1, \cdots, x_d)\in \RR^d
	\end{equation}
	where $c_{H_j }=H_j (2H_j -1)$ for $j=0,\cdots, d$.
Throughout this work we assume that  the Hurst parameters $(H_0, H_1, \cdots, H_d)$ satisfy 
	\begin{align*}
		{ H_0\in[1/2,1)\,,\text{ and }H_j\in(0,1)  \quad \forall j=1,\cdots,d.}
	\end{align*}
	
The case $H_0 =1/2$ corresponds to  
  $\gamma_0(t)=\delta(t) $ and in this case the noise is called time white; the case $H_1=\cdots=H_d=1/2$ corresponds to  
  $\gamma (x)=\delta(x) $ and in this case the noise is called space white.  When the noise is space time white  
  the square integrable  solution exists only when 
  the space dimension is one. In this case (and when the initial
  condition is the Dirac delta function)  
the seminal paper \cite{Hairer} connects the equation \eqref{PAM} with the Kardar-Parisi-Zhang (KPZ) equation via the  Hopf-Cole transformation  $h(t,x)=\log(u(t,x))$ and develops the theory of regularity structures. In the last decades, the parabolic Anderson model has received great attention partly due to its connection with the KPZ equation. Many sharp properties of the solution to \eqref{PAM} have been obtained for general Gaussian noise,  including
the space time white noise.
For further reading, we recommend the works of \cite{Balan, Chen2016, Corwin2020, HL2022, HWIntermittency, KKX2017, Tsai2022} and the references  therein. 

Equation
 \eqref{PAM} depends only on the initial condition and the covariance structure of noise $\dot W$. If we assume that the initial condition is as nice as needed (e.g. $u_0 \equiv 1$),  then the solvability and the  properties of the solution to  \eqref{PAM} are completely determined by   the covariance structure of $\dot W$. 
 It is interesting to ask under what conditions on the covariance structure of the noise,  the equation \eqref{PAM} has a unique solution. Several progresses 
 have been made along this direction, mostly in the form of   sufficient condition,
  which will be recalled in the next subsection. 
Now it is natural to ask if such condition is also necessary or not.  If not, it is interesting to find    conditions  that are  both necessary and sufficient.   In this paper
we shall present some partial results for the above problem.   

 Usually there are two different interpretations of  the product  $u(t,x)\dot W(t,x)$
		in \eqref{PAM} used so far in literature.  One is in the Stratonovich sense (or pathwise sense), and the other  one  is in the    It\^o/Skorohod  sense.  We chose the latter one {for \eqref{PAM}}  since 
		it enables us  to  immediately  express the formal  solution candidate through its It\^o-Wiener chaos expansion:
		\begin{equation}\label{Def.iwce}
			u(t,x)=\sum_{n=0}^{\infty}u_n(t,x)=\sum_{n=0}^{\infty} \frac{1}{n!} I_n(f_n(\cdot,t,x))\,, \quad {\text{ for any } (t,x)\in \RR_+\times\RR^d}
		\end{equation}
	where $f_n$ is  given by \eqref{e.3.3}  in the next section    through the heat kernel associated with 
\eqref{PAM} 	and $I_n(f_n(\cdot,t,x))$ is the multiple It\^o-Wiener integral
with respect to $ f_n(\cdot,t,x) $.  The advantage of using the chaos expansion 
lies in the fact that multiple   It\^o-Wiener integrals of different orders are orthogonal.   It is easy to verify 
	(e.g. \cite{CH2021})  that {\it the PAM  \eqref{PAM} has a unique square integrable solution if and only if the above chaos expansion is convergent in $L^2(\Omega)$  for all $(t, x)\in
	\RR_+\times \RR^d$}.  {When the chaos expansion is convergent, one has}
\begin{equation}\label{e.square}
	\EE \left[u(t,x)^2\right] = \sum_{n=0}^{\infty}   \EE \left[ {u_n(t,x)^2} \right]\,. 
\end{equation}	
		From the above discussion we see that  the solvability 
		problem of \eqref{PAM} is decomposed into the following two consecutive questions.		
\begin{enumerate}[leftmargin=8mm]
\item[{(Q1)}] {What are the necessary and sufficient conditions for the finiteness of $\EE \left[ u_n^2(t, x) \right]$
for all positive integer $n$? This involves the challenging 
problem of finding necessary and sufficient conditions  for the integrability of    singular multiple integrals presented in Proposition \ref{est_un} of Section \ref{s.spatial}. } 
\item[{(Q2)}] {
Under what  necessary and sufficient conditions 
Equation \eqref{e.square} is convergent? In Section \ref{Sec5.Conv}, we answer this  question in one dimensional case by deriving proper bounds for $\EE \left[ u_n^2(t, x)\right]$ as a function of $n$.}
\end{enumerate} 

\subsection{Main results} 

In the course of completing these two tasks,   we will see that the cases of Hurst parameters greater or less than  $3/4$ need to be treated differently.
This is a bit contrary to the conventional division which usually divides the region  of   
Hurst parameters into the region that  the Hurst parameter
is  greater   than  $1/2$
and region that  the Hurst parameter  is  less    than  $1/2$.  For this reason,  and  without loss of generality, we assume that the Hurst parameters are so arranged   that there is an integer $d_*$  between $0$ and $d$ such that
   $H_k<\frac34$ for $k=1,2,\cdots,d_*$ and $H_k\ge \frac34$ for $k=d_*+1,\cdots,d$.  The case $d_*=0$ means that 
   all Hurst parameters are   greater than or equal to 
   $3/4$ and $d_*=d$ means that 
   all Hurst parameters are   less  than  
   $3/4$. We also denote by $|H|$ the sum of all spatial  Hurst parameters, by $ H_*$ the sum of all Hurst parameters that are less than $3/4$, and by $ H^*$ the sum of all Hurst parameters that are greater than or equal to  $3/4$.  Namely, we denote
\begin{equation}\label{sum_H}
	\begin{cases}
		H_1, \cdots,  H_{d_*}<3/4,\quad H_{d_*+1}, \cdots, H_d\ge 3/4\,,\\
		H_*:=H_1+\cdots+H_{d_*},\ \ \ H^*:= H_{d_*+1} +\cdots+
		H_d\,,\\
		d^*=d-d_*\,,\quad |H|:=H_1+\cdots+H_d.\\
	\end{cases}
\end{equation}

Concerning the  above first question (Q1), {we have the following result.}  
\begin{Thm}\label{main_result1}
	Suppose $u_0\equiv1$ and $H_0\ge\frac12$.
	Let $u_n$ be the $n$-th chaos candidate of the solution to \eqref{PAM}, defined by \eqref{e:3.3}-\eqref{e.3.3} in the next section.  Then 
	$\EE[u_n(t,x)^2]<+\infty$ {{for any $(t,x)\in\RR_+\times\RR^d$}} and all $n\ge  1$ if and only if {all of the following conditions hold:}{
		\begin{subequations}\label{cond.0<H<1}
			\begin{align}[left={\empheqlbrace}]
				&H_*>\frac34 d_*-1,\label{cond.0<H<1_a}\\
				& |H|+2H_0>d,\label{cond.0<H<1_b}\\
				&H_*+2H_0>\frac34 d_*+\frac12,\label{cond.0<H<1_c}\\
				&|H|+2H_*+4H_0>d+\frac32 d_*.\label{cond.0<H<1_d}
			\end{align}
	\end{subequations}}
\end{Thm}
The initial condition $u_0\equiv1$ may be extended to any initial condition that is uniformly bounded below and bounded above.  
Using the It\^o isometry $\EE[u_n(t,x)^2] $ is expressed as a multiple integral  involving the spatial and temporal variables. Bounding this multiple integral  by integrating the space variables
will yield a multiple integral  of the form 
\eqref{u_n_est} involving  multiple  integral  of powers of some singular kernels.
This multiple integral appears elementary. However, determining the necessary and sufficient conditions for the exponents to ensure the integral's finiteness is a highly complex problem. 
There are some studies   about similar integrals in analysis
(e.g., \cite{Shi, Wu, Zhou}).  However, they seem bard to  be applied to
the above case.  We shall divide the   integral \eqref{Mainterm_d=2} associated with   $\EE \left[ u_n(t,x)^2 \right]$ into two distinct cases determined  by a new mechanism \eqref{Algo}. Then,  the  proof relies on  applications of both Hardy-Littlewood-Sobolev and  H\"older-Young-Brascamp-Lieb inequalities. 
The detailed proof will be given in Section  \ref{s.proof}. 


{If $H_k<3/4$ for all   $k=1,\cdots,d$, then we have} $|H|=H_*$ and $d=d_*$. The conditions \eqref{cond.0<H<1_a}-\eqref{cond.0<H<1_d} in \eqref{cond.0<H<1} can be slightly simplified to:
\begin{corollary}  If $u_0\equiv1$, $H_0\ge\frac12$ and $H_k<3/4$ for all   $k=1,\cdots,d$,
	then the necessary and sufficient condition  so that  $\EE[u_n(t,x)^2]<+\infty$ {for any $(t,x)\in\RR_+\times\RR^d$} and  all $n\ge   1$ becomes
	\begin{equation}\label{cond.H<3/4}
		\begin{cases}
			&H_*>\frac34 d-1,\\
			&H_*+2H_0>(d\vee \frac{3d+2}{4}),\\
			&3H_*+4H_0>\frac52 d.
		\end{cases}
	\end{equation}
\end{corollary}

When $d=1$, $H_0\ge\frac12$ and $H=H_1<\frac12$, the first condition \eqref{cond.0<H<1_a}  is trivial. The second and the last are implied by the third. Thus, we reduce Theorem \ref{main_result1}  to the following theorem. 
	\begin{Thm}\label{c.4.2}  
	If  $d=1$, $u_0\equiv1$ and $H_0\ge\frac12$, $0<H<\frac12$,    then  the necessary and sufficient condition so that   $\EE[u_n(t,x)^2]<+\infty$ {for any $(t,x)\in\RR_+\times\RR^d$} and all $n\ge  1$ is
	\begin{equation}\label{Ineq:Key_Cond}
	2H_0+H>5/4\,. 
\end{equation}	 
\end{Thm}

\begin{example}\label{example.1.4} 
It is easy to verify that when $d=3$, $H_0=1$,  $H_1=H_2=H_3=1/2$,  the condition 
\eqref{cond.H<3/4} is satisfied. Thus, if  $d=3$ and when the noise is time independent and
space white, then $\EE[u_n(t,x)^2]<\infty$ {for any $(t,x)\in\RR_+\times\RR^d$} and all $n$ (which is known in \cite[Theorem 4.1]{Hu2002}).  
However, also   in \cite[Theorem 4.1]{Hu2002} it is  shown that  in this case 
    $\sum_{n=0}^\infty  \EE[u_n(t,x)^2]=\infty$  {for any $(t,x)\in\RR_+\times\RR^d$} 
    if the initial condition
    $u_0(x)\ge c $ for some constant $c>0$.
This  stresses the point   that $ \EE[u_n(t,x)^2]<\infty$ for all $n$ 
does not    automatically   imply  $\sum_{n=0}^\infty  \EE[u_n(t,x)^2]<\infty$.  
\end{example}

The following results are byproducts while proving Theorem \ref{main_result1}.

\begin{proposition}\label{main_cor4}
	Suppose $u_0\equiv1$, $H_0\ge\frac12$ and suppose {all the conditions} in \eqref{cond.0<H<1} hold.
	\begin{enumerate}
		\item[(i)]   If      $H_0+H_*>\frac{3d_*}{4} $  and if $|H|>d-1$,   then
		the solution is in   $L^p(\Omega)$ for any $p\ge 1$.
		Moreover, there are two positive constants $C_{1, H}$ and $c_{2,H}$, independent of $p$,   $t$ and $x$  such that 
		\begin{equation}
			\begin{split}
				\EE\lt[|u(t,x)|^p\rt]\le  &C_{1, H} \exp
				\left[ c_{2, H} p^{\frac{|H|-d+2} {|H|-d+1}}  t  ^{\frac{|H|+2H_0-d}{|H|-d+1}}\right]
			\end{split}
		\end{equation}
		{for all $(t,x)\in \RR_+\times \RR^d$}.
		\item[(ii)] If      $H_0+H_*>\frac{3d_*}{4} $   and if $|H|=d-1$,  then there is a $t_0>0$ such that
		when $t<t_0$, $\EE[u(t,x)^2]=\sum_{n=0}^\infty \EE[u_n(t,x)^2]<+\infty$ for $x\in\RR^d$.
	\end{enumerate}
\end{proposition}
{The critical time $t_0$ in Proposition \ref{main_cor4} (ii) is usually called  the blow-up time, see for example \cite{Hu2002}. In \cite[(1.9)]{CDOT2021}, the authors show that $t_0=t_0(p)=\frac{C_{H_0,H}}{(p-1)^{2H_0-1}}$ under the critical condition \cite[(1.8)]{CDOT2021}. Quastel, Ramirez, and Virag \cite{QRV2024} use an elementary version of the Skorokhod integral to define the solution at all times, including $t > t_0$. They construct $u$ as a randomized shift or as the free energy of an undirected polymer in a random environment. }  

When $H_0=1/2$, {the first condition \eqref{cond.0<H<1_a} and the fourth condition \eqref{cond.0<H<1_d}} are entailed by the second condition \eqref{cond.0<H<1_b} and third one \eqref{cond.0<H<1_c}. Therefore, combining Theorem \ref{main_result1} and Proposition \ref{main_cor4}, we have the following proposition.
\begin{proposition}\label{main_prop5}
	Suppose $u_0\equiv1$ and $H_0=\frac12$, namely, the noise is white in time. Then the necessary and sufficient condition for $\EE[u(t,x)^2]<+\infty$ {for any $(t,x)\in\RR_+\times\RR^d$}
	is
	\begin{equation}\label{H0=1/2}
		|H|>d-1\quad\hbox{ and }\quad
		H_*>\frac34d_*-\frac12.
	\end{equation}
\end{proposition}
This means that when the noise is white in time, the convergence of the It\^o-Wiener chaos expansion is equivalent to the finiteness of each chaos.

{In general the above second question (Q2) is more difficult to answer since we need more precise bound of $\EE \left[ u_n(t,x)^2\right]$ in terms of $n$ so that the series \eqref{e.square}  is summable 
(Namely  $\sum_{n=0}^\infty \EE \left[ u_n(t,x)^2\right]<\infty$). Proposition \ref{main_cor4}   provides a sufficient condition for the convergence of the It\^o-Wiener chaos expansion \eqref{e.2.series}.  It is interesting to know if this condition is necessary or not. Unfortunately,  as we shall see in this manuscript,  it is not necessary. Next,  we enlarge the known range of Hurst parameters for which the It\^o-Wiener chaos expansion converges.   We shall focus   on $d=1$ 
since  the problem in  general dimension  case is much more difficult.   } 


It is known  (from  subsection \ref{Subsec1.2})   that when $d=1$ and $H_0>1/2$,
a sufficient condition for  the  convergence of the
It\^o-Wiener chaos expansion is $H_0+H_1>3/4$ and a necessary condition
is $2H_0+H_1>5/4$.  These two conditions do not coincide
so we don't know which one of them is both necessary and sufficient, 
or none  of them is.  It is natural to seek a condition that is both necessary and sufficient.  This problem seems hard. First, let us point out that the condition \eqref{cond.0<H<1}  may not be sufficient for the convergence of the It\^o-Wiener chaos expansion.  In fact, assuming the initial condition $u_0(x)=1$ and $d=3$, it is proved in  \cite{Hu2002} (see also Example \ref{example.1.4})  that if  the noise is time independent ($H_0=1$) and space white ($H_1=H_2=H_3=1/2$) then  $\EE \left[u_n(t,x)^2\right]<\infty$  for each $n$ (the conditions in Theorem \ref{main_result1} are satisfied),  but
$\sum_{n=0}^\infty \EE \left[u_n(t,x)^2\right]=\infty$  for all  $(t, x)\in \RR_+\times \RR^3$.  The second main result of {this paper} is  to establish that the existing necessary condition ($2H_0+H_1>5/4$) for the convergence of the It\^o-Wiener chaos expansion is also sufficient under some circumstances.

%
%
%

 \begin{figure}[htbp]
 	\centering
 	\includegraphics[width=0.8\textwidth]{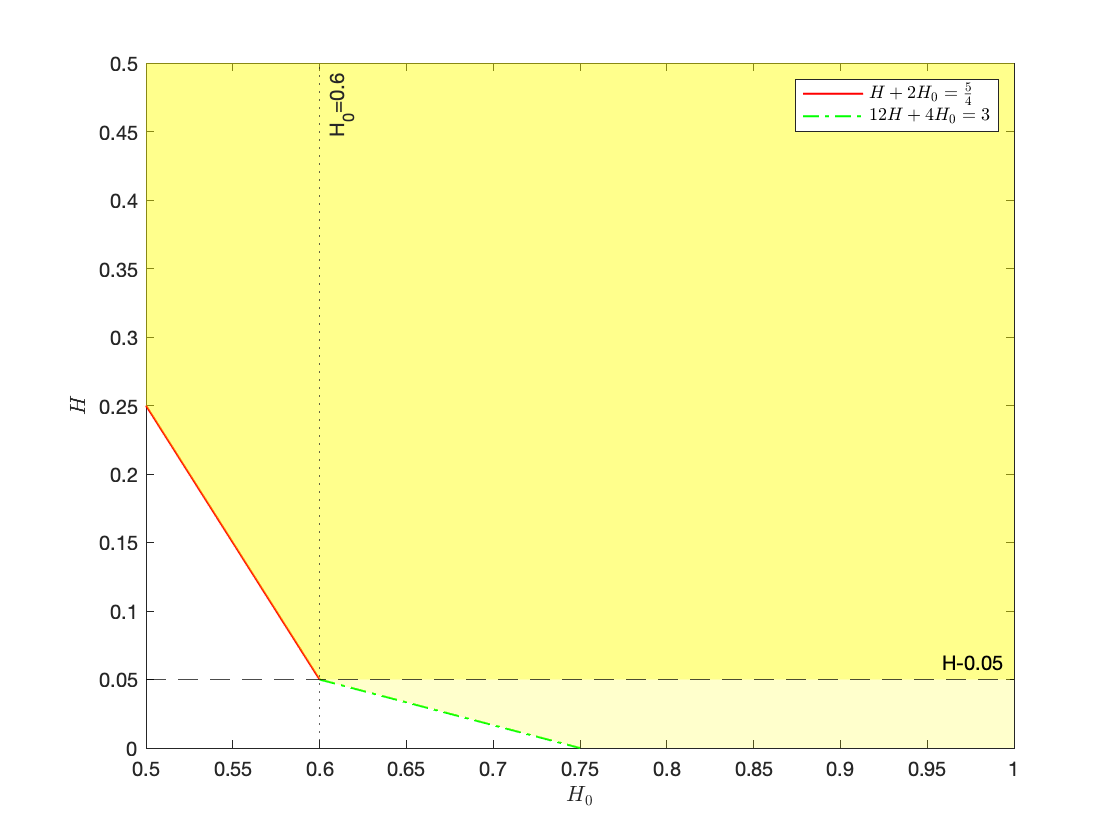}
 	\caption{The dark yellow region is for region  $\mathcal{A}_1$ and the light yellow region is for region  $\mathcal{A}_2$}
 	\label{Fig.Area1+2}
 \end{figure}
%

%


\begin{Thm}\label{thm_conver}
 Let $u(t,x)$ be the solution candidate given by  \eqref{Def.iwce}.
Suppose $d=1$, $u_0\equiv1$ and $H_0>\frac12$,  $H=H_1<\frac12$. If $(H_0, H )\in \mathcal{A}_1\cup\mathcal{A}_2$, where
\begin{align}
 \mathcal{A}_1=&\lt\{(H_0, H )\in(1/2,1)\times(1/20,1/4): \  2H_0+H>5/4 \rt\}\,, \label{area_1} \\
\mathcal{A}_2=&\lt\{(H_0, H )\in(1/2,1)\times(0,1/20): \  4H_0+12H>3 \rt\}\,, \label{area_2}
\end{align}
then {for any $(t,x)\in\RR_+\times\RR^d$}
\begin{align}\label{Fin_Result}
\EE[u(t,x)^2]=\sum\limits_{n=1}^{\infty}\EE[u_n(t,x)^2]<+\infty.
\end{align}
\end{Thm} 
Compared with Theorem \ref{main_result1}, the proof of the above 
theorem needs   some sharp uniform bounds  of $\EE[u_n(t,x)^2]$ {for any $(t,x)\in\RR_+\times\RR^d$} as $n\to  \infty$.  The technique in the proof of Theorem \ref{main_result1} is insufficient. The details of the proof are given 
 in Section \ref{Sec5.Conv}.

  In Figure \ref{Fig.Area1+2}, the domain \eqref{area_1} is shaded in light yellow and domain \eqref{area_2} is shaded in dark yellow. 
{This combined with Theorem \ref{thm_conver} implies immediately the following corollary.}
\begin{corollary} 
Suppose $d=1$, $u_0\equiv1$. On the region $(H_0, H )\in(1/2,3/5)\times(1/20,1/4)$ the condition $2H_0+H>5/4$ in \eqref{Ineq:Key_Cond}
	serves as a necessary and sufficient condition for the solvability of \eqref{PAM}.
\end{corollary}    

The condition $4H_0+12H>3 $  can be written as
\[
2H_0+H>\frac14 +\frac53H_0 \,,
\]
which is implied by $2H_0+H>\frac54$  and  $H_0\le 3/5$.
Similarly, $4H_0+12H>3$ is also implied by
$2H_0+H>\frac54$  and  $H\ge 1/20$.
This combined with the above corollary   implies   immediately    that
\begin{corollary}
	Suppose  $d=1$.  If $(H_0,H)\in(\frac 12,1)\times(\frac 1{20},\frac 12)$ or $(H_0,H)\in(\frac 12,\frac 35)\times(0,\frac 12)$,  then the necessary and sufficient condition for the It\^o-Wiener chaos expansion \eqref{e.2.series} to converge (namely for Equation   \eqref{PAM} to be   solvable) 
	is $2H_0+H>\frac54$.
\end{corollary}

 \subsection{Comparison with existing results}\label{Subsec1.2}

%
%

The solvability of equation \eqref{PAM}   depends completely on the covariance kernels $\ga_0(\cdot)$ and $\ga(\cdot)$ when the initial condition is given and is assumed to be nice.  When  the temporal kernel $\gamma_0(\cdot)=\delta_0(\cdot)$
is the Dirac delta function, and  when the spatial kernel is positive and   $\gamma(x)=\int_{\RR^d} e^{ix\xi }\mu(d\xi)$, $x\in \RR^d$,    is represented by the Fourier transform of a positive measure $\mu$ on $\RR^d$ satisfying   
\begin{equation}\label{Dalang}
	{ \int_{\RR^d}\frac{\mu(d\xi)}{1+|\xi|^2}<\infty \quad \hbox{(Dalang's condition)}\,, }
\end{equation}
the equation \eqref{PAM} is solvable  (e.g., \cite{Dalang}). 
In the literature to obtain the necessary and sufficient 
condition,  one usually takes the temporal kernel as   $\gamma_0(t)=c_{\alpha_0}|t|^{-\alpha_0}$  
with  some $\al_0>0$ and with some constant  $c_{\alpha_0}\in \RR$ 
(one identifies the Dirac delta function case as $\al_0=1)$. As for  the spatial covariance function $\gamma(\cdot)$, three   cases are commonly studied: (i)
Riesz kernel, i.e., $\gamma(x)=c_{\alpha,d}|x|^{-\alpha}  $ or equivalently $\mu(d\xi)=C_{\alpha,d}|\xi|^{-(d-\alpha)}\prod_{j=1}^d d\xi_j$ for some $ \alpha \in  \RR$ and some positive constants $c_{\alpha,d},C_{\alpha,d}\in \RR$; (ii) fractional kernel, i.e., $\gamma(x)=\prod_{j=1}^d c_{\alpha_j}|x_j|^{-\alpha_j}$ or $\mu(d\xi)=\prod_{j=1}^d C_{\alpha_j}|\xi_j|^{-(1-\alpha_j)}d\xi_j $ for some $ \alpha_j>0$ and some positive constants $c_{\alpha_j},C_{\alpha_j}\in \RR$.  

For the fractional noise, the relation between $\al_i$ and the Hurst parameters is given by 
$\al_i=2-2H_i$, $i=0, 1, \cdots, d$.  In the above  particular cases of Riesz or fractional kernels, 
 Dalang's condition becomes $\al_0=1$ and  
\begin{equation}\label{Dalang2}
	\begin{cases}
		0<\alpha<2\,,  & \text{Riesz kernel noise};\\
	\al_1, \cdots, \al_d \in (0, 1)\,,\ 	\sum_{j=1}^d \al_j <2\,, & \text{fractional noise}\,. 
	\end{cases}	
	\end{equation} 
  When $\gamma_0(\cdot)$ is locally   integrable, the Dalang's condition \eqref{Dalang} is also proved 
	(e.g. \cite{HHNT}) to be sufficient for the solvability of \eqref{PAM}  in the general   Gaussian noise
	setting. Moreover, when $\ga_0(t)=c_\alpha |t|^{-\alpha_0}$ and   the spatial covariance $\gamma(\cdot)$ is given by the Riesz potential, the necessity of \eqref{Dalang} has been verified in \cite{BC2014}. Notice that in terms of Hurst parameters,  the second one in  Dalang's condition \eqref{Dalang2}   becomes $H_1, \cdots, H_d\in (1/2, 1)\,,  \ \sum_{j=1}^d H_j>d-1$.   As a comparison, we draw attention to another random evolution model,  the hyperbolic Anderson model (HAM) (i.e.,  we replace $\frac{\partial}{\partial t}$ by $\frac{\partial^2}{\partial t^2}$ in \eqref{PAM}). It is known that Dalang's condition \eqref{Dalang} is a sufficient condition for the solvability of   HAM. The readers are referred to \cite{Balan} and references therein for details. During the preparation of this article, Chen, Deya, Song, and Tindel in \cite{CDST2021} obtain that
	\begin{equation}\label{CDST_Condi}
\int_{\RR^d}  \frac{1}{1+|\xi|^{3-\alpha_0}}   \mu(d\xi)<\infty
\end{equation}
is the necessary and sufficient condition for HAM to admit a unique Skorohod solution when $\alpha_0\in(0,1)$. Moreover, in the case of Riesz kernel or (regular) fractional noise, \eqref{CDST_Condi} is equivalent to $\alpha_0+\alpha<3$ or $\alpha_0+\sum_{j=1}^d \alpha_j<3$.

Let us emphasize that for the above mentioned Gaussian noises,
	it is critical to assume all $\al_j\in (0, 1) $, and the solvability under this condition is 
	now clear. Conventionally, the case that  all the Hurst parameters exceed $1/2$ (namely,  $\al_j\in (0, 1)$
	for all $j=1, \cdots, d$)  is referred to   the \emph{regular} case. 
Otherwise (namely,    $\al_j>1$  for some  $j=1, \cdots, d$), 
it is termed \emph{rough} case. {The \emph{rough case}} is more intriguing and poses significantly greater challenges due to the fact that the spatial covariance function $\gamma(\cdot)$ is no longer locally integrable and is no longer positive. 
	 
	 There are very limited results in the case when time is rough, i.e., $\al_0>1$ (or equivalently $H_0<1/2$).   In the case of fractional noise, 
	 let us mention the work   \cite{HN09}, where {it is proved that all chaoses} of the solution candidate to \eqref{PAM} exist in $L^2(\Omega)$ when $d=1$, $H_1=1/2$ and $3/8<H_0<1/2$.  This result is further strengthened by   Chen, where  
	 the domain of solvability   \cite[(1.19)]{Ch2} is established as ($d=1$):  
		\begin{equation}\label{Eq:Chen_Cond}
			\begin{cases}
				H_0\geq \frac 12 \text{ and } H_1< \frac 12 : \text{ solvable if } H_0+H_1 > \frac{3}{4} \,; \\
				H_0 < \frac 12 \text{ and } H_1 \geq \frac 12 : \text{ solvable if } 4H_0+H_1 > 2 \,; \\
				H_0< \frac 12 \text{ and } H_1< \frac 12 : \text{ solvable if } 2H_0+H_1 > \frac{5}{4} \,.
			\end{cases}
		\end{equation}
		Notice that when $H_1=1/2$, the above second condition 
		is exactly  $ H_0>3/8$.
		Furthermore,  if the noise {is more smooth in the spatial variable},  then $H_0$ can be arbitrarily
		between $0$ and $1$  (see, e.g., \cite{CHKN18, HL17}). 
		
		When $H_0\ge 1/2$, and when some of the spatial Hurst parameters $(H_1,\cdots,H_d)$ are less than $1/2$  while others are greater than $1/2$,  the solvability problem becomes much more complex.
%
   Some notable achievements have been made  
when $H_0=\frac 12$  (i.e., the noise is white in time) and $d=1$.  In this case, it is known that $H_1>1/4$ is necessary and sufficient for the parabolic Anderson model (or hyperbolic Anderson model) to be solvable. 
We refer to  the works of \cite{BJQ2015,BJQ2017,HHLNT2,HHLNT1,HW2019,HLW2021} and references therein
for detailed discussions.  

For the more general case when the noise is rough in space, the best results 
up to date seem to be in  \cite{CH2021}, where both sufficient and necessary conditions are obtained for some special cases.  More specifically, 
it is proved   
\begin{enumerate}
	\item[(i)]  When $H_0=1/2$, namely when the noise is time  white,   the necessary and sufficient condition for the solvability of \eqref{PAM} (e.g., \cite[Theorem 1]{CH2021}, see also our previous Proposition \ref{main_prop5})  is  $|H|:=\sum_{i=1}^d H_i>d-1$  and $H_*>\frac34d_*-\frac12$.  Notice  that the condition $H_*>\frac{3}{4}d_*-\frac{1}{2}$, which is missing in \cite[Theorem 1]{CH2021},  is addressed  in    Section \ref{s.proof} of the present paper.  In fact,  when $d=1$ the above condition
	 becomes $H:=H_1>1/4$. 
	
	\item[(ii)] If $H_0>1/2$, then  a sufficient condition for \eqref{PAM} to be solvable (e.g. 
	 \cite[Theorem 3 (i)]{CH2021}) is
	\begin{equation}
		\begin{cases}
			H+H_0 >\frac{3}{4} &\qquad \hbox{if $d=1$}\,, \\
			|H|>d-1 &\qquad \hbox{if $d\ge 2$}\,. \\
		\end{cases}
		\label{e.sufficient}
	\end{equation}
	\item[(iii)] If $H_0>1/2$,    then  a necessary
	condition for \eqref{PAM} to be solvable (e.g.  \cite[Theorem 3 (ii)]{CH2021})  is as follows
	\begin{equation}
		\begin{cases}
			H +2H_0>\frac{5}{4}  &\qquad \hbox{if $d=1$}\,, \\
			|H|+2H_0 >\frac{3d+2}{4} &\qquad \hbox{if $d\ge 2$}\,. \\
		\end{cases}
		\label{e.necessary}
	\end{equation}
\end{enumerate}   
 On the other hand, when $H_0=1/2$ and $d>1$,  it is pointed out in \cite{CH2021} that for \eqref{PAM} to be solvable, the number of  Hurst parameters $(H_1,\cdots,H_d)$ that are less than $1/2$ can be at most $1$.



One notices that the sufficient condition \eqref{e.sufficient}
	and the necessary condition \eqref{e.necessary} are different, and this poses an intriguing open question: In situations where $H_0>\frac{1}{2}$, is \eqref{e.sufficient} a necessary condition, or is \eqref{e.necessary} sufficient? Put differently, is it possible to find a condition that is both necessary and sufficient?  
This current research is driven by the pursuit of an answer 
 to this problem. 
More specifically, our first main result, as stated in Theorem \ref{main_result1},  
gives a necessary and sufficient condition ensuring that $\EE \left [u_n(t,x)^2\right]<\infty$ holds for all values of $n$ and {$(t,x)\in\RR_+\times\RR^d$}. Here, $u_n$ denotes the $n$-th chaos expansion of the solution candidate to \eqref{PAM}.  
To obtain the necessary and sufficient condition for  \eqref{e.square} to be convergent, we focus on   $d=1$.   Our second primary result, Theorem \ref{thm_conver} asserts that the condition $H+2H_0>\frac{5}{4}$ specified in the first line of  \eqref{e.necessary} serves as both necessary and sufficient criterion for the solvability of \eqref{PAM} in two pieces of domains: when $H_0<\frac{3}{5}$ or when $H>\frac{1}{20}$. 
We believe that the above additional condition  $H_0<\frac{3}{5}$ or   $H>\frac{1}{20}$ 
is due to some technicality and   conjecture that the condition \eqref{Ineq:Key_Cond} is both necessary and sufficient for \eqref{PAM} to be solvable in the case of $d=1$ and $H\leq \frac{1}{2}$. 

In the literature, one usually uses the powerful Hardy-Littlewood-Sobolev inequality 
to obtain a sufficient condition to ensure that  \eqref{e.square} is convergent. 
However, our results demonstrate that the application of this inequality alone 
cannot give a condition that is both necessary and sufficient. We need to combine 
both the Hardy-Littlewood-Sobolev and the  H\"older-Young-Brascamp-Lieb inequalities 
to arrive at our main results (Theorems \ref{main_result1} and \ref{thm_conver}). 
See Figure \ref{Fig.Area1+2} for an explanation.
We hope this methodology sheds light on the search of a condition that is both necessary and sufficient.

\begin{figure}[htbp]
	\centering
	\includegraphics[width=0.8\textwidth]{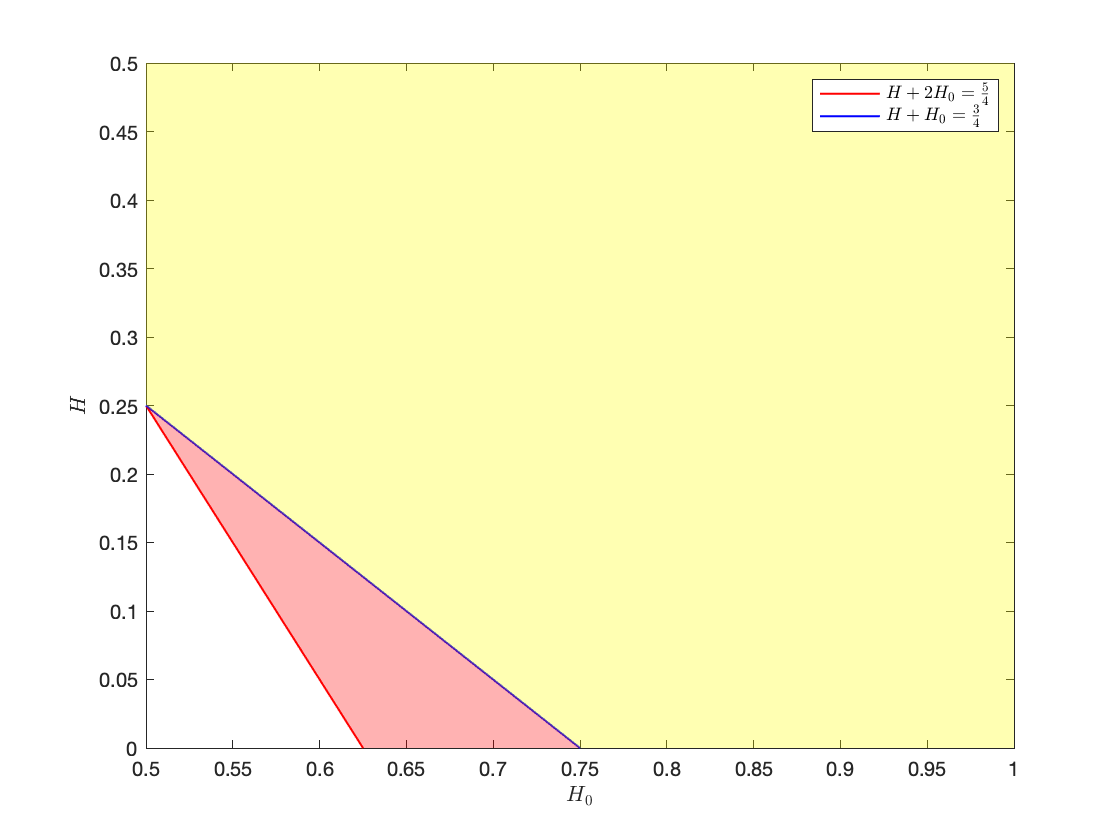}
	\caption{H\"older-Young-Brascamp-Lieb   region (red) and Hardy-Littlewood-Sobolev region  (yellow)}
	\label{Fig.BLArea}
\end{figure}

Let us point out that  
there is a necessary and sufficient condition (\cite[Theorem 1.1]{ams19_huliutindel}) for the solvability
of the stochastic heat equation
with additive noise: $\frac{\partial}{\partial t}u_\add (t,x) =\frac{1}{2}\Delta u_\add (t,x)+ \dot{W}(t,x)$, for $(t,x)\in\RR_+\times\RR^d$.
When the noise is fractional as in our case here, the condition
becomes (\cite[Equation (5.2)]{ams19_huliutindel}) 
	\[
	2H_0+\sum_{i=1}^d H_i>d\,,
	\]
which is exactly \eqref{cond.0<H<1_b}. 	 In terms of our notation we see
that $\EE [u_\add^2  (t,x)]=\EE[u_1^2(t,x)]$ for all $(t,x)\in\RR_+\times\RR^d$.  Hence, Theorem \ref{main_result1}
means that  $\EE [u_\add^2  (t,x)]=\EE[u_1^2(t,x)]<\infty$ 
does not automatically imply  $ \EE[u_n(t,x)^2]<\infty$ for all $n$ and $(t,x)\in\RR_+\times\RR^d$.

\subsection{Structure of the paper}

After some preparations in Section \ref{s.pre}, we aim to bound $\EE\left  [u_n(t,x)^2\right]$ which amounts to compute {some complicated multiple integrals} with both space  and time variables. In Section \ref{s.spatial}, we use the Fourier transforms to compute the spatial integral and then reduce the calculation of $\EE\left  [u_n(t,x)^2\right]$ to multiple integrals on the simplex on time variables. We deal with the latter one carefully to obtain the necessary and sufficient condition so that $\EE \left  [u_n(t,x)^2\right]<\infty$ in Section \ref{s.proof}. Section \ref{Sec5.Conv} focuses on the solvability of \eqref{PAM} in dimension one, {the well-accepted condition $H+H_0>\frac 34$ is improved to $H+2H_0>\frac 54$}, which is shown to be necessary and sufficient in some regions as we discussed earlier.

In order to make the paper more readable, we delay some of the detailed computations to the appendix,  where  we also recall  briefly the Hardy-Littlewood-Sobolev inequality and the H\"older-Young-Brascamp-Lieb inequality since they are the main tools in this work.

\section{Preliminaries}\label{s.pre}
In this section, we begin by introducing the notations and facts that will be used throughout this paper. Let   $\cC_0^{\infty}(\RR_+\times\RR^d)$  be the space of infinitely differentiable functions with compact support on $\RR_+\times\RR^d$.
%
Given that we will be dealing with the fractional Brownian noise whose Hurst
parameters can be greater than $1/2$ for some coordinates while being less than $1/2$
for others, and since we are only concerned with the
parabolic Anderson model,  it is more convenient for us to introduce the
scalar product by  using the Fourier transform with respect to the
spatial variables.

We denote the Fourier transform $\widehat f$ with respect to the $d$ spatial variables $x_1, \cdots, x_d$ as follows:
\[
\widehat  f(\xi) :=\cF[f](\xi)=\int_{\RR^d} e^{-\iota \xi x} f(x) dx,\ \ \hbox{ where } \iota=\sqrt{-1}.
\]
Let $H_0\in (1/2, 1)$ and $H_1, \cdots, H_d\in (0, 1)$.
Also, let us denote $\gamma_0(r )=\ga_{H_0}(r):=
H_0(2H_0-1)|r |^{2H_0-2}\,,  r\in \RR  $.
 We introduce a scalar product on  $\cC_0^{\infty}(\RR_+\times\RR^d)$    as defined below:
 \begin{equation}
\langle \varphi, \psi\rangle_\HH=  \int_{\RR_{+}^2\times \RR^{d}}\hat{\varphi}(r,\bm{\xi})\overline{\hat{\psi}}(s,\bm{\xi})\prod_{k=1}^d|\xi_k|^{1-2H_k}\gamma_0(r-s)drdsd\bm{\xi},
\label{eq.def2_H}
\end{equation}
where  $\bm{\xi}:=(\xi_{1}, \cdots, \xi_{d})$, $d\bm{\xi}:=d\xi_{1}\cdots d\xi_{d}$.  We denote $\HH$ as the Hilbert space obtained from the completion of
$\cC_0^{\infty}(\RR_+\times\RR^d)$ with respect to the scalar product
$\langle \cdot, \cdot\rangle_\HH$.   

The noise in the present paper is given by an isonormal Gaussian process $W=\{W(\varphi);\varphi\in \HH\}$ with
covariance
$$\EE[W(\varphi)W(\psi)]=\langle \varphi, \psi\rangle_\HH.$$
It is routine  to prove that {the function
$\1_{[0, t)\times \prod_{k=1}^d [0, x_k]}$} belongs to $\HH$
for any $t>0$ and $x=(x_1, \cdots,x_d)\in \RR^d$
(here $\1_{[b, a]}=-\1_{[a, b]}$ if $a<b$). {We denote}
$W(t,x):=W( \1_{[0, t)\times \prod_{k=1}^d [0, x_k]})$.   The Gaussian
random field  $\{W(t,x):t\geq 0, x\in\RR^d\}$ has    mean zero and    covariance   given by
\begin{align}\label{cov}
	\EE[W(t,x)W(s,y)]=&\ C_{H_0, H } \ga_{H_0}(t,s)\prod_{k=1}^d\ga_{H_k}(x_k,y_k)\,,\quad s,t\geq 0\,,x,y\in\RR^d\,,
\end{align}
where $C_{H_0, H }$ is a constant that depends on $H_0$ and  $H=(H_1, \cdots, H_d)$. It should be noticed that this constant $C_{H_0, H}$ may differ from those used in other 
literature, as we have set the constant in \eqref{eq.def2_H} to be $1$.

Since the Gaussian field $W$ is not a martingale in time (due to $H_0\not=1/2$), we cannot use the classical method of martingale measure to {define the stochastic integral}. Therefore, we shall use the chaos expansion to deal with our problem, and the most effective way to do this  is  to  introduce  the stochastic integral via {multiple chaos}   It\^o-Wiener integrals (see \cite{hubook,hurecent}).

Let $e_1, \cdots, e_n, \cdots\in   \cC_0^{\infty}(\RR_+\times\RR^d)$
be an orthogonal basis of $\HH$. Then
$\{\tilde e_n=W(e_n), n=1, 2, \cdots\}$ are independent
standard normal random variables.   We denote the symmetric tensor product by $\otimes $.   The tensor product $\HH^{\otimes n}$ is completion of the linear span of   $e_{\ell _1} 
\otimes  \cdots \otimes  e_{\ell_n} $ 
with respect to the scalar product
generated by 
\[
\langle e_{\ell _1} 
\otimes  \cdots \otimes  e_{\ell_n} \,, e_{j _1} 
\otimes  \cdots \otimes  e_{j_n} \rangle =\frac{1}{n!}
\sum_{\si \in \Sigma_n}
\langle e_{\ell _1}   \,, e_{j _{\si(1)} } 
 \rangle \cdots \langle e_{\ell _n}   \,, e_{j _{\si(n)} } 
 \rangle 
\]
where $\Sigma_n$ denotes the set of all permutations of   $\{1, \cdots, n\}$.  
 Let {$\tilde H_m(x)=(-1)^me^{\frac{x^2}{2} }\frac{d^m}{dx^m}(e^{-\frac{x^2}{2}})$}  be the
$m$-th Hermite polynomial.  Let $e_{i _1} 
,  \cdots , e_{i_k}$ be different and $n_1, \cdots, n_k$ are positive integers  such that $n_1+\cdots+n_k=n$.  
We define
the multiple integral as follows:
\[
 I_n(e_{i_1}^{\otimes
n_1}\otimes \cdots \otimes e_{i_k}^{\otimes
n_k} )={\tilde H_{n_1}(\tilde e_{i_1})\cdots \tilde H_{n_k}(\tilde e_{i_k})}\,.
\]
Any element in $   \HH^{\otimes n}$ can be approximated by
$f_n=\sum\limits_{0\leq n_1,\cdots,n_k\leq n} a_{i_1, \cdots, i_k}e_{i_1}^{\otimes
n_1}\otimes \cdots \otimes e_{i_k}^{\otimes
n_k} $, whose multiple It\^o-Wiener integral is defined as:
\begin{align*}
	I_n(f_n)=&\sum_{0\leq n_1,\cdots,n_k\leq n} a_{i_1, \cdots, i_k}
I_n(e_{i_1}^{\otimes
n_1}\otimes \cdots \otimes e_{i_k}^{\otimes
n_k}) \\
=&{\sum_{0\leq n_1,\cdots,n_k\leq n} a_{i_1, \cdots, i_k} \tilde H_{n_1}(\tilde e_{i_1})\cdots \tilde H_{n_k}(\tilde e_{i_k})} \,.
\end{align*}
%
%
%
%
%

{The Wiener chaos expansion theorem (e.g., \cite{hubook, Nualartbook})} states that  any random variable $F\in L^2(\Omega,\cF,\PP)$ admits a chaos expansion
\begin{equation}\label{WienerC}
{F=\EE[F]+\sum_{n=1}^{\infty}I_n(f_n),}
\end{equation}
where the series converges in $L^2(\Omega)$, and the elements $f_n\in \HH^{\otimes n}$ ($n\geq 1$) depend on $F$, and
\begin{equation}\label{eq_isome}
\EE[|F|^2]=(\EE [F])^2+\sum_{n=1}^{\infty}\EE[|I_n(f_n)|^2]
 =(\EE [F])^2+\sum_{n=1}^{\infty}n!\| {f}_n\|_{\HH^{\otimes n}}\,. 
\end{equation}
 
Let us recall  that 
$\HH$ is a space of (possibly generalized) functions with $d+1$ variables.
  $f_n\in \HH^{\otimes n}$ is  a 
  (possibly generalized) function  with $(d+1)\times n$ variables 
and   can be approximated by smooth functions with compact support from $(\RR_+ \times \RR^{ d})^n $ to $\RR$ with respect to the norm of $\HH^{\otimes n}$.   Additionally, the multiple integral $I_n(f_n)$ is identified  as follows:
\[
I_n(f_n)=\int_{\RR_+^n\times \RR^{nd}}
f_n(t_1, x_1, \cdots, t_n, x_n) W(dt_1, dx_1)\cdots  W(dt_n, dx_n)\,.
\]
The Fourier transform of $f$ with respect to the spatial variables is defined as:
\[
\hat f _n (t_1, \xi_1, \cdots, t_n, \xi_n)=
\int_{\RR^{nd} } f_n (t_1, x_1, \cdots, t_n, x_n)e^{-\iota \sum_{i=1}^n \xi_i x_i}
dx_1\cdots dx_n\,,
\]
where $\xi_i x_i=\sum_{k=1}^d \xi_{ik}x_{ik}$ is the Euclidean product of
$\xi_i=(\xi_{i1}, \cdots, \xi_{id})^T$,  $x_i =
(x_{i1}, \cdots, x_{id})^T$ and $dx_i=dx_{i1}\cdots d x_{id}$.
Using the above notation, we can express the $\HH^{\otimes n}$ norm of $f$ as:
\beq
\|f_n \|_{\HH^{\otimes n}}^2
&=&\int_{\RR_+^{2n}\times \RR^{nd}}
 \hat{f_n }(r_1,\xi_1, \cdots, r_n, \xi_n)
\bar {\hat{f_n}}(s_1,\xi_1, \cdots, s_n, \xi_n)\nonumber\\
&&\quad \prod_{i=1}^n \prod_{k=1}^d|\xi_{ik}|^{1-2H_k}\prod_{i=1}^n \gamma_0(r_i-s_i)dr_1\cdots dr_n ds_1\cdots ds_nd\bm{\xi}\,.\label{Eq:2.5}
\eeq
Thus, any square integrable nonlinear functional of $W$   can be written
as
\[
F=\sum_{n=0}^\infty I_n(f_n)=\sum_{n=0}^\infty\int_{\RR_+^n\times \RR^{nd}}
f_n(t_1, x_1, \cdots, t_n, x_n)  W(dt_1, dx_1)\cdots   W(dt_n, dx_n)\,,
\]
for some sequence $f_n\in \HH^{\otimes n}$. The expectation of $F^2$ can be expressed as:
\begin{align*}
	\EE [F^2]   &= {C_{H_0,H}^n}\sum_{n=0}^\infty n!  \int_{\RR_+^{2n}\times \RR^{nd}}
 \hat{f_n}(r_1,\xi_1, \cdots, r_n, \xi_n)
\bar {\hat{f_n }}(s_1,\xi_1, \cdots, s_n, \xi_n)\nonumber\\
&\qquad\qquad \times \prod_{i=1}^n \prod_{k=1}^d|\xi_{ik} |^{1-2H_k}\prod_{i=1}^n \gamma_0(r_i-s_i)dr_1\cdots dr_n ds_1\cdots ds_nd\bm{\xi}\,.
\end{align*}
 Consider a random field $f(t, x)$ defined on $\mathbb{R}_+ \times \mathbb{R}^d \times \Omega$, where $f$ is square integrable with $\mathbb{E}[f(t, x)^2] < \infty$.  The field $f(t, x)$ can be expressed as a chaos expansion:
\[
f(t,x)=f_0(t,x)+\sum_{n=1}^\infty I_n (f_n(t,x))\,,
\]
where {$f_n(t,x,\cdot)( t_1, x_1, \cdots, t_n, x_n)=f_n(t,x;t_1, x_1, \cdots, t_n, x_n)$} is an element in $\HH^{\otimes n} $ for any {fixed} $t$ and  $x$. The symmetrization of $f_n$, denoted as $\tilde{f}_n$, is defined as
\ba
\tilde{ f_n}(t_1, x_1 , \cdots, t_{n+1}, x_{n+1})
  = &{ \frac{1}{(n+1)! }\sum_{\si }  f_n(t_{\si(1)}, x_{\si(1)}, t_{\si(2)}, x_{\si(2)}, 
\cdots , t_{\si( i )}, x_{\si(i )}} \\
& \qquad\qquad\qquad  { t_{\si( i+1)},  x_{\si( i+1)}, \cdots,
t_{\si( n+1)}, x_{\si( n+1)})}  \nonumber
\end{align}
where the summation is taken over all permutation $\sigma$ 
on $\{ 1,\cdots,n+1 \}$. 

\begin{definition}\label{d.2.1}
{For such $f$ as above, $t\in\RR_+$ and $x\in\RR^d$ we say that $f$ is integrable} if, for every $n\ge 0$,
$\tilde{f_n} \in \HH^{\otimes (n+1)}$ and the series 
$\sum_{n=0}^\infty  I_{n+1}(\tilde {f_n})$  converges in $L^2(\Om)$. We define the stochastic integral of $f(t,x)$ as
\begin{equation}
\int_{\RR_+\times \RR^d} f(t,x)  W(dt,dx)=
\sum_{n=0}^\infty  I_{n+1}(\tilde {f_n})\,.
\end{equation}
\end{definition}

\bigskip

We now define the concept of a strong (random field) solution to equation \eqref{PAM}.\begin{definition}\label{def-sol-sigma}
A  real-valued adapted   stochastic process $u=\{u(t,x),t\geq 0, x \in \mathbb{R}^d
\}$ is said to be a  \emph{ (global) random field  solution (or mild solution)}   of \eqref{PAM} if {for all $t\in[0, \infty)$ and $x\in\mathbb{R}^d$},
$\{G_{t-\cdot}(x-\cdot) u(\cdot,\cdot)\}   $ is integrable and the following equality holds almost surely:
\begin{equation}\label{e.3.1}
u(t,x)= \int_{\RR^d} G_t(x,y) u_0(y)dy + \int_0^t \int_{\mathbb{R}^d}G_{t-s}(x-y) u(s,y)   W( ds, dy) ,
\end{equation}
where $G_t(x)=(2\pi t)^{-d/2}
\exp\left( -\frac{|x|^2}{2t}\right)$ is the heat kernel, and the stochastic integral is understood in the sense of Definition \ref{d.2.1}.  If equation \eqref{e.3.1} holds up to some time instant $t<t_0$ for a positive $t_0$, then we say that equation \eqref{PAM} has a local random field solution.
\smallskip

%
%
%
\end{definition}

If $u(t,x)$ satisfies equation \eqref{e.3.1},  then $u(s,y)$ has a similar representation.
By substituting this expression into \eqref{e.3.1}, we derive a new equation for $u$. Carrying this procedure  repeatedly,
 we observe that {if $u$ is a strong solution of \eqref{PAM}}, for each integer $N \geq 1$, we have
 \begin{equation}\label{e:3.2a}
u(t,x)=\sum_{n=0}^{N}  u_n(t,x) +\mathfrak{R}_{N}(t, x),
\end{equation}
 where
\begin{equation} \label{e:3.3}
u_n(t,x)= I_n(\tilde{f_n}  (t,x))
\end{equation}
and $\mathfrak{R}_{N}(t, x)= I_{N+1} (g_{N} (t,x))$. 
The element $\tilde{f_n}(t,x)$ is the symmetric extension with respect to $(s_1,x_1), \cdots, (s_n, x_n)$ of the product of Heat kernels
\begin{eqnarray}\label{e.3.3}
 & &\prod_{i=0}^n G_{s_{n+1}-s_{n}}(x_{n+1}-x_n)\nonumber \\
 &=&G_{t-s_n}(x-x_n)G_{s_n-s_{n-1}}(x_ n-x_{n-1})\cdots G_{s_2-s_1}(x_ 2-x_1)
 G_{s_1}u_0(x_1) \,,
\end{eqnarray}
where $0=s_0<s_1<s_2<\cdots<s_n<s_{n+1}=t$. 
More precisely, 
\begin{align}\label{e.3.3}
\tilde{f_n}(t,x) &:= \tilde{f_n}(t,x; s_1,x_1,\dots,s_n,x_n ) \\
&= \frac{1}{n!}\sum_{\sigma} \prod_{i=0}^n G_{s_{\sigma(i+1)}-s_{\sigma(i)}}(x_{\sigma(i+1)}-x_{\sigma(i)})\1_{\{0=s_0<s_{\sigma(1)}<\cdots<s_{\sigma(n)}<s_{n+1}=t\}} \,.\nonumber
\end{align}
{The summation above is taken over all permutation $\sigma$ 
on $\{ 1,\cdots,n \}$, which can be identified as a permutation  on $\{0,1,\cdots,n,n+1\}$ such that $\sigma(0)=0$ and $\sigma(n+1)=n+1$.} 
The function $g_{N} (t,x)$ is given by
\begin{eqnarray*}
g_{N} (t,x)= G_{t-s_{N+1}}(x-x_{N+1})G_{s_{N}-s_{N}}(x_ {N+1}-x_{N})\cdots G_{s_2-s_1}(x_ 2-x_1) u (s_1, x_1)\,.
\end{eqnarray*} 
It is evident that $\mathfrak{R}_{N}(t, x)$ is orthogonal to any multiple integral, with a deterministic kernel, of order less than or equal to $N$. {Therefore, if $u(t,x)$ belongs to $L^2(\Omega)$ for all $(t,x)\in\mathbb{R}_+\times \mathbb{R}^d$}, it must admit the chaos expansion {(e.g., \cite{hubook, Nualartbook})}:
\begin{equation}\label{e.2.series}
u(t,x) = \sum_{n=0}^\infty  u_n(t,x) = \sum_{n=0}^\infty  I_n(\tilde{f_n}(t,x)),
\end{equation}
where $\tilde{f_n}(t,x)$ represents the symmetric extension in \eqref{e.3.3}. Conversely, if the series \eqref{e.2.series} converges in $L^2$, it can be easily verified that $\{G_{t-\cdot}(x-\cdot) u(\cdot,\cdot)\}$ is integrable and \eqref{e.3.1} holds
true, establishing $u$ as a mild solution of \eqref{PAM}. Therefore, to study the solvability of \eqref{PAM} we need only to investigate the 
mean square convergence of the series \eqref{e.2.series}.

Throughout the paper, the notation $A \lesssim B$ (respectively $A \gtrsim B$ and $A \simeq B$) indicates the existence of strict positive universal constants $C_1$ and $C_2$, such that $A\le C_1B$ (respectively $A\ge C_2B$ and $C_2 B\leq A\leq C_1 B$). {We write constants $c_a$ and $C_b$ that depend only on $a$ and $b$ and write constants $c$ and $C$ for absolute constants. Their values may change from line to line.}

\section{Spatial Integration}\label{s.spatial}
To compute the $n$-th chaos $\EE [u_n(t,x)^2]$, we shall use the Fourier transform for the Hilbert scalar product \eqref{eq.def2_H}. In this section, we aim to find sharp bounds for the integral with respect to $d\bm{\xi}$ that appears in $\EE[u_n(t,x)^2]$. To simplify our notation, we introduce $\vec{\mathbf{s}}_j := (s_j, \cdots, s_n)$ for $1 \leq j \leq n$, and $\vec{\mathbf{s}} := \vec{\mathbf{s}}_1 = (s_1, \cdots, s_n)$. Conventionally, we will use
\begin{equation}\label{dsr}
 \begin{cases}
  d\vec{\bfs}_j:=ds_jds_{j+1}\cdots ds_n, \ \forall~ 1\le j\leq n,\\
  d\vec{\bfr}_j:=dr_jdr_{j+1}\cdots dr_n, \ \forall~ 1\le j\leq n,
 \end{cases}
\end{equation}
as well as $d\vec{\mathbf{s}} = d\vec{\mathbf{s}}_1$, and $d\vec{\mathbf{r}} = d\vec{\mathbf{r}}_1$. {Moreover, {for any permutation $\sigma$ on $\{j, j+1, \cdots, n\}$ for $1 \leq j \leq n$}, we denote
\begin{equation}\label{Tsr_sig}
	\{\vec{\bfs}_j\in\bT_t^{\sigma}\}:=\{(s_j,s_{j+1},\cdots,s_n):0=s_0<s_{\sigma(j)}<\cdots<s_{\sigma(n)}<s_{n+1}=t\}\,,
\end{equation}
and $\{\vbfs\in \bT_t^{\sigma}\}:=\{\vbfs_1\in \bT_t^{\sigma}\}$.
If $\sigma$ is the natural permutation, then $\bT_t^{\sigma}$ is abbreviated as $\bT_t$. This is,
\begin{equation}\label{Tsr}
\{\vbfs_j\in\bT_t\}:=\{(s_j,s_{j+1},\cdots,s_n):0=s_0<s_j<\cdots<s_n<s_{n+1}=t\}\,,
\end{equation}
and  $\{\vbfs \in\bT_t\}:=\{\vbfs_1 \in\bT_t\}$. }

{ 
It is easy to see that the Fourier transform of the symmetric function $f_n(t,x)$ in \eqref{e.3.3} with respect to  spatial variables $x_1,x_2,\cdots,x_n$  is given by
\begin{equation*}
	\widehat f_n(t,x; s_1, \xi_1, \cdots, s_n, \xi_n)=\frac{1}{n!}\sum_{\sigma}\widehat f_{n,\sigma}^{(t,x)}(s_1, \xi_1, \cdots, s_n, \xi_n)\1_{\{0<s_{\sigma(1)}<\cdots<s_{\sigma(n)}<t\}}
\end{equation*}
where
\begin{equation}
\widehat f_{n,\sigma}^{(t,x)}(s_1, \xi_1, \cdots, s_n, \xi_n)
=\prod_{i=1}^n  e^{-\frac{1}{2} (s_{\si(i+1)}-s_{\si(i)})
|\xi_{\si(i)}+\cdots+\xi_{\si(1)}|^2} e^{-  \iota  x(\xi_n+\cdots+\xi_1)}\,.
\label{e.3.4a}
\end{equation}
{The readers can observe that symmetrization enables us to apply classical results such as \eqref{Eq:2.5}. However, the calculations and estimations related to time variables become significantly intricate.}

We assume $u_0\equiv1$ on $\mathbb{R}^d$ throughout the remaining part of this paper. 
For $n\geq 1$,  {$\vbfs\in\bT_t^{\varsigma}$ and $\vbfr\in \bT_t^{\sigma}$},  {where $\varsigma$ and $\si$ are two
permutations of $\{1, 2, \cdots, n\}$,} we can use equation \eqref{eq.def2_H}, the definition of the tensor product norm and \eqref{e.3.4a} to obtain that{
\begin{align}
 \EE &[ u_n(t,x)^2 ] \nonumber \\
 =& \frac{1}{n!} \sum_{\varsigma,\si} \int_{{0<s_1,\cdots, s_n<t\atop 0< r_1, \cdots, r_n<t}}\int_{\RR^{nd}}\prod_{i=1}^n e^{- \frac12 (s_{\varsigma(i+1)}-s_{\varsigma(i)} )|\sum\limits_{j=1}^i \xi_{\varsigma(j)}|^2 - \frac12 (r_{\si(i+1)}-r_{\si(i)})|\sum\limits_{j=1}^i \xi_{\sigma(j)}|^2  }\nonumber\\
 &\qquad\qquad\qquad\qquad\qquad\quad
 \prod_{i=1}^n\prod_{k=1}^d
|\xi_{ik}|^{1-2H_k}d\xi_i
 \times \prod_{i=1}^n \gamma_0(s_i-r_i) d\vec{\bfs} d\vec{\bfr}\nonumber\\
 =&\sum_{\sigma} \int_{{0<s_1<\cdots<s_n<t\atop 0< r_1, \cdots, r_n<t}
}\int_{\RR^{nd}}
\prod_{i=1}^n e^{- \frac12 (s_{i+1}-s_i )|\sum\limits_{j=1}^i \xi_{j}|^2 - \frac12 (r_{\si(i+1)}-r_{\si(i)})
 |\sum\limits_{j=1}^i \xi_{\sigma(j)}|^2  }\nonumber\\
&\qquad\qquad\qquad\qquad\qquad\quad
 \prod_{i=1}^n\prod_{k=1}^d
|\xi_{ik}|^{1-2H_k}d\xi_i
 \times \prod_{i=1}^n \gamma_0(s_i-r_i) d\vec{\bfs} d\vec{\bfr} \,, \label{Eq:3.5}
\end{align}
where the second equality follows by restricting $\{0<s_1,\cdots, s_n<t\}$ to $\{0<s_1<\cdots<s_n<t\}$ and by the symmetry. When {$\vbfs\in\bT_t$} we define $h_{k,n}(\vbfs)$ as:
  \begin{align}\label{hks}
h_{k,n}  (\vbfs ) :=& \int_{\RR^{n}}
\prod_{i=1}^n e^{-\frac 12(s_{i+1 }-s_i )|\xi_{ik}+\cdots+\xi_{1k} |^2}   \prod_{i=1}^n
|\xi_{ik}|^{1-2H_k}d\xi_{1k}\cdots d\xi_{nk}  \nonumber\\
=& \int_{\RR^{n}}
\prod_{i=1}^n e^{-\frac 12(s_{i+1}-s_i )|\eta_i |^2}   \prod_{i=1}^n
|\eta_{i}-\eta_{i-1} |^{1-2H_k}d\eta_i\,,
\end{align}
where the substitution $\eta_i:=\xi_{ik}+\cdots+\xi_{1k}$ is used in  \eqref{hks} and $\eta_0:=0$ by convention. Moreover, {for any permutation $\si$ of $\{1, 2, \cdots, n\}$}, we define similarly with the notation $\eta_i^{\sigma}=\xi_{\si(i)k}+\cdots+\xi_{\si(1)k}$:
\begin{align}
	{h_{k,n}  (\vec{\bfs},\vec{\bfr}^{\sigma})} =& \int_{\RR^{n}}
\prod_{i=1}^n e^{-\frac12(s_{i+1}-s_i )|\xi_{ik}+\cdots+\xi_{1k}|^2 -\frac12(r_{\si(i+1)}-r_{\si(i)})
 |\xi_{\si(i)k}+\cdots+\xi_{\si(1)k} |^2  }\nonumber\\
&\qquad\qquad\qquad\qquad\qquad\qquad\qquad\qquad \times
 \prod_{i=1}^n
|\xi_{ik}|^{1-2H_k} d\xi_{1k}\cdots d\xi_{nk} \label{hk2} \nonumber \\
=& \int_{\RR^{n}}
\prod_{i=1}^n e^{-\frac12(s_{i+1}-s_i )|\eta_i|^2 -\frac12(r_{\si(i+1)}-r_{\si(i)})
 |\eta_i^{\sigma}|^2  } \prod_{i=1}^n
|\eta_{i}-\eta_{i-1} |^{1-2H_k}d\eta_i\,. 
\end{align}
Then by applying H\"older's inequality, we obtain
\begin{align}
{h_{k,n}  (\vec{\bfs},\vec{\bfr}^{\sigma})} \le & C_H^n h_{k,n}^{1/2}(\vec{\bfs})    h_{k,n}^{1/2} (\vec{\bfr})\,,\quad \vec{\bfs},\vec{\bfr} \in \bT_t\,.\label{hk}
\end{align}
Furthermore, we can simplify \eqref{Eq:3.5} as
\begin{equation}
	\EE \left[ u_n(t,x)^2 \right] = \sum_{\sigma} \int_{{0<s_1<\cdots<s_n<t\atop 0< r_1, \cdots, r_n<t}}\prod_{k=1}^d {h_{k,n}  (\vec{\bfs},\vec{\bfr}^{\sigma})}\prod_{i=1}^n \gamma_0(s_i-r_i) d\vec{\bfs} d\vec{\bfr}\,.\label{e.4.1}
\end{equation}
}

In order to obtain a sharp bound for $h_{k,n}(\vbfs)$, we can rewrite equation \eqref{hks} as an expectation of normal variables. Let $X_0=0$ and let $\{X_1, \cdots, X_n\}$ be i.i.d. standard Gaussian random variables. Denote
  \begin{equation}\label{def_wi}
{ w_0=1 }\quad \hbox{ and } \quad w_i:=s_{i+1}-s_i  \quad  \hbox{for $1\leq i\leq  n$} .
  \end{equation}
Then for $0<s_1<\cdots < s_{n}<t$, {  for some positive constants $c_{H_k}$  depending on Hurst parameters $H_k$ we can reformulate $h_{k,n} (\vec{\bfs} )$ in \eqref{hks} as }
\begin{align}
h_{k,n} (\vec{\bfs} )
&=  c_{H_k}^n   \left( \prod_{i=1}^n w_i^{-1/2} \right) \EE\left[\prod_{i=1}^n
\left|\frac{X_i}{\sqrt{w_i}}-\frac{X_{i-1}}{\sqrt{w_{i-1}}}\right|^{1-2H_k}\right]\nonumber\\
& = c_{ H_k}^n    \left( \prod_{i=1}^n w_i^{-1/2} \right)
\left(   \prod_{i=1}^n    (w_iw_{i-1})^{H_k-1/2} \right)\nonumber\\
&\qquad\qquad\qquad\qquad\quad\times \EE\left[ \left|  X_1
 \right| ^{1-2H_k}  \prod_{i=2}^n
\left|  \sqrt{w_{i-1}}  X_i - \sqrt{w_i}  X_{i-1}  \right|^{1-2H_k} \right]
\nonumber\\
  &= c_{ H_k}^n   w_n^{H_k-1}
\left( \prod_{i=1}^{n-1} w_i^{2H_k-3/2} \right)
\left( \prod_{  i=2 }^{n }  (w_i+w_{i-1})^{\frac12- H_k} \right)  \nonumber\\
&\qquad  \times   \EE\left[ \left|  X_1  \right| ^{1-2H_k} \prod_{i=2}^n
\left| \sqrt{\frac{w_{i-1}}{w_{i-1}+w_i} }X_i - \sqrt{\frac{w_{i }}{w_{i-1}+w_i} }X_{i-1} \right|^{1-2H_k}
\right] \,. \label{e.4.3}
\end{align}
{Let us introduce the following notations to simplify the expectation in \eqref{e.4.3}. We set
\begin{equation}\label{lam}
\la_1=1  \hbox{ and } \lambda_i:=\sqrt{\frac{w_{i-1}}{w_{i-1}+w_i}} \in(0,1)\,, \quad i\geq 2\,,
\end{equation}
and
\begin{equation}\label{theta}
\theta_k :=2H_k-1\,, \quad \zeta_n:=\prod_{i=1}^{n}
\big|  {\lambda_i }X_i - \sqrt{1-\lambda_i^2 } \, X_{i-1} \big|^{- \theta_k }\,.
\end{equation}
Let us denote the expectation in \eqref{e.4.3} as $\mathfrak{I}_{k,n}(\lambda_1,\cdots,\lambda_n)$. In particular, we amend the expectation in \eqref{e.4.3} to
\begin{equation}\label{I_k}
\mathfrak{I}_{k,n}(\lambda_1,\cdots,\lambda_n):=\EE\left[ \zeta_n \right]
= \EE\left[ \prod_{i=1}^{n} \left| 
 {\lambda_i }X_i - \sqrt{1-\lambda_i^2 }X_{i-1} \right|^{1-2H_k}\right],
\end{equation}
with the conventions $\prod_{i= m}^n a_i :=1$ if $m>n$, and the conventions $\lambda_1=1$, $X_0=1$. 
Recall that our objective is to obtain a sharp bound for $h_{k,n}(\vbfs)$. To achieve this, it is adequate to establish a precise upper limit for $\mathfrak{I}_{k,n}(\lambda_1,\cdots,\lambda_n)$, as presented below.}

\begin{lemma}\label{lem_J_kn}
We have the following upper bounds for $\mathfrak{I}_{k,n}(\lambda_1,\cdots,\lambda_n)$.
\begin{enumerate}
\item[(1)] If $H_k<3/4$,  then there is a positive constant $C_{H_k}$ such that {for all $\la_1,\cdots,\la_n$}
\begin{equation}
\mathfrak{I}_{k,n}(\lambda_1,\cdots,\lambda_n) \le C_{H_k}^n\,.
\end{equation}
\item[(2)] If $H_k\ge 3/4$,  then for any $\be_k\in (4H_k-3,2H_k-1)$,  there is a positive constant $C_{\be_k}$ such that {for all $\la_1,\cdots,\la_n$}
\begin{equation}\label{e.ikn3}
\mathfrak{I}_{k,n}(\lambda_1,\cdots,\lambda_n) \le {C_{H_k, \be_k}^n} \prod_{i=2} ^n
\la_i^{-\beta_k} \,.
\end{equation}
\end{enumerate}
\end{lemma}

\begin{proof} We divide the proof into three steps based on the value of $H_k$: $(0, 1/2)$, $[1/2, 3/4)$, and $[3/4, 1)$.

\noindent \emph{{\bf Step 1:}  The case $0<H_k<1/2$, i.e., $\theta_k< 0$}.   This case has been proven in \cite{CH2021}. More precisely, we have
\begin{eqnarray}
\mathfrak{I}_{k,n}(\lambda_1,\cdots,\lambda_n)
 &\le & C_{H_k}\EE\left[ \left|  X_1
 \right| ^{-\theta_k} \left( \prod_{i=2}^{n}  \left(
 |  X_i | \vee |X_{i-1}| \right)^{-\theta_k} \right)
 \right] \nonumber\\
 &\le & C_{H_k}\EE\left[ \left|  X_1
 \right| ^{-\theta_k} \left( \prod_{i=2}^{n}  \left(
 |  X_i | + |X_{i-1}| \right)^{-\theta_k} \right)
 \right] \nonumber\\
 \nonumber \\
&\le&  C^n_{H_k} \,.    \label{e:4.4}
\end{eqnarray}

\smallskip
\noindent \emph{{\bf Step 2:}  The case $3/4 \leq H_k< 1$.}
 To bound $\mathfrak{I}_{k,n}(\lambda_1,\cdots,\lambda_n)$ in this case, we shall use the following estimation for {a standard Gaussian random variable} $X$ \ (see Lemma A.1 in \cite{HNS} for details): for any $0<\alpha <1$, $ \lambda >0$ and $b\in \RR$, there is a constant $C_{\alpha}>0$ independent of $\lambda$ and $b$ so that
 \begin{equation}\label{e:4.5}
 \EE \left[ | \lambda X +b |^{-\alpha} \right] \leq C_{\alpha}  (\lambda \vee |b| )^{-\alpha}\simeq (\lambda + |b |)^{-\alpha} .
 \end{equation}

{Denote by} $\EE^Y$ the expectation with respect to the random variable $Y$ while considering other random variables as constants.
Thus, it is clear that for $\theta_k\geq \frac 12$ ($\theta_k$ is defined in \eqref{theta})  and $\be_k\in (4H_k-3,2H_k-1)$
\begin{align} \label{e.ikn1}
\fI_{k,n}(\lambda_1,\cdots,\lambda_n)&=\EE\left[ \EE^{X_n}  \Big[  \big|  {\lambda_n }X_n - \sqrt{1-\lambda_n^2 } \, X_{n-1} \big|^{-\theta_k  }  \zeta_{n-1} \Big] \right] \nonumber \\
& \leq   C_{H_k}  \EE\lt[     \left(    \sqrt{1-\lambda_n^2 } \, |X_{n-1}|   + \lambda_n \right)^{-\theta_k } \zeta_{n-1}
  \right]\nonumber\\
  & \leq   C_{H_k}\la_n^{-\be_k}   \EE\bigg\{    \EE^{X_{n-1}}\bigg[ \left(   \sqrt{1-\lambda_n^2 } \, |X_{n-1}|   + \lambda_n \right)^{-\theta_k+\be_k } \nonumber\\
  & \qquad \qquad\qquad \qquad \big|  {\lambda_{n-1} }X_{n-1} -\sqrt{1-\lambda_{n-1}^2 } \, X_{n-2} \big|^{-\theta_k}\bigg]  \zeta_{n-2}
  \bigg\}\nonumber\\
& \leq   C_{H_k}\la_n^{-\be_k}   \EE\Bigg\{    \lt[ \EE^{X_{n-1}} \left(   \sqrt{1-\lambda_n^2 } \, |X_{n-1}|   + \lambda_n \right)^{-(\theta_k-\be_k)p }  \rt]^{1/p}\nonumber\\
&\qquad \qquad  \left[  \EE ^{X_{n-1}}\big|  {\lambda_{n-1} }X_{n-1} -\sqrt{1-\lambda_{n-1}^2 } \, X_{n-2} \big|^{-\theta_k q  } \right]^{1/q}   \zeta_{n-2}
  \Bigg\},
\end{align}
by applying H\"older inequality with $\frac1p+\frac1q=1$ in the last inequality.

 It is easy to see that under our assumption $4H_k-3<\be_k<2H_k-1$, we have $2\theta_k-1<\be_k<\th_k$.
 This enables us to  find  $p$ and $q$ such
that $0<(\theta_k-\be_k)p<1$ and  $0<\theta_k q<1$ hold. Therefore,
\begin{align}\label{bound1}
\EE^{X_{n-1}} \bigg[ \Big(   \sqrt{1-\lambda_n^2 }& \, |X_{n-1}|   + \lambda_n \Big)^{-(\theta_k-\be_k)p }\bigg] \nonumber \\
&\le { C_{\be_k,\theta_k}}(\sqrt{1-\lambda_n^2 }+ \lambda_n )^{-(\theta_k-\be_k)p }\le {C_{\be_k,\theta_k}} <\infty\,,
\end{align}
and
\begin{align}\label{bound2}
\EE ^{X_{n-1}} \bigg[\big|  {\lambda_{n-1} }X_{n-1} &-\sqrt{1-\lambda_{n-1}^2 } \, X_{n-2} \big|^{-\theta_k q  }\bigg]\nonumber \\
&\le {C_{\theta_k}} (\lambda_{n-1}
+\sqrt{1-\lambda_{n-1}^2 } \, |X_{n-2}|)^{-\theta_k q}\,.
\end{align}
Substituting \eqref{bound1} and \eqref{bound2} into \eqref{e.ikn1} we have
\begin{align}
\fI_{k,n}(\lambda_1,\cdots,\lambda_n) \le
&{ C_{H_k,\be_k}} \la_n^{-\be_k}   \EE\bigg[   (\lambda_{n-1}
+\sqrt{1-\lambda_{n-1}^2 } \, |X_{n-2}|)^{-\theta_k  }    {\zeta_{n-2}}
  \bigg]\,. \label{e.ikn2}
\end{align}
Continuing this way we obtain that for any $2\theta_k-1<\be_k<\theta_k$, \eqref{e.ikn3} holds.


\noindent \emph{{\bf Step 3:} The case $\frac12\le H_k<\frac34$.}  When $H_k=1/2$, it is easy to see that $\fI_{k,n}(\lambda_1,\cdots,\lambda_n)\equiv1$.  So we can assume
 $\frac12< H_k<\frac34 $.  The proof is similar to the case $H_k>\frac34$ except that  we take  $\be_k=0$ now and the proof is simpler. In fact, the H\"older inequality is valid   since we can still find  $p$ and $q$ such that $0<\theta_k p<1$ and  $0<\theta_k q<1$ hold. In conclusion, we have
\begin{align}\label{bound3}
\fI_{k,n}(\lambda_1,\cdots,\lambda_n) \le \     C_{H_k}^n\,.
\end{align}
This proves the lemma.
\end{proof}

To deal with the necessary condition in our main Theorem \ref{main_result1} in the next section, we also need a lower bound for $\mathfrak{I}_{k,n}(\lambda_1,\cdots,\lambda_n)$. We can obtain the lower bound when $H_k< 3/4$.  However, we are still not clear about the lower bound of $\mathfrak{I}_{k,n}(\lambda_1,\cdots,\lambda_n)$  when $H_k\ge 3/4$ except when $n=2$.
But this result is sufficient for our purpose.
\begin{lemma} \label{lem_J_kn_low}
We have the following lower bounds for $\mathfrak{I}_{k,n}(\lambda_1,\cdots,\lambda_n)$ defined by \eqref{I_k}.
\begin{enumerate}
\item[(1)] If $H_k<3/4$,  then there is a positive constant $c_{H_k}$,  such that {for all $\la_1,\cdots,\la_n$}
\begin{equation}
\mathfrak{I}_{k,n}(\lambda_1,\cdots,\lambda_n) \ge c_{ H_k}^n \,.\label{e.3.20lower}
\end{equation}
\item[(2)] If $H_k\ge 3/4$,  then there is a positive constant $c_{H_k}$, independent of $\la_2$    such that 
\begin{equation}
\mathfrak{I}_{k,2}(\lambda_1,\lambda_2) \ge c_{  H_k}
\la_2^{-(4H_k-3)} \,,\label{e.3.21lower}
\end{equation}
where $\la_2$ is defined by \eqref{lam}.
\end{enumerate}
\end{lemma}

\begin{proof}
First, we prove statement (1).
{Set  
\[
A_n:=\begin{cases}
\{X_1\leq 0,X_2\geq 0,\cdots,X_{n-1}\leq 0,   X_{ n}\geq 0\}&\qquad \hbox{when $n$ is even}\,;\\
\{X_1\leq 0,X_2\geq 0,\cdots,X_{n-1}\geq 0,   X_{ n}\leq 0\}&\qquad \hbox{when $n$ is odd}\,. \\
\end{cases}
\]
When $0<H_k\le \frac12$, {namely, $-1\leq\theta_k<0$ for $\theta_k$ defined by  \eqref{theta},  } we have
\begin{align}
\mathfrak{I}_{k,n}(\lambda_1,\cdots,\lambda_n)
:&= \EE\left[{ \lt|X_1\rt|^{-\theta_k}} \prod_{i=2}^{n}
\left|{\lambda_i }X_i-\sqrt{1-\lambda_i^2 }X_{i-1} \right|^{-\theta_k}\right]\nonumber\\
&\geq \EE\left[ { \lt|X_1\rt|^{-\theta_k}} \prod_{i=2}^{n}
\left|{\lambda_i }X_i - \sqrt{1-\lambda_i^2 }X_{i-1} \right|^{-\theta_k}\cdot\1_{A_n}\right]\nonumber \\
&\geq c_{H_k}^n\ \EE\left[ { \lt|X_1\rt|^{-\theta_k}} \left( \prod_{i=2}^{n}  \left(|X_i| \wedge |X_{i-1}| \right)^{-\theta_k} \right)\cdot\1_{A_n}\rt]\nonumber\\
&\geq  c_{H_k}^n\
\EE  \left[ { \lt(|X_1| \wedge |X_2|\wedge\cdots\wedge|X_n| \rt)^{-n\theta_k}}\cdot\1_{A_n}\rt]\,, \nonumber
\end{align}
 where we used the fact that $\la_i+\sqrt{1-\la_i^2}\geq 1$ in the above second inequality.} 
{Now we let 
\[
B_n:=\begin{cases}
\{X_1\leq -1,X_2\geq 1,\cdots,X_{n-1}\leq -1,   X_{ n}\geq 1\}&\qquad \hbox{when $n$ is even}\,;\\
\{X_1\leq -1,X_2\geq 1,\cdots,X_{n-1}\geq 1,   X_{ n}\leq -1\}&\qquad \hbox{when $n$ is odd}\,, \\
\end{cases}
\] 
which is contained in $A_n$.  
We proceed to get}
\begin{align*}
\mathfrak{I}_{k,n}(\lambda_1,\cdots,\lambda_n)~ &{\geq  c_{H_k}^n\
\EE  \left[ { \lt(|X_1| \wedge |X_2|\wedge\cdots\wedge|X_n| \rt)^{-n\theta_k}}\cdot   \1_{B_n} \rt]}\nonumber\\
&{\geq  c_{H_k}^n 
\left[ \PP( X_1\le -1 )\right]^n =c_{H_k}^n>0\,,}\nonumber
\end{align*}
where we  recall that $c_{H_k}$ is a generic positive constant which may be different in different places. 

When $\frac12< H_k\leq\frac34$, {i.e., $0< \theta_k \leq 1/2$,  recalling } $0\le \la_i\le 1$, we have
\begin{align*}
\mathfrak{I}_{k,n}(\lambda_1,\cdots,\lambda_n)
:&= \EE\left[\lt|X_1\rt| ^{-\theta_k}\prod_{i=2}^{n}
\left|{\lambda_i }X_i-\sqrt{1-\lambda_i^2 }X_{i-1} \right|^{-\theta_k}\right]\\
&\geq  \EE\left[ |X_1|^{-\theta_k} \left( \prod_{i=2}^{n}  \left(|X_i| +|X_{i-1}| \right) \right)^{-\theta_k}  \right]\,. 
\end{align*}
Notice that the factors $|X_i|^2$ may appear in the product $\prod_{i=2}^{n}  (|X_i| +|X_{i-1}|)$. But the expectation in the second line is bounded thanks to the assumption $0>-\theta_k\geq -1/2$. Thus, we have 
$$\mathfrak{I}_{k,n}(\lambda_1,\cdots,\lambda_n)\geq c_{H_k}^n>0\,.$$


Now, we prove statement (2).
In this case \emph{  $H_k>\frac34$}, {namely $\theta_k>1/2$}. We only need to consider the lower bound of { $\fI_{k,2}(\lambda_1,\lambda_2)$ defined in \eqref{I_k} with $n=2$ there}. Remember that $\la_1=1$, $\fI_{k,2}(\lambda_1,\lambda_2)$ is a function of $\la_2$.
{ It is clear that $\fI_{k,2}(1,\cdot)$ is continuous on $(0, 1]$}.   If $\la_2\to 1$,  then
$\fI_{k,2}(1,\la_2) \to \EE\lt[|X_1|^{-\theta_k}\lt|  X_2 \rt|^{-\theta_k}\rt]>0$.  {We want to understand the behaviour of this function as $\lambda_2$ approaches 0.} For simplicity, we assume   $0<\la_2 <\frac12$.     Then
{\begin{align}
\fI_{k,2}(\la_1,\la_2)&= \EE\lt[|X_1|^{-\theta_k}\lt|\la_2 X_2-X_1\rt|^{-\theta_k}\rt]\nonumber\\
&= \frac{1}{\sqrt{2\pi}} \EE^{X_2}\lt[ \int_{-\infty}^\infty |x_1|^{-\th_k}|\la_2 X_2-x_1|^{-\th_k}e^{-\frac{x_1^2}{2}} dx_1 \rt]\nonumber\\
&\geq c \EE^{X_2}\lt[ \int_0^1|x_1|^{-\th_k}|\la_2 X_2-x_1|^{-\th_k}dx_1 \rt]\nonumber\\
&=\EE^{X_2}\lt[|\la_2 X_2|^{-2\th_k+1}\cdot\int_0^{1/(\la_2 |X_2|)}|\tilde{x}_1|^{-\th_k}|1-\tilde{x}_1|^{-\th_k}d\tilde{x}_1\rt] \,, \nonumber
\end{align}
where in the last equality we make a change of variables $x_1\to \la_2 |X_2|\tilde{x}_1$. Then, by the conditions $0<\la_2 <\frac12$ and  $1/2<\theta_k<1$ we get
\begin{align}
	\fI_{k,2}(\la_1,\la_2)&\geq 
	\EE^{X_2}\lt[|\la_2 X_2|^{-2\th_k+1}\cdot\int_0^{2/  |X_2| }|\tilde{x}_1|^{-\th_k}|1-\tilde{x}_1|^{-\th_k}d\tilde{x}_1 \1_{\{|X_2|\leq 1\}} \rt] \nonumber\\
&\geq\EE^{X_2}\lt[ |\la_2 X_2|^{-2\th_k+1}\1_{\{|X_2|\leq 1\}} \cdot\int_0^{2}|\tilde{x}_1|^{-\th_k}|1-\tilde{x}_1|^{-\th_k}d\tilde{x}_1 \rt] \nonumber\\
&\geq  c_{\theta_k} \EE^{X_2}\lt[ |\la_2 X_2|^{-2\th_k+1}\1_{\{|X_2|\leq 1\}} \rt]\,. \label{J_21}
\end{align}
}
Moreover, it is not hard to see
\begin{align}\label{J_22}
\EE^{X_2}\lt[|\la_2X_2|^{-2\th_k+1} \1_{\{|X_2|\leq 1\}}\rt]&\geq c_{\theta_k}\la_2 ^{-2\th_k+1}\cdot\int_0^1x_2^{-2\th_k+1}dx_2\nonumber\\
&\geq c_{\theta_k} \la_2^{-2\th_k+1}.
\end{align}
{ Substituting \eqref{J_22} into \eqref{J_21} implies
\begin{align}\label{Est_J_2}
\fI_{k,2}(\la_1,\la_2)\geq c \la_2^{-2\th_k+1}= c\la_2^{3-4H_k}.
\end{align}}
This completes the proof.
\end{proof}

Before stating Lemma \ref{est_gk}, we introduce a set of indices $\cD_n$, which consists of all indices $\bm{\al}=(\al_1,\cdots,\al_n)$  such that
\begin{eqnarray}
&&\al_i
 \in  \lt\{0,\frac12d- |H|+\frac12\be^*, d- 2|H|+\be^* \rt\} \,,\quad \al_i+\al_{i+1}\not=0\,,  \label{Def:al_i}
\label{al_i}  \\
\hbox{and}&&|\bm{\al}|
 =\sum_{i=1}^n\al_i=(n-1)\lt(\frac12d- |H|+\frac12\be^*\rt)\,.
 \nonumber
\end{eqnarray}

\begin{lemma}\label{est_gk}
Let $h_n(\vbfs ):=\prod\limits_{k=1}^d h_{k,n} (\vbfs )$, where  $h_{k,n} (\vbfs ) \ (k=1,\cdots,d)$ are defined by \eqref{hks}.   Take any  \ $ \be_k\in (4H_k-3, 2H_k-1)$ for $k=d_*+1,\cdots,d$  and set
\begin{equation}\label{beta}
\be^*:=\be_{d_*+1}+\cdots+\be_d\,.
\end{equation}
Then, we have the following estimations for
$h_n(\vbfs )$.
\begin{enumerate}
\item[(i)] If $\frac12d- |H|+\frac12\be^*>0,$  then
\begin{align}\label{est_g}
h_n (\vec{\bfs} )
\leq C_{H }^nw_n^{|H|-d+\al_n}\sum_{\bm{\al}\in \cD_n}\lt(\prod_{i=1}^{n-1} w_i^{\al_i+2|H|-\frac32d-\frac12\be^*}\rt)\,,
\end{align}
where $w_i$ $(1\leq i\leq n)$ are given by \eqref{def_wi}.
Moreover, if $\al_i=d- 2|H|+\be^*$, then both of  $\al_{i-1}$ and $\al_{i+1}$ cannot be   $\ d- 2|H|+\be^*$.

\item[(ii)] If $\frac12d- |H|+\frac12\be^*\leq 0,$   then
\begin{equation}\label{est_h}
h_n (\vec{\bfs} )\leq C_{H }^n\prod_{i=1}^nw_i^{|H|-d}.
\end{equation}
\end{enumerate}
\end{lemma}

\begin{proof}
By Lemma \ref{lem_J_kn}, when $0<H_k<\frac34$,
\[\fI_{k,n}(\la_1,\cdots,\la_n)\leq C^n_{H_k}\,.
\]
Then, from \eqref{e.4.3} we obtain when $k\leq d_*$,
\begin{align}\label{hk_k<d}
h_{k,n} (\vec{\bfs} )\leq  { C_{H_k}^n} w_n^{H_k-1}
\left( \prod_{i=1}^{n-1} w_i^{2H_k-3/2} \right)
 \end{align}
When $\frac34\le H_k<1$, { substituting \eqref{e.ikn3} into \eqref{e.4.3}} we get for $d_*<k\leq d$,
\begin{eqnarray}\label{h_k,k>d}
h_{k,n}(\vec{\bfs} )\leq { C_{H_k,\beta_k}^n} w_n^{H_k-1}
\left( \prod_{i=1}^{n-1} w_i^{2H_k-3/2} \right)
\left( \prod_{  i=2 }^{n }  (w_i+w_{i-1})^{\frac12- H_k} \right) \times
\left(\prod_{i=2}^n\lambda_i^{-\be_k}\right).\nonumber\\
\end{eqnarray}
{Recall that $\lambda_i=\sqrt{\frac{w_{i-1}}{w_{i-1}+w_i}}$ $(i\geq 2)$ are defined} in \eqref{lam}.  Combining \eqref{hk_k<d} and \eqref{h_k,k>d},   we have
\begin{align}
h_n (\vec{\bfs} ):= & \prod_{k=1}^d {  h_{k,n}(\vec{\bfs} )  }
\nonumber    \\
\leq & C_{H }^n \prod_{k=1}^{d}\left[w_n^{H_k-1}
\left( \prod_{i=1}^{n-1} w_i^{2H_k-3/2} \right)
\left( \prod_{  i=2 }^{n }  (w_i+w_{i-1})^{\frac12- H_k} \right)\right]
 \prod_{k=d_*+1}^{d}\prod_{i=2}^n\la_i^{-\be_k}\nonumber \\
{=} & C_{H }^n w_n^{|H|-d}\left( \prod_{i=1}^{n-1} w_i^{2|H|-\frac32d-\frac12\be^*} \right)\left( \prod_{  i=2 }^{n }  (w_i+w_{i-1})^{\frac12d- |H|+\frac12\be^*} \right).\label{re.4.6}
\end{align}
If $\frac12d- |H|+\frac12\be^*>0$, {then it follows from \eqref{re.4.6} by expanding the  product $\prod_{  i=2 }^{n }  (w_i+w_{i-1})^{\frac12d- |H|+\frac12\be^*}$ that}
\begin{align*}
h_n (\vbfs )\leq C_{H }^nw_n^{|H|-d+\al_n}\sum_{\bm{\al}\in \cD_n}\lt(\prod_{i=1}^{n-1} w_i^{\al_i+2|H|-\frac32d-\frac12\be^*}\rt)\,,
\end{align*}
where $\bm{\al}=(\al_1,\cdots,\al_n)$ is in $\cD_n$. 
If $\frac12d- |H|+\frac12\be^*\leq 0$, {then we can see from \eqref{re.4.6} that}
\begin{align*}
h_n (\vec{\bfs} )&\leq C_{H }^n w_n^{|H|-d}\left( \prod_{i=1}^{n-1} w_i^{2|H|-\frac32d-\frac12\be^*} \right)\left( \prod_{  i=2 }^{n }  w_{i-1}^{\frac12d- |H|+\frac12\be^*} \right)= C_{H }^n\prod_{i=1}^{n} w_i^{|H|-d}\,.
\end{align*}
We have finished the proof.
\end{proof}

{Finally, we have our main result in this section.}
\begin{proposition}\label{est_un}
 The second moment of $n$-th chaos   $u_n(t, x)$ given by \eqref{e:3.3}-\eqref{e.3.3}   can be bounded as follows.~\\
(1)  If \ $\frac12d- |H|+\frac12\be^*>0$, then {for any $(t,x)\in\RR_+\times\RR^d$}
 \begin{align}\label{E[u_n(t,x)^2]}
	\EE[u_n(t,x)^2]&\leq C^n t^{n(|H|+2H_0-d)} n!\cdot\sum_{\bm{\al}\in \cD_n} \int_{{\bT_1\times\bT_1}}(1-s_n)^{-\rho_n}(1-r_{n})^{-\rho_n}\nonumber\\
	&\qquad\times \prod_{i=1}^{n-1} |s_{i+1}-s_i|^{-\rho_i} \prod_{i=1}^{n-1}|r_{i+1}-r_{i}|^{-\rho_i}{\prod_{i=1}^n\gamma_0(s_i-r_{i})} d\vec{\bfs} d\vec{\bfr}\,,
\end{align}
with $\bm{\al}=(\al_1,\cdots,\al_n) \in \cD_n $ is given by \eqref{al_i} and { the dependence of $\{\rho_i\}$ on $\bm{\al}$ is determined by}
\begin{equation}\label{rho_i}
{\begin{cases}
\rho_i=\frac{1}{2}(\frac32d+\frac12\be^*-2|H|-\al_i)\,, 1\leq i\leq n-1,\\
\rho_n=\frac{1}{2}(d-|H|-\al_n)\,.
\end{cases}}
\end{equation}
(2)  If \ $\frac12d- |H|+\frac12\be^*\le 0$, then
 \begin{align}\label{E[u_n(t,x)^2_2]}
	\EE[u_n(t,x)^2]&\leq C^n t^{n(|H|+2H_0-d)} n!\cdot \int_{{\bT_1\times\bT_1}}{ \prod_{i=1}^{n}|r_{i+1}-r_{i}|^{-\frac{1}{2}(d-|H|) }}\nonumber\\
	&\qquad{ \times \prod_{i=1}^{n} |s_{i+1}-s_i|^{-\frac{1}{2}(d-|H|) } } {\prod_{i=1}^n\gamma_0(s_i-r_{i})} d\vec{\bfs} d\vec{\bfr}\,.
\end{align}
\end{proposition}

\begin{proof}
By \eqref{hk} and \eqref{e.4.1}, we see that
\begin{align}
 \EE  \left[ u_n(t,x)^2 \right]
\leq   & C^n \int_{0<s_1<\cdots<s_n<t\atop 0< r_1, \cdots, r_n<t}  h_{ n}^{1/2}  (\vbfs )h_{ n}^{1/2}  (\vbfr )
\prod_{i=1}^n \gamma_0(s_i-r_i) d\vec{\bfs} d\vec{\bfr}\,,
 \end{align}
By substituting the bound for $h_n$ as derived in Lemma \ref{est_gk} into the inequality above, we obtain
\begin{align}\label{u_n_est}
\EE\left[ u_n(t,x)^2\right]
\leq& C^n \sum_{\bm{\al}\in \cD_n}  n! \int_{{\bT_t\times\bT_t }} (t-s_n)^{\frac{1}{2}(\al_n+|H|-d)}(t-r_{n})^{\frac{1}{2}(\al_n+|H|-d)} \nonumber \\
&\qquad\qquad \times\prod_{i=1}^n\gamma_0(s_i-r_{i}) \prod_{i=1}^{n-1} |s_{i+1}-s_i|^{\frac12 (\al_i+2|H|-\frac32d-\frac12\be^*)}\nonumber \\
&\qquad\qquad\times \prod_{i=1}^{n-1}|r_{i+1}-r_{i}|^{\frac12 (\al_i+2|H|-\frac32d-\frac12\be^*)} d\vec{\bfs} d\vec{\bfr} \,.
\end{align}
Now a  change of  variables $s_i \to t\cdot s_i$ and $r_i\to t\cdot r_i$ $(1\leq i\leq n)$  yields
\eqref{E[u_n(t,x)^2]}. The estimate \eqref{E[u_n(t,x)^2_2]} can be proved similarly.
\end{proof}

\begin{remark} { We make a remark on the condition 
 $ \be_k\in (4H_k-3, 2H_k-1)$. It is used in  the equation \eqref{E[u_n(t,x)^2]}, {where we wish $\rho_i$ defined in \eqref{rho_i} or equivalently $\be^*$ to be as small as possible.} Thus, during  the  proofs in the future we may take
$\beta_k=\lt\{\left(4H_k-3\right)\vee 0\rt\}+\vare$ {and}
$\be^*=4H^*-3d^*+\vare$ for any arbitrarily small $\vare>0$.  
However, the presence of this $\vare$ 
would distract the main idea of the proof.   Since { our final conditions in \eqref{cond.0<H<1} are in strict inequality}, we can let $\vare=0$,
namely, $\beta_k=\lt\{\left(4H_k-3\right)\vee 0\rt\} $ {and}
$\be^*=4H^*-3d^* $,  in the following proofs to simplify  the presentation.   This will not cause problem. 
For instance, let us take a look of  the condition $\rho_1<1$ in equation \eqref{Cond_z's}, which  is equivalent to 
$\frac{1}{2}(\frac32d+\frac12\be^*-2|H|)<1$. If this is true, then  
it is easy to see that one can find a (sufficiently small)  $\vare$ such that $\frac{1}{2}(\frac32d+\frac12(\be^*+\vare)-2|H|)<1$. 
In another word, the condition that there is a $\vare>0$ such that $\frac{1}{2}(\frac32d+\frac12(\be^*+\vare)-2|H|)<1$ is equivalent to the condition 
$\frac{1}{2}(\frac32d+\frac12 \be^* -2|H|)<1$.  
This means that if we use  $\be^*$ to replace
$\be^*+\vare$,  we will obtain the same set of conditions 
since we have only finite set of conditions. } %

\end{remark}

\section{Necessary conditions and sufficient conditions}\label{s.proof}

In this section, we give a proof of Theorem \ref{main_result1}.
\subsection{Proof of Theorem \ref{main_result1}: Sufficiency}\label{sufficient}
Now we begin to prove the sufficiency of the conditions in Theorem \ref{main_result1}.  Proposition \ref{main_cor4} and Proposition \ref{main_prop5} will be by-products  of our proof. We only consider $H_0>1/2$. The case when $H_0=1/2$ can be proved in the same way, and conditions \eqref{cond.0<H<1} becomes \eqref{H0=1/2} in this case.
%
%
%
%
%
%
We shall divide   the proof into  two cases: { (I)  $\frac12 d_*-H_*\le d^*-H^*$  and (II) $\frac12 d_*-H_*> d^*-H^*$.}


\textbf{The case $\frac12 d_*-H_*\le d^*-H^*$.}
By \eqref{E[u_n(t,x)^2_2]}, we have {for any $(t,x)\in\RR_+\times\RR^d$}
\begin{align}
\EE[u_n(t,x)^2]
\leq& C^n n! t^{n(|H|+2H_0-d)} \int_{{\bT_1\times\bT_1} }   \prod_{i=1}^n(s_{i+1}-s_i)^{\frac{|H|-d}{2}}\nonumber\\
&\qquad\qquad\qquad\qquad\qquad (r_{i+1}-r_{i})^{\frac{|H|-d}{2}}
\cdot|s_i-r_i|^{2H_0-2}  d\vec{\bfs} d\vec{\bfr} \,,
\label{e.4.4}
\end{align}
here in this case, {recall the conventions $\bT_1$ is ordered set defined in \eqref{Tsr} and} $s_{n+1}=r_{n+1}=1$.
Let $\tilde h_n(\vbfs )$ be the symmetric extension of
$\prod_{i=1}^n(s_{i+1}-s_i)^{\frac{|H|-d}{2}} \1_{\{0< s_1<\cdots<s_n< 1\}}$ to $[0, 1]^n$.
The application of  the { Hardy-Littlewood-Sobolev}   inequality (see Lemma \ref{Hardy-Littlewood-Sobolev} in Appendix) yields
\begin{align}
\EE[u_n(t,x)^2]\leq&
C^n(n!)^{-1} t^{n(|H|+2H_0-d)} \int_{[0,1]^{2n}  }   \tilde h_n(\vbfs) \tilde h_n(\vbfr)
\cdot|s_i-r_i|^{2H_0-2}  d\vec{\bfs} d\vec{\bfr}\nonumber \\
=&C^n(n!)^{-1} t^{n(|H|+2H_0-d)} \left[ \int_{[0,1]^{ n}  }   \tilde h_n(\vbfs)^{\frac{1}{H_0}}    d\vec{\bfs}  \right]^{2H_0} \nonumber \\
\le &C^n (n!)^{2H_0-1}  t^{n(|H|+2H_0-d)} C_H^n\lt(\int_{\bT_1}\prod_{i=1}^n(s_{i+1}-s_i)^{\frac{|H|-d}{2H_0}}d\vbfs\rt)^{2H_0}\,.
\nonumber 
\end{align}
Then, by \cite[Lemma 4.5]{HHNT} we get
\begin{align}
	\EE[u_n(t,x)^2]\leq &C^n (n!)^{2H_0-1} C_H^n t^{n(|H|+2H_0-d)} \left[ \Gamma \left(\left(\frac{|H|-d}{2H_0}+1\right)n +1\right)\right]^{-2H_0}\nonumber \\
=&C^n (n!)^{-( |H|-d+1)} C_H^n t^{n(|H|+2H_0-d)}\,.\label{e.4.6}
\end{align}
We write the statement \eqref{e.4.6} in short, for any $(t,x)\in\RR_+\times\RR^d$
\begin{equation}
\|u_n(t,x)\|_{2}\le C_H^n (n!)^{-\frac{ |H|-d+1}{2}}   t^{\frac{n(|H|+2H_0-d)}{2}}
\end{equation}
under the condition
\[\frac{|H|-d}{2H_0}>-1\Leftrightarrow |H|+2H_0>d,
\]
which is exactly the { second condition \eqref{cond.0<H<1_b}}. 
Furthermore, { by the hypercontractivity inequality}, we have {for any $(t,x)\in\RR_+\times\RR^d$}
\begin{equation}
\|u_n(t,x)\|_p\le p^{n/2} C_H^n (n!)^{-\frac{ |H|-d+1}{2}}   t^{\frac{n(|H|+2H_0-d)}{2}}\,. \label{e.4.7}
\end{equation}
Now,  if
\begin{equation}
|H|>d-1\,,
\end{equation}
then the chaos expansion of the solution is summable in $L^p(\Omega)$. {By the Stirling formula, we have $\Gamma (an+1)\sim (n!)^a$, and then by the evaluation of Mittag-Leffler summation (see, e.g., \cite[Lemma A.3]{BalanSong2019}), it holds that {for any $(t,x)\in\RR_+\times\RR^d$}}
\begin{align*}
\EE \left[\left| u (t,x)\right|^p \right]
\le &\left(\sum_{n=0}^\infty \|u_n(t,x)\|_p\right)^p  \\
\le & \left( \sum_{n=0}^\infty  p^{n/2} C_H^n (n!)^{-\frac{ |H|-d+1}{2}}   t^{\frac{n(|H|+2H_0-d)}{2}} \right)^p \\
\leq & \left( \sum_{n=0}^\infty \frac{p^{n/2} C_H^n t^{\frac{n(|H|+2H_0-d)}{2}}}{\Gamma(\frac{ |H|-d+1}{2}n+1)}   \right)^p \\
\leq & C_{ H} \exp
\left[ c_{H} p^{\frac{|H|-d+2} {|H|-d+1}}  t  ^{\frac{|H|+2H_0-d}{|H|-d+1}}\right] \,,  
\end{align*}
where the generic positive constants $C_H$ and $c_H$ may differ from line to line.

\begin{remark}
{It is interesting to note that the right-hand side of \eqref{e.4.4}, as a function of $H_0\in[\frac 12,1]$, takes the minimum value at $H_0=1$}, which is exactly equal to
\begin{align}
 &n! t^{n(|H|+2-d)} \left[\int_{ {\bT_1} }   \prod_{i=1}^n(s_{i+1}-s_i)^{\frac{|H|-d}{2}}  d\vec{\bfs} \right]^2 \nonumber\\
 &\qquad = C^n    t^{n(|H|+2-d)} (n!)^{1-2\left(1+
 \frac{|H|-d}{2}\right)}= (n!)^{-( |H|-d+1)}t^{n(|H|+2-d)}\,.
\end{align}
 In terms of the exponent of $n!$,  it is the same as the estimate \eqref{e.4.6}. This means that our approach to using the Hardy-Littlewood-Sobolev inequality is ``sharp" if we want $u(t,x)$ to be summable. {Moreover, $H_0$ has no contribution to the exponent of $n!$ on the right-hand side of \eqref{e.4.6}, which means 
 that it plays no role in guaranteeing the summability of the right-hand side of \eqref{e.4.6}.}
 \end{remark}


\smallskip
\textbf{The case $  \frac12 d_*-H_*> d^*-H^*$.} This case is more complicated. From \eqref{E[u_n(t,x)^2]} in Proposition \ref{est_un}, it suffices to consider the following integral:
\begin{align}\label{Mainterm_d=2}
	{\cI_{\vec{\bm{\rho}},\gamma_0}}:=&\int_{{\bT_1\times\bT_1}}(1-s_n)^{-\rho_n}(1-r_{n})^{-\rho_n} {\prod_{i=1}^n\gamma_0(s_i-r_{i})}\\
	&\qquad\qquad \times \prod_{i=1}^{n-1} |s_{i+1}-s_i|^{-\rho_i} \prod_{i=1}^{n-1}|r_{i+1}-r_{i}|^{-\rho_i} d\vec{\bfs} d\vec{\bfr} \,,\nonumber
\end{align}
where we denote $\vrho   \,$  by $(\rho_1,\cdots,\rho_n)$ and $\rho_i$, $i=1,\cdots,n$ are given by \eqref{rho_i}.

If we use the Hardy-Littlewood-Sobolev inequality as for \eqref{e.4.6} { and by \cite[Lemma 4.5]{HHNT}}, we obtain   when $\rho_i<H_0$ for $i=1, \cdots, n$
\begin{align}
\EE[u_n(t,x)^2]\leq&  (n!)^{2H_0-1}  t^{n(|H|+2H_0-d)} C_H^n\lt(\int_{{\bT_1}}\prod_{i=1}^n(s_{i+1}-s_i)^{-\frac{\rho_i}{ H_0}}d\vbfs\rt)^{2H_0}
\nonumber \\
= & (n!)^{2H_0-1} C_H^n t^{n(|H|+2H_0-d)} \left[ \Gamma \left(\sum_{i=1}^n (1-\frac{ \rho_i}{H_0})\right)\right]^{-2H_0}\nonumber \\
=& (n!)^{-( |H|-d+1)} C_H^n t^{n(|H|+2H_0-d)}\,, \label{e.4.12}
\end{align}
where the last step is followed by the Stirling formula.  
From the definition of $\rho_i$'s, the condition  $\rho_i<H_0$ for $i=1, \cdots, n$ is equivalent to
\[
 H_0+ |H|> \frac34 d+\frac{\beta^*}{4} \quad \hbox{or}\quad H_0+ H_*> \frac34 d_* \,.
\]
{By the similar arguments as in the case $\frac12 d_*-H_*\le d^*-H^*$ below \eqref{e.4.6}, the  Proposition \ref{main_cor4}   is  proved.}

Now we return to the question concerning the finiteness of $\EE \left[ u_n^2 (t,x)\right]$. It is self-evident that  $|H|\ge  d-1$ is sufficient.   But it is certainly not necessary since $H_0$  may  play a role now.  We would like to seek the necessary and sufficient conditions.

 The right-hand side of \eqref{E[u_n(t,x)^2]} is a multiple integrals with some
simple integrands. At first glance, we may think that the integrability problem of these kernels is simple. However,  it  can be { a complicated problem} in analysis. { There are some studies} about similar integrals
(e.g., \cite{Shi, Wu, Zhou}). However, we cannot find results which are directly applicable.
The difficulty is that the  {$\rho_i's$ appeared }in the integrand in
\eqref{Mainterm_d=2}  are different.   To solve this integrability problem,
our idea is to use both the Hardy-Littlewood-Sobolev inequality and
  the  H\"older-Young-Brascamp-Lieb inequality obtained in \cite{BCCT2008}, which we shall recall in Appendix \ref{appenA}. 

Take arbitrary   $\bm{\alpha}:=(\alpha_1,\cdots,\alpha_n)
\in \cD_n$. Notice  that $\al_i+\al_{i+1}\neq 0$.
 We divide our discussion on the integration  $\cI_{\vec{\bm{\rho}},\gamma_0}$ according to the following two cases.
\begin{equation}\label{Algo}
	\begin{cases}
\hbox{\bf Case 1:} \qquad 		\al_1\neq 0\,, &\text{we use}   \text{ the Hardy-Littlewood}\\
&\qquad\qquad\qquad\qquad\hbox{-Sobolev inequality } \eqref{ineq_Hardy-Littlewood-Sobolev} \,;\\
\hbox{\bf Case 2:} \qquad		\al_1= 0\,, \al_2 \neq 0\,, &\text{we   use  }   \text{ the H\"older-Young}\\
&\qquad\qquad\ \ \hbox{-Brascamp-Lieb  inequality  }\eqref{Ineq.HYBL} \,.
	\end{cases}
\end{equation}

\begin{remark}
As mentioned previously in   introduction, it is widely believed  that     the Hardy-Littlewood-Sobolev inequality is the best tool to  handle  the integral \eqref{Mainterm_d=2} over the simplex $\bT_1\times\bT_1$.  However, for the specific case of $d=1$ with $H_1$ abbreviated as $H$, this approach yields the sufficient  condition $H+H_0\geq \frac 34$   which  is,  as we see now,  not a necessary  one.  This approach was first introduced in \cite{HHLNT2,HHLNT1} and then used in some related references \cite{Balan,BalanSong2019,CH2021,SSX2020}, just to name a few. In particular, it utilizes the Hardy-Littlewood-Sobolev inequality for all $\bm{\alpha}\in \cD_n$ simultaneously, overlooking the nuanced yet crucial distinctions we've highlighted in the algorithm \eqref{Algo}, especially Case 2.
\end{remark}

Let us discuss the cases listed in \eqref{Algo} separately.

\noindent {\bf Case 1:} When $\al_1\neq0$, we   integrate $s_1$ and $r_1$ in  \eqref{Mainterm_d=2} first. This means we write
\begin{align}\label{Mainterm.1}
	{\cI_{\vec{\bm{\rho}},\gamma_0}}:=&\int_{{0<s_2<s_3<\cdots<s_{n}<t \atop 0<r_2<r_3<\cdots<r_{n}<t}}(1-s_n)^{-\rho_n}(1-r_{n})^{-\rho_n} {\prod_{i=2}^n\gamma_0(s_i-r_{i})}\\
	&\times \prod_{i=2}^{n-1} |s_{i+1}-s_i|^{-\rho_i} \prod_{i=2}^{n-1}|r_{i+1}-r_{i}|^{-\rho_i} \cI_1(s_2, r_2)
%
d\vec{\bfs}_2 d\vec{\bfr}_2 \,, \nonumber
\end{align}
where
\begin{equation}
\cI_1(s_2, r_2):=\int_{{0<s_1<s_2 \atop 0<r_1<r_2}} \gamma_0(s_{1}-r_{1})|s_{2}-s_{1}|^{-\rho_{1}}|r_{2}-r_{1}|^{-\rho_{1}} dr_1ds_1 \,. \label{e.def_I1}
\end{equation}
Noticing  that under the second inequality \eqref{cond.0<H<1_b} :$|H|+2H_0>d$,  it holds
\begin{align}
H_0-\rho_1&=H_0-\frac12\lt[\frac32d+\frac12\be^*-2|H|-\al_1\rt]>0,\label{e.cI_1}
\end{align}
since $\al_1\in \lt\{\frac12d- |H|+\frac12\be^*, d- 2|H|+\be^* \rt\}$. Then an application of  the Hardy-Littlewood-Sobolev inequality (Lemma \ref{Hardy-Littlewood-Sobolev} in Appendix) yields
\begin{align*}
\cI_1(s_2, r_2)	\leq& C_{H_0,\rho_1} \lc\int_{0<s_1<s_2}|s_{2}-s_{1}|^{-\frac{\rho_{1}}{H_0}} ds_1\rc^{H_0} \lc\int_{0<r_1<r_2}|r_{2}-r_{1}|^{-\frac{\rho_{1}}{H_0}} dr_1\rc^{H_0} \\
	=& C_{H_0,\rho_1} B\lc1,1-\frac{\rho_{1}}{H_0}\rc^{2H_0} s_2^{H_0-\rho_1}r_2^{H_0-\rho_1} \leq C_{H_0,\rho_1}  t^{2(H_0-\rho_1)}\,,
\end{align*}
where { we recall $B(x,y)=\frac{\Gamma(x)\Gamma(y)}{\Gamma(x+y)}$} is the Beta function.

\bigskip
\noindent {\bf Case 2:}  When $\al_1=0$, $\al_2 \neq 0$, we   integrate $s_1, s_2, r_1, r_2$ in  \eqref{Mainterm_d=2} first.  Namely, we write now
 \begin{align}\label{Mainterm.2}
	{\cI_{\vec{\bm{\rho}},\gamma_0}}:=&\int_{{0<s_3<\cdots<s_{n}<t \atop 0<r_3<\cdots<r_{n}<t}}(1-s_n)^{-\rho_n}(1-r_{n})^{-\rho_n} {\prod_{i=3}^n\gamma_0(s_i-r_{i})}\\
	&\times \prod_{i=3}^{n-1} |s_{i+1}-s_i|^{-\rho_i} \prod_{i=3}^{n-1}|r_{i+1}-r_{i}|^{-\rho_i} \cI_2(s_3, r_3)  d\vec{\bfs}_3 d\vec{\bfr}_3 \,,\nonumber
\end{align}
where
\begin{align}\label{I_2(s3,r3)}
\cI_2  (s_3, r_3):=&\int_{{0<s_1<s_2<s_3 \atop 0<r_1<r_2<r_3}}{\prod_{i=1}^{2}\gamma_0(s_{i}-r_{i})}   |s_{3}-s_{2}|^{-\rho_{2}}|r_{3}-r_{2}|^{-\rho_{2}}\nonumber\\
&\qquad\qquad |s_{2}-s_{1}|^{-\rho_{1}}|r_{2}-r_{1}|^{-\rho_{1}} ds_1ds_2dr_1dr_2 \,.
\end{align}
We claim that
\begin{equation*}
\sup_{0\le s_3, r_3\le t} \cI_2(s_3, r_3)\leq C_t<+\infty\,.
\end{equation*}
We shall prove it via the so-called non-homogeneous H\"older-Young-Brascamp-Lieb inequality (see Theorem \ref{Thm.HYBL_N} in Appendix). {This inequality aims at the following multilinear functional
\begin{equation}\label{Def.Multi}
	\Lambda(f_1,\dots,f_m)=\int \prod_{j=1}^m f_j(l_j(\bfx))\prod_{j=1}^n d\mu(x_j)\,,
\end{equation}
where $l_j(\cdot)$'s are linear functions. In the seminal paper \cite[Theorem 2.2]{BCCT2010},  the authors established the condition \eqref{Cond.codim} that is both necessary and sufficient for the finiteness of $\Lambda(f_1,\dots,f_m)$.} 
To use this theorem, we introduce the following linear transformations {$l_j :\RR^4\to \RR$}
\begin{align*}
	&l_1(s_1,s_2,r_1,r_2)=r_2-r_1\,,~ l_3(s_1,s_2,r_1,r_2)=r_2\,,~ l_5(s_1,s_2,r_1,r_2)=s_1-r_1\,,\\
	&l_2(s_1,s_2,r_1,r_2)=s_2-s_1\,,~ l_4(s_1,s_2,r_1,r_2)=s_2\,,~ l_6(s_1,s_2,r_1,r_2)=s_2-r_2\,,
\end{align*}
and  the  nonnegative functions $f_j:\RR^4\to \RR_+$
\begin{align*}
	&f_1(x)=|x|^{-\rho_1}\1_{\{0< x< r_3\}}\,,~ f_3(x)=|r_3-x|^{-\rho_2}\1_{\{0< x< r_3\}}\,,~ f_5(x)=|x|^{-\gamma_0}\1_{\{|x|<1\}}\,,\\
	&f_2(x)=|x|^{-\rho_1}\1_{\{0< x< s_3\}}\,,~ f_4(x)=|s_3-x|^{-\rho_2}\1_{\{0< x< s_3\}}\,,~ f_6(x)=|x|^{-\gamma_0}\1_{\{|x|<1\}}\,,
\end{align*}
where {$\gamma_0=2-2H_0$ is the fractional power in the kernel $\gamma_0(x)=|x|^{-\gamma_0}=|x|^{2H_0-2}$} and
\begin{align*}
\begin{cases}	
\rho_1=\frac{1}{2}\lt(\frac32d+\frac12\be^*-2|H|\rt),\\
\rho_2=\frac{1}{2}\lt(\frac32d+\frac12\be^*-2|H|-\al_2\rt)\,.
\end{cases}
\end{align*}
With these notations we can first bound $\cI_2(s_3, r_3)$
\begin{align*}
\cI_2(s_3, r_3)
\leq&\int_{{0<s_1,s_2<1 \atop 0<r_1,r_2<1}} {\prod_{i=1}^{2}\gamma_0(s_{i}-r_{i})}|s_{3}-s_{2}|^{-\rho_{2}}\1_{\{0<s_2<s_3\}} |r_{3}-r_{2}|^{-\rho_{2}}\1_{\{0<r_2<r_3\}}\\
	&\qquad\times |s_{2}-s_{1}|^{-\rho_{1}}\1_{\{0<s_2-s_1<s_3\}}|r_{2}-r_{1}|^{-\rho_{1}}\1_{\{0<r_2-r_1<r_3\}}ds_1ds_2dr_1dr_2\\
	=:&\int_{[0,1]^4} \prod_{j=1}^6 f_j(l_j(s_1,s_2,r_1,r_2)) ds_1ds_2dr_1dr_2 =:\cJ(f_1,\cdots,f_6).
\end{align*}

Now, { the above integral $\cJ(f_1,\cdots,f_6)$} is in a form that the H\"older-Young-Brascamp-Lieb theorem may apply.  We want to show that under the condition
\eqref{cond.0<H<1} we  can appropriately choose $p_1, \cdots, p_6$ such that {the dimension condition \eqref{Cond.codim}} of Theorem \ref{Thm.HYBL_N} in the appendix is verified so that  we can apply  this H\"older-Young-Brascamp-Lieb   theorem   to obtain
\begin{equation}\label{Mainterm.2_Goal}
\cI_2(s_3, r_3)\le 	\cJ(f_1,\cdots,f_6)\leq \prod_{j=1}^6 \|f_j\|_{L^{p_j}(\RR)}\leq C<\infty \,.
\end{equation} 
To verify {the dimension condition \eqref{Cond.codim}} of Theorem \ref{Thm.HYBL_N}, we want to show that there exist {$p_1, \cdots, p_6\in[1,+\infty)$} so that
\begin{equation}
\begin{split}
\hbox{$f_j\in L^{p_j}(\RR)$ and for every subspace $V\subseteq \RR^4$:}\\
co\dim_{\cH} (V)\geq \sum\limits_{j=1}^6  \frac{1}{p_j} \cdot  co\dim_{\cH_j} (l_j(V))\,.\quad 
\end{split} \label{e.dimension}
\end{equation}

Denote $z_j=p_j^{-1}\in(0,1]$ for $j=1,\cdots,6$.    Take $\be^*=\be_{d_*+1}+\cdots+\be_d=4H^*-3d^*$ in Lemma \ref{est_gk}.  Then the integrability conditions $f_j\in L^{p_j}(\RR)$,    $j=1,\cdots,6$ are equivalent to
\begin{equation}\label{Cond_z's}
	\lt\{\begin{split}
	&1\geq z_1\,,z_2>\rho_1=\frac{1}{2}(\frac32d+\frac12\be^*-2|H|) {=\frac12\lt(\frac32d_*-2H_*\rt)}\,,\\
	&1\geq z_3\,,z_4>\rho_2=\frac{1}{2}(\frac32d+\frac12\be^*-2|H|-\al_2) {=\frac12\lt(\frac32d_*-2H_*-\al_2\rt)},\\ &1\geq z_5\,,z_6 >\gamma_0=2-2H_0\,.
	\end{split}\rt.
\end{equation}

By the assumption  $H_0>\frac12$ we can choose
$z_5, z_6$ so that the  third line  in \eqref{Cond_z's} holds true. As for the first and second lines, since $\al_2>0$, we only need to explain that there are $z_1$ and $z_2$ such that  the first line in \eqref{Cond_z's} holds.
Note that
\begin{align*}
\rho_1<1&\Leftrightarrow \frac32d_*-2H_*<2 \quad \\
&\Leftrightarrow H_*>\frac34d_*-1. 
\end{align*}
Therefore,  the integrability conditions \eqref{Cond_z's} are implied by   the assumption $H_0>\frac12$ and the first inequality \eqref{cond.0<H<1_a}:
\begin{align}\label{condi_1}
H_*>\frac34d_*-1\,.
\end{align}


On the other hand, we need to verify the dimensional conditions in  \eqref{e.dimension}. For fixed co-dimension of $V$, our strategy is to select $V$ such that $co\dim_{\cH_j} (l_j(V))=1$ {as many $j$ as possible}, i.e. $\dim_{\cH_j} (l_j(V))=0$ {as many $j$ as possible}. It is not difficult to { see that} 
\[\dim_{\cH_j} (l_j(V))=0\ \hbox{if }\ V=\ker(l_j).\]
Therefore, we only need to take into account those $V$'s which are given by  $\ker(l_j)$ ($1\leq j\leq 6$) or their intersections.
Thus,  we shall select $J\subseteq\{1,\cdots,6\}$ such that $V=\cap_{j\in J} \ker (l_j)$ and in this case
\begin{align}\label{Cond.dim_ker}
	\dim_{\cH}(V)=&\dim(\cap_{j\in J} \ker (l_j))\nonumber\\
=&\dim(\cH)-\dim\lt(span\{l_j, j\in J\}\rt)\nonumber \\
	=&\dim(\cH)-\rank([l_j]_{j\in J})\,.
\end{align}
In the following, we shall discuss the cases  $co\dim_\cH (V)=4,3,2,1,0$. It is trivial to see that the dimension condition \eqref{e.dimension} holds when $co\dim_\cH (V)=0$. Therefore, we just have to verify  \eqref{e.dimension}
for $co\dim_\cH (V)=4,3,2,1$ step by step.  This is done in  Lemma
\ref{verify_HYBL} in Appendix
under conditions \eqref{cond.0<H<1}. This means that  \eqref{Mainterm.2_Goal} is achieved under conditions \eqref{cond.0<H<1} by using  the H\"older-Young-Brascamp-Lieb  inequality.


{To summarize {{\bf Case 1}} and {{\bf Case 2}},} we obtain that $\cI(\rho,\gamma)$ defined by \eqref{Mainterm_d=2} can be bounded as
\begin{numcases}{{\cI_{\vec{\bm{\rho}},\gamma_0}}\leq}
	B_1\cdot \int_{{0<s_2<s_3<\cdots<s_{n}<t \atop 0<r_2<r_3<\cdots<r_{n}<t}}(1-s_n)^{-\rho_n}(1-r_{n})^{-\rho_n} {\prod_{i=2}^n\gamma_0(s_i-r_{i})}\nonumber \\
\qquad  \times \prod_{i=2}^{n-1} |s_{i+1}-s_i|^{-\rho_i} \prod_{i=2}^{n-1}|r_{i+1}-r_{i}|^{-\rho_i}
d\vec{\bfs}_2 d\vec{\bfr}_2
 	\quad \hbox{$\al_1\neq 0$};  \label{Maintm_1} \\
	C\cdot\int_{{0<s_3<\cdots<s_{n}<t \atop 0<r_3<\cdots<r_{n}<t}}(1-s_n)^{-\rho_n}(1-r_{n})^{-\rho_n} {\prod_{i=3}^n\gamma_0(s_i-r_{i})}\nonumber \\
\qquad 	 \times \prod_{i=3}^{n-1} |s_{i+1}-s_i|^{-\rho_i} \prod_{i=3}^{n-1}|r_{i+1}-r_{i}|^{-\rho_i}    d\vec{\bfs}_3 d\vec{\bfr}_3\,,\quad  \hbox{$\al_1= 0\,,\al_2\neq 0$}. \nonumber\\
\label{Maintm_2}
\end{numcases}
where $B_1:=B\lc1,1-\frac{\rho_{1}}{H_0}\rc^{2H_0}$.
The remaining integration has the same form as the original one but with strictly less multiplicity.  We can then  use  the same argument
as above to  prove 
\[{\cI_{\vec{\bm{\rho}},\gamma_0}}\leq C^n<\infty.\] 
Thus, we have {for any $(t,x)\in\RR_+\times\RR^d$}
\begin{equation}\label{Result}
	\EE[u_n(t,x)^2]\leq C^n\cdot n!\cdot t^{n(|H|+2H_0-d)}<\infty\,.
\end{equation}

As a result,  we have proven \eqref{Result} under the conditions \eqref{cond.0<H<1} and we finish the proof of sufficiency in Theorem \ref{main_result1}. 


\subsection{Sufficient condition of   Theorem \ref{main_result1}: An alternative  proof }
In this subsection, we give another proof of the sufficient condition using some elementary computations in the Appendix \ref{appenA}. {Let us stress that Lemma \ref{lem_mainterm} is the key for the proofs here.}

As explained in the previous section, { we need to show that}
$\cI_1(s_2, r_2)$ defined by \eqref{e.def_I1} is
bounded  under the condition   $\al_1\not=0$
and $\cI_2(s_3, r_3)$ defined by \eqref{I_2(s3,r3)} is
bounded  under the condition   $\al_1 =0$ but $\al_2\not=0$. Recall that
\begin{equation*}
\begin{split}
\rho_i=&\frac{1}{2}(\frac32d+\frac12\be^*-2|H|-\al_i)
\\
= &
\begin{cases}
\frac{3}{4} d_*-H_*,  &\qquad \hbox{if $\al_i=0$}\\
\frac{d-|H|}{2},   &\qquad \hbox{if $\al_i=\frac12d- |H|+\frac12\be^*$}.\\
\end{cases}
\end{split}
\end{equation*}

\smallskip
\noindent\textbf{Case 1: $\al_1\not=0$}.  { In this case}, we apply Lemma \ref{lem_mainterm} to $\cI_1(s_2, r_2)$   with
\begin{eqnarray*}
  \mathfrak{\be}&=&0\,,\quad \ga_0=2-2H_0\,,\quad a=b=0\,,\quad A=s_2\,,\quad B=r_2 \,,\\
 \al&=&\rho_1=\frac{1}{2}(\frac32d+\frac12\be^*-2|H|-\al_1)\,,
\end{eqnarray*}
and $\al_1
 \in  \lt\{\frac12d- |H|+\frac12\be^*, d- 2|H|+\be^* \rt\}$. In order to use this lemma, we should require the following condition
\begin{align*}
\ga_0<2-\al\Leftrightarrow  H_*+2H_0>\frac34d_* ,
\end{align*}
which is implied by the third inequality \eqref{cond.0<H<1_c}.
Then,
\begin{equation}
\cI_1(s_2, r_2)\lesssim
s_2^{1-\rho_1-\frac{\ga_0}{2}} r_2^{1-\rho_1-\frac{\ga_0}{2}}+
s_2^{1-\rho_1-q_2} r_2^{1-\rho_1-q_2} |r_2-s_2|^{-\ga_0+2q_2}.
\label{e.cI_1est}
\end{equation}
The right-hand side of \eqref{e.cI_1est} is bounded if
\begin{equation}
\begin{cases}
1-\rho_1-\frac{\ga_0}{2}>0 \,;\\
-\ga_0+2q_2\ge 0
\,,
\end{cases} \label{e.case1}
\end{equation}
for some  $q_2\in (0,(1-\rho_1)\wedge \frac{\ga_0}{2}]$. Since $\alpha_1\not=0$, it follows that $\rho_1\le \frac12(d-|H|)$.
The first condition of \eqref{e.case1} is then equivalent to
$|H|+2H_0>d $,  which is the second inequality \eqref{cond.0<H<1_b}.   Since the first condition of \eqref{e.case1}  holds true, we have $\frac{\gamma_0}{2}<1-\rho_1$. Then we can    take $q_2=\frac{\gamma_0}{2}$ and  we see easily   that the second condition
 of \eqref{e.case1} is true.

\bigskip
\noindent\textbf{Case 2: $\al_1=0$ and $\al_2\not=0$}.  We shall use Lemma \ref{lem_mainterm} to bound $\cI_2(s_3, r_3)$ given by \eqref{I_2(s3,r3)}. Notice that
in this case we should  also integrate  $\cI_1(s_2, r_2)$ first to obtain
 \eqref{e.cI_1est}. But in this case since  $\al_1=0$, it follows \[\rho_1=\frac{1}{2}(\frac32d+\frac12\be^*-2|H|)=\frac12\lt(\frac32d_*-2H_*\rt).\]
In order to use Lemma \ref{lem_mainterm},
  we need $\rho_1<1$. This is equivalent to
\begin{equation}\label{alt_fir}
H_*>\frac34d_*-1
\end{equation}
which is the first inequality \eqref{cond.0<H<1_a}. Under this condition we can estimate $\cI_1(s_2, r_2)$ as
  \eqref{e.cI_1est}.

If $\frac{\ga_0}{2}\le 1-\rho_1$, then we can choose $q_2=\frac{\ga_0}{2}$ and $\cI_1(s_2, r_2)$ is also bounded.  Hence, in this case we may proceed as in the case $\al_1\not=0$.

If $\frac{\ga_0}{2} > 1-\rho_1$, namely,
\begin{equation}
H_0+H_*< \frac34 d_*\,, \label{e.cond_cI2}
\end{equation}
{then in this case} we choose $q_2=1-\rho_1$ and we have to consider  the  integral $\cI_2(s_3,r_3)$ given by \eqref{I_2(s3,r3)}.   We integrate $s_1$ and $r_1$ first. By \eqref{e.cI_1est}, we have
\begin{align}
\cI_2  (s_3, r_3)\le &\int_{{0< s_2<s_3 \atop 0 <r_2<r_3}}     |s_{3}-s_{2}|^{-\rho_{2}}|r_{3}-r_{2}|^{-\rho_{2}}(s_2 r_2)^{1-\rho_1-\frac{\ga_0}{2}} |s_2-r_2|^{-\ga_0}ds_2dr_2\nonumber\\
&+ \int_{{0< s_2<s_3 \atop 0 <r_2<r_3}}|s_{3}-s_{2}|^{-\rho_{2}}|r_{3}-r_{2}|^{-\rho_{2}}
|s_2-r_2|^{-2\ga_0+2-2\rho_1} ds_2dr_2 \nonumber\\
&=:\cI_{2,1}  (s_3, r_3)+\cI_{2,2}  (s_3, r_3)\,,
\end{align}
where we {recall that $\ga_0=2-2H_0$},
\[
\rho_2=\frac{1}{2}(\frac32d+\frac12\be^*-2|H|-\al_2)
 \quad\hbox{and}\quad \al_2
 \in  \lt\{\frac12d- |H|+\frac12\be^*, d- 2|H|+\be^* \rt\}\,.
 \]
 To use Lemma \ref{lem_mainterm} for $\cI_{2,1}  (s_3, r_3)$, we need
\[
{1-\rho_1-\frac{\ga_0}{2}}>-1\Leftrightarrow H_*+H_0>\frac34d_*-1\,,
\]
which is implied by the first condition \eqref{cond.0<H<1_a}. Similarly, for the term $\cI_{2,2}  (s_3, r_3)$,  we need
\[
-2\ga_0 +2-2\rho_1>-1\,.
\]
This condition is equivalent to
\begin{equation}\label{alt_third}
2H_0+H_*>\frac34 d_*+\frac12  \,,
\end{equation}
which is the third condition \eqref{cond.0<H<1_c}.

Now we  apply Lemma \ref{lem_mainterm}  to  $\cI_{2,1}  (s_3, r_3)$   with
  \[\al =\rho_2, \ \be =-(1-\rho_1-\frac{\ga_0}{2}), \ a=b=0,\  A=s_3 \hbox{ and } B=r_3.\]
It is clear that the following condition
 \begin{align*}
 &\ga_0<2-\be -(\be \vee \al )=2+1-\rho_1-\frac{\ga_0}{2}-(-1+\rho_1+\frac{\ga_0}{2})\vee\rho_2\\
& \Leftrightarrow\ {H_*+2H_0>\frac34d_* \hbox{ and } |H|
+2H_*+6H_0>d+\frac32d_*},
 \end{align*}
 holds under the third and {fourth inequalities} \eqref{cond.0<H<1_d}. Thus,
 we have for $\bar q_2\in[\be +\ga_0-1,(1-\al )\wedge \frac{\ga_0}{2}]=[-2+\rho_1+\frac32\ga_0,(1-\rho_2 )\wedge\frac{\ga_0}{2}]$,
\begin{align}\label{est_cI_1}
\cI_{2,1}  (s_3, r_3)\lesssim & |s_3|^{-\rho_2+(1-\rho_1-\frac{\ga_0}{2})+1-  \frac{\ga_0}{2}}|r_3|^{-\rho_2+(1-\rho_1-\frac{\ga_0}{2})+1-  \frac{\ga_0}{2}} \nonumber\\
   &+   |s_3|^{-\rho_2+(1-\rho_1-\frac{\ga_0}{2})+1-  \bar q_2}|r_3|^{-\rho_2+(1-\rho_1-\frac{\ga_0}{2})+1-  \bar q_2}  |s_3-r_3 |^{-\ga_0+2\bar q_2}.
\end{align}
Noticing that under the second inequality $|H|+2H_0>d$ of condition \eqref{cond.0<H<1}, we have $1-\rho_2>\frac{\ga_0}{2}$. Taking $\bar q_2=\frac{\ga_0}{2}$, $\cI_{2,1}  (s_3, r_3)$ is bounded if
\begin{align}\label{con_cI_21}
{-\rho_2+(1-\rho_1-\frac{\ga_0}{2})+1-  \frac{\ga_0}{2}}>0\Leftrightarrow {|H|+2H_*+4H_0>d+\frac32 d_*}.
\end{align}
We can see that the above first inequality is exactly \eqref{cond.0<H<1_d}.

{ Next, we apply} Lemma \ref{lem_mainterm}  to  $\cI_{2,2}  (s_3, r_3)$ with
\[\al =\rho_2,\ \be =0,\ \tilde\ga=2\ga_0-2(1-\rho_1),\ a=b=0,\ A=s_3 \hbox{ and } B=r_3.\]
Thus, under the condition
 \begin{align*}
 &\tilde\ga<2-\be -(\be \vee \al )=2-\rho_2 \\
& \Leftrightarrow\ |H|+4H_*+8H_0>d+3d_*,
 \end{align*}
 which is implied by the third and fourth inequalities of \eqref{cond.0<H<1},
we have
\begin{align}
\cI_{2,2}  (s_3, r_3)\lesssim & |s_3|^{-\rho_2+1-  \frac{\tilde\ga}{2}}|r_3|^{-\rho_2+1-  \frac{\tilde\ga}{2}} \nonumber\\
   &+   |s_3|^{-\rho_2+1-  \tilde q_2}|r_3|^{-\rho_2+1-  \tilde q_2}  |s_3-r_3 |^{-\tilde\ga+2\tilde q_2},
\end{align}
where $\tilde q_2\in[\be+\tilde\ga-1,(1-\al)\wedge \frac{\tilde\ga}{2}]=[2\ga+2\rho_1-3, (1-\rho_2)\wedge\frac{\tilde\ga}{2}]$.
Notice  that under the fourth inequality \eqref{cond.0<H<1_d}, we have $1-\rho_2> \frac{\tilde\ga}{2}$.  Thus, letting $\tilde q_2=\frac{\tilde\ga}{2}$, $\cI_{2,2}  (s_3, r_3)$ is bounded.


As a result, by Lemma \ref{lem_mainterm} in Appendix, we have proved  that under conditions \eqref{cond.0<H<1}, $\cI_1(s_2, r_2)$  and $\cI_2(s_3, r_3)$ are bounded. Then, in a similar argument to that in   Section \ref{sufficient}, we see that the conditions \eqref{cond.0<H<1} are sufficient for $\EE[u_n(t,x)^2]<+\infty$ for any $n\ge  1$.

\subsection{Proof of Theorem \ref{main_result1}: Necessity}
{In this subsection
 We shall show   the necessity of  condition \eqref{cond.0<H<1}.} Focusing on $H_0>1/2$ firstly, we shall prove the four inequalities in \eqref{cond.0<H<1} separately.


\medskip
\noindent \textbf{The necessity  of condition \eqref{cond.0<H<1_b}.} This is in fact  a consequence of
\cite{ams19_huliutindel}. But we shall give a simpler proof due 
to our special situation.  We only need to focus on the first chaos expansion $u_1(t,x)$. Namely, {for any $(t,x)\in\RR_+\times\RR^d$}
\begin{align*}
\EE&[u_1(t,x)^2]=\EE[I_1(f_1)^2]\\
&=\int_{[0,t]^2}\int_{\RR}e^{-\frac12 (t-s_1+t-r_1)|\xi_1|^2}\times\prod_{k=1}^d|\xi_{1k}|^{1-2H_k}\gamma_0(s_1-r_1)d\xi_1ds_1dr_1\\
&=\int_{[0,t]^2}\prod_{k=1}^dh_k(s_1,r_1)\gamma_0(s_1-r_1)ds_1dr_1\,, 
\end{align*}
where $h_k(s_1,r_1)=\int_{\RR}e^{-\frac12 (t-s_1+t-r_1)|\xi_{1k}|^2}\times |\xi_{1k}|^{1-2H_k}d\xi_{1k}.$ It is not difficult to verify that $$h_k(s_1,r_1)\geq C(t-s_1+t-r_1)^{H_k-1}.$$
Thus, taking $\widetilde{s}_1=t-s_1$, $\widetilde{r}_1=t-r_1$,  $\widetilde{v}=\widetilde{r}_1-\widetilde{s}_1$ and $ \widetilde{w}=\widetilde{s}_1+\widetilde{r}_1$
\begin{align*}
\EE[u_1(t,x)^2]&\geq C \int_{[0,t]^2}(\widetilde{s}_1+\widetilde{r}_1)^{|H|-d} \cdot|\widetilde{s}_1-\widetilde{r}_1|^{2H_0-2}ds_1dr_1\\
&\geq C \int_{0}^t\int_0^{\widetilde{w}}|\widetilde{w}|^{|H|-d}\cdot|\widetilde{v}|^{2H_0-2}d\widetilde{v}d\widetilde{w}+\int_t^{2t}\int_0^{2t-\widetilde{w}}|\widetilde{w}|^{|H|-d}\cdot|\widetilde{v}|^{2H_0-2}d\widetilde{v}d\widetilde{w}\\
&\geq C\int_{0}^t|\widetilde{w}|^{|H|-d+2H_0-1}d\widetilde{w}.
\end{align*}
Accordingly, $\EE[u_1(t,x)^2]<+\infty$ happens only if
\[
|H|-d+2H_0-1>-1\Leftrightarrow |H|+2H_0>d.
\]

\medskip
\noindent \textbf{The necessity of conditions \eqref{cond.0<H<1_c} and \eqref{cond.0<H<1_d}.}
We shall show the necessity by means of the finiteness of  $\EE[u_2(t,x)^2]$ {for any $(t,x)\in\RR_+\times\RR^d$}.

%
{By definition of $h_{k,n}(\vbfs,\vbfr)$ with $n=2$ in \eqref{hk2},} it follows  that 
\begin{align*}
h_{k,2}(\vbfs,\vbfr)&=c_{H_k}^n
 \lt(\prod_{i=1}^2 \widetilde w_i^{-1/2}\rt)
 \EE \lt[ \prod_{i=1}^2 \left|\frac{X_i}{\sqrt{\widetilde w_i}}-\frac{X_{i-1}}{\sqrt{\widetilde w_{i-1}}} \right|^{1-2H_k}\rt] \,,\\
&= c_{H_k}^n \widetilde w_2^{H_k-1}
 \widetilde w_1^{2H_k-3/2} (\widetilde w_2+\widetilde w_1)^{\frac12- H_k} \nonumber\\
&\qquad \quad  \times   \EE\left[ \left|  X_1  \right| ^{1-2H_k}
\left| \sqrt{\frac{\widetilde w_{1}}{\widetilde w_{1}+\widetilde w_2} }X_2 - \sqrt{\frac{\widetilde w_{2 }}{\widetilde w_{1}+\widetilde w_2} }X_1 \right|^{1-2H_k}
\right]
\end{align*}
where $\widetilde w_1= s_2-s_1+r_2-r_1$ and $\widetilde w_2=t-s_2+t-r_2$.
{Recalling $\cJ_{k,2}(\widetilde\la_1,\widetilde\la_2)$ defined in \eqref{I_k} with $n=2$, and letting $X_0=0$, $\widetilde\la_1=1$ and $\widetilde\la_2:=\sqrt{\frac{\widetilde w_{1}}{\widetilde w_{1}+\widetilde w_2} }$, one has 
\begin{align*}
\cJ_{k,2}(\widetilde\la_1,\widetilde\la_2)&= \EE\left[ \prod_{i=1}^{2}
\left|{\widetilde\la_i }X_i-\sqrt{1-\widetilde\la_i^2 }X_{i-1} \right|^{1-2H_k}\right]\\
&=\EE\left[ \left|  X_1  \right|^{1-2H_k}
\left| \sqrt{\frac{\widetilde w_{1}}{\widetilde w_{1}+\widetilde w_2} }X_2 - \sqrt{\frac{\widetilde w_{2 }}{\widetilde w_{1}+\widetilde w_2} }X_1 \right|^{1-2H_k}
\right]. 
\end{align*}
It follows from \eqref{e.3.21lower} that
\[\cJ_{k,2}(\widetilde\la_1,\widetilde\la_2)\ge c_{H_k}\widetilde\la_2^{-(4H_k-3)}.\]
 }
{Since $2H_*-\frac32 d_*<0$ and $(\tilde w_2+\tilde w_1)^{|H|-d}>\tilde w_1^{2H_*-\frac32 d_*}(\tilde w_2+\tilde w_1)^{|H|-d-2H_*+\frac32 d_*}$, we have
\begin{align}\label{lower_u_2}
&\EE[u_2(t,x)^2]\geq c_{H_k} \int_{0<s_1<s_2<t\atop 0< r_1<r_2<t}
 \prod_{k=1}^d h_{k,2} (\vbfs,\vbfr)
\prod_{i=1}^2 \gamma_0(s_i-r_i) ds_1dr_1ds_2dr_2\nonumber\\
&=c_{H_k}\int_{0<s_1<s_2<t\atop 0< r_1<r_2<t}
  |t-s_2+t-r_2|^{|H|-d}|t-s_1+t-r_1|^{|H|-d}
\prod_{i=1}^2 \gamma_0(s_i-r_i) ds_idr_i\nonumber\\
&\geq c_{H_k}\int_{0<s_1<s_2<t\atop 0<r_1<r_2<t}|t-s_2+t-r_2|^{|H|-d}|s_2-s_1+r_2-r_1|^{2H_*-\frac32 d_*}\nonumber\\
&\quad\times|t-s_1+t-r_1|^{|H|-2H_*-d+\frac32 d_*}\times|s_1-r_1|^{2H_0-2}|s_2-r_2|^{2H_0-2}ds_1dr_1ds_2dr_2.
\end{align}
}
%

Denote the { integral domain $\mathbf{I_1}:=\lt\{(s_1,r_1,s_2,r_2):0<s_1<r_1<\frac t2<s_2<r_2<t\rt\}$}. Taking $\tilde{v}= {r}_1- {s}_1$, $\tilde{x}=s_2-r_1$, $\tilde{y}=r_2-s_2$ and $\tilde{z}=t-r_2$
yields
\begin{align*}
\EE[u_2(t,x)^2]&\geq c_{H} \int_{\mathbf{I_1}}|t-s_2+t-r_2|^{|H|-d}|s_2-s_1+r_2-r_1|^{2H_*-\frac32 d_*}\nonumber\\
&\qquad\qquad\qquad\times|s_1-r_1|^{2H_0-2}|s_2-r_2|^{2H_0-2}ds_1dr_1ds_2dr_2\nonumber\\
&\geq c_{H} \int_{0<\tilde{v},\tilde{x},\tilde{y},\tilde{z}<\frac t2}|\tilde{y}+2\tilde{z}|^{|H|-d}|\tilde{v}+2\tilde{x}+\tilde{y}|^{2H_*-\frac32 d_*}\nonumber\\
&\qquad\qquad\qquad\qquad\qquad\times|\tilde{v}|^{2H_0-2}|\tilde{y}|^{2H_0-2}d\tilde{v}d\tilde{x}d\tilde{y}d\tilde{z}\nonumber\\
&\geq c_{H} \int_{0<\widetilde{x},\widetilde{y},\widetilde{z}<\frac t2}|\tilde{y}+2\tilde{z}|^{|H|-d}|2\tilde{x}+\tilde{y}|^{2H_*-\frac32 d_*+2H_0-1}|\tilde{y}|^{2H_0-2}d\tilde{x}d\tilde{y}d\tilde{z},
\end{align*}
where in the last inequality, we used  Lemma \ref{l.3.2}.} {By integrating $\tilde z$, we have
\begin{align}\label{lower1_u_2}
\EE[u_2(t,x)^2]&\geq c_{H} \int_{0<\widetilde{x},\widetilde{y}<\frac t2}\lt(|\tilde{y}+t|^{|H|-d+1}-|\tilde y|^{|H|-d+1}\rt) \nonumber \\
&\qquad\quad \cdot |2\tilde{x}+\tilde{y}|^{2H_*-\frac32 d_*+2H_0-1}|\tilde{y}|^{2H_0-2}d\tilde{x}d\tilde{y}.
\end{align} 
From the above, we can see that for the two cases $|H|- d+1> 0$ and $|H|- d+1<0$, the terms that really matter are different. Hence, to study the boundedness of the above integral, we shall 
discuss the two cases $|H|- d+1> 0$ and $|H|- d+1<0$
separately.

{When $|H|- d+1> 0$, it holds that $|\tilde{y}+t|^{|H|-d+1}\geq 
	|\tilde{y}+2\tilde y|^{|H|-d+1}\geq 3^ {|H|-d+1} |\tilde y|^{|H|-d+1}$. So there exists a constant  which  may depend  on $H$: $c_H\in(0,1)$ such that $|\tilde{y}+t|^{|H|-d+1}- |\tilde y|^{|H|-d+1}\geq c_{H}|\tilde{y}+t|^{|H|-d+1}$. Then from \eqref{lower1_u_2} we have}
\begin{align*}
\EE[u_2(t,x)^2]&\geq c_{H} \int_{0<\widetilde{x},\widetilde{y}<\frac t2}|\tilde{y}+t|^{|H|-d+1}|2\tilde{x}+\tilde{y}|^{2H_*-\frac32 d_*+2H_0-1}|\tilde{y}|^{2H_0-2}d\tilde{x}d\tilde{y}\\
&\geq c_{H} t^{|H|-d+1} \int_{0<\widetilde{x},\widetilde{y}<\frac t2}|2\tilde{x}+\tilde{y}|^{2H_*-\frac32 d_*+2H_0-1}|\tilde{y}|^{2H_0-2}d\tilde{x}d\tilde{y}\\
&\geq c_{H} \int_{0<\widetilde{x}<\frac t2}|\tilde{x}|^{2H_*-\frac32 d_*+4H_0-2}d\tilde{x},
\end{align*}
where the last inequality is due to Lemma \ref{l.3.2}. Thus, $\EE[u_2(t,x)^2]<+\infty$ happens in this case only if
\[
2H_*-\frac32 d_*+4H_0-2>-1\Leftrightarrow H_*+2H_0>\frac{3d_*}{4}+\frac12.
\]
This implies the third inequality  in \eqref{cond.0<H<1}.
We add  $|H|\geq d-1$ and $2H_*+4H_0>\frac{3d_*}{2}+1$  to obtain  $|H|+2H_*+4H_0>d+\frac32 d_*$.  This is the fourth
condition in \eqref{cond.0<H<1}.

{When $|H|- d+1< 0$, it holds that $|\tilde{y}+t|^{|H|-d+1}\le 3^ {|H|-d+1} |\tilde y|^{|H|-d+1}$. So there exists  a constant   depending  on $H$: $c_H\in(0,1)$ such that $|\tilde y|^{|H|-d+1}-|\tilde{y}+t|^{|H|-d+1}\geq c_H|\tilde{y}+t|^{|H|-d+1}$. Then from \eqref{lower1_u_2} and  Lemma \ref{l.3.2}} it follows 
\begin{align*}
\EE[u_2(t,x)^2]&\geq c_{H}\int_{0<\widetilde{x},\widetilde{y}<\frac t2}|\tilde{y}|^{|H|-d+1}|2\tilde{x}+\tilde{y}|^{2H_*-\frac32 d_*+2H_0-1}|\tilde{y}|^{2H_0-2}d\tilde{x}d\tilde{y}\\
&\geq c_{H} \int_{0<\widetilde{x}<\frac t2}|\tilde{x}|^{|H|-d+1+2H_*-\frac32 d_*+4H_0-2}d\tilde{x}.
\end{align*}
Therefore, $\EE[u_2(t,x)^2]<+\infty$ happens in this case $|H|<d-1$ only if
\[
|H|-d+1+2H_*-\frac32 d_*+4H_0-2>-1\Leftrightarrow |H|+2H_*+4H_0>d+\frac32 d_*.
\]
Subtracting $|H|<d-1$ from $|H|+2H_*+4H_0>d+\frac32 d_*$  implies  that $H_*+2H_0>\frac{3d_*}{4}+\frac12$. This is the third condition in  \eqref{cond.0<H<1}.


\medskip
\noindent \textbf{The necessity of condition \eqref{cond.0<H<1_a}.}
   We shall still use  \eqref{lower1_u_2}:
   \begin{align*}
\EE[u_2(t,x)^2]&\geq c_{H} \int_{\mathbf{I_1}}|t-s_2+t-r_2|^{|H|-d}|s_2-s_1+r_2-r_1|^{2H_*-\frac32 d_*}\nonumber\\
&\qquad\qquad\qquad\times|s_1-r_1|^{2H_0-2}|s_2-r_2|^{2H_0-2}ds_1dr_1ds_2dr_2.
\end{align*}
{Assuming  $H_*\leq (\frac34 d_*-1)$, we shall derive that $\EE[u_2^2(t,x)]=\infty$, which contradicts with $\EE[u_n(t,x)^2  ]<\infty$ when $n\geq 1$.}  Let us consider the {integral domain $\mathbf{I_2}=\{\frac t4<r_1<r_2<\frac t3<\frac34 t<s_1<s_2<\frac56 t\}$}.  We have
\begin{align*}
\EE[u_2(t,x)^2]&\geq c_{H}\int_{\mathbf{I_1}\cap\mathbf{I_2}}|s_2-s_1+r_2-r_1|^{2H_*-\frac32 d_*}ds_1dr_1ds_2dr_2.
\end{align*}
Integrating $r_1$ first and then integrating $r_2$ yield
\begin{align*}
\int_{\frac t4<r_1<r_2<\frac t3}& |s_2-s_1+r_2-r_1|^{2H_*-\frac32 d_*}dr_1dr_2\\
=&-\frac{t}{12}(2H_*-\frac32 d+1)^{-1}|s_2-s_1|^{2H_*-\frac32 d_*+1}\\
&+(2H_*-\frac32 d_*+1)^{-1}(2 H_*-\frac32 d_*+2)^{-1}|s_2-s_1+\frac{t}{12}|^{2H_*-\frac32 d_*+2}\\
&-(2H_*-\frac32 d_*+1)^{-1}(2H_*-\frac32 d_*+2)^{-1}|s_2-s_1|^{2H_*-\frac32 d_*+2}\\
\ge &-\frac{t}{12}(2H_*-\frac32 d+1)^{-1}|s_2-s_1|^{2H_*-\frac32 d_*+1}  \\
&-(2H_*-\frac32 d_*+1)^{-1}(2H_*-\frac32 d_*+2)^{-1}|s_2-s_1|^{2H_*-\frac32 d_*+2} \,,
\end{align*}
where the inequality holds because we drop a positive term.  
Since $H_*\leq (\frac34 d_*-1)$, we have $2H_*-\frac32 d_*+1\leq-1$. Thus,  we see
\begin{equation*}
\int_{\frac34 t<s_1<s_2<\frac56 t}|s_2-s_1|^{2H_*-\frac32 d_*+1}ds_1ds_2= \infty,
\end{equation*}
which contradicts  with $\EE[u_n(t,x)^2  ]<\infty$ ($n\geq 1$). This says that  the condition $H_*>(\frac34 d_*-1)$ is necessary.

 When $H_0=1/2$, the lower bound of $\EE[u_2(t,x)^2]$ in \eqref{lower_u_2} becomes
\begin{align}\label{lowH0=1/2}
\EE[u_2(t,x)^2]
&\geq c_H\int_{0<s_1<s_2<t}|t-s_2|^{|H|-d}|s_2-s_1|^{2H_*-\frac32 d_*}\nonumber\\
&\qquad\qquad\qquad\qquad\times|t-s_1|^{|H|-2H_*-d+\frac32 d_*}ds_1ds_2\nonumber\\
&=c_H \int_{0<\widetilde{x},\widetilde{z}<\frac t2}|\tilde{z}|^{|H|-d}|\tilde{x}|^{2H_*-\frac32 d_*}|\tilde{x}+\tilde{z}|^{|H|-2H_*-d+\frac32 d_*}d\tilde{x}d\tilde{z}.
\end{align}
The right-hand side of  \eqref{lowH0=1/2} is integrable only if
\begin{equation*}
\begin{cases}
|H|-d>-1,\\
2H_*-\frac32 d_*>-1,\\
|H|-2H_*-d+\frac32 d_*>-1,
\end{cases}
\end{equation*}
which amounts to
\begin{equation*}
\begin{cases}
|H|-d>-1,\\
H_*>\frac34 d_*-\frac12.
\end{cases}
\end{equation*}
Moreover, applying Lemma \ref{l.3.2} to \eqref{lowH0=1/2}, we have
\begin{align*}
\EE[u_2(t,x)^2]
&\geq c_H \int_{0<\widetilde{z}<\frac t2}|\tilde{z}|^{|H|-d}|\tilde{z}|^{|H|-d+1}d\tilde{z}.
\end{align*}
Thus, $\EE[u_2(t,x)^2]<+\infty$ implies the integral on the right-hand side of the above inequality is finite, which happens only if
\[|H|-d+|H|-d+1>-1\Leftrightarrow |H|>d-1.\]
Therefore, when $H_0=1/2$, { we proved that} the sufficient condition \eqref{H0=1/2} is also necessary.

As a result, we prove that all the conditions in \eqref{cond.0<H<1} are necessary for $\EE[u_n(t,x)^2]<\infty$ for fixed $n> 1$ and $(t,x)\in\RR_+\times\RR^d$. This completes the proof of Theorem \ref{main_result1}.

\section{Convergence of the It\^o-Wiener chaos expansion}\label{Sec5.Conv}


In this section, we prove Theorem \ref{thm_conver}. The result is visualized in Figure \ref{Fig.Improve}. More precisely, the existing condition $H+H_0>\frac 34$ (yellow shaded region in the left figure)  is expanded to encompass the yellow shaded region in the right figure.

\begin{figure}[H]
	\centering
	\includegraphics[width=0.8\textwidth]{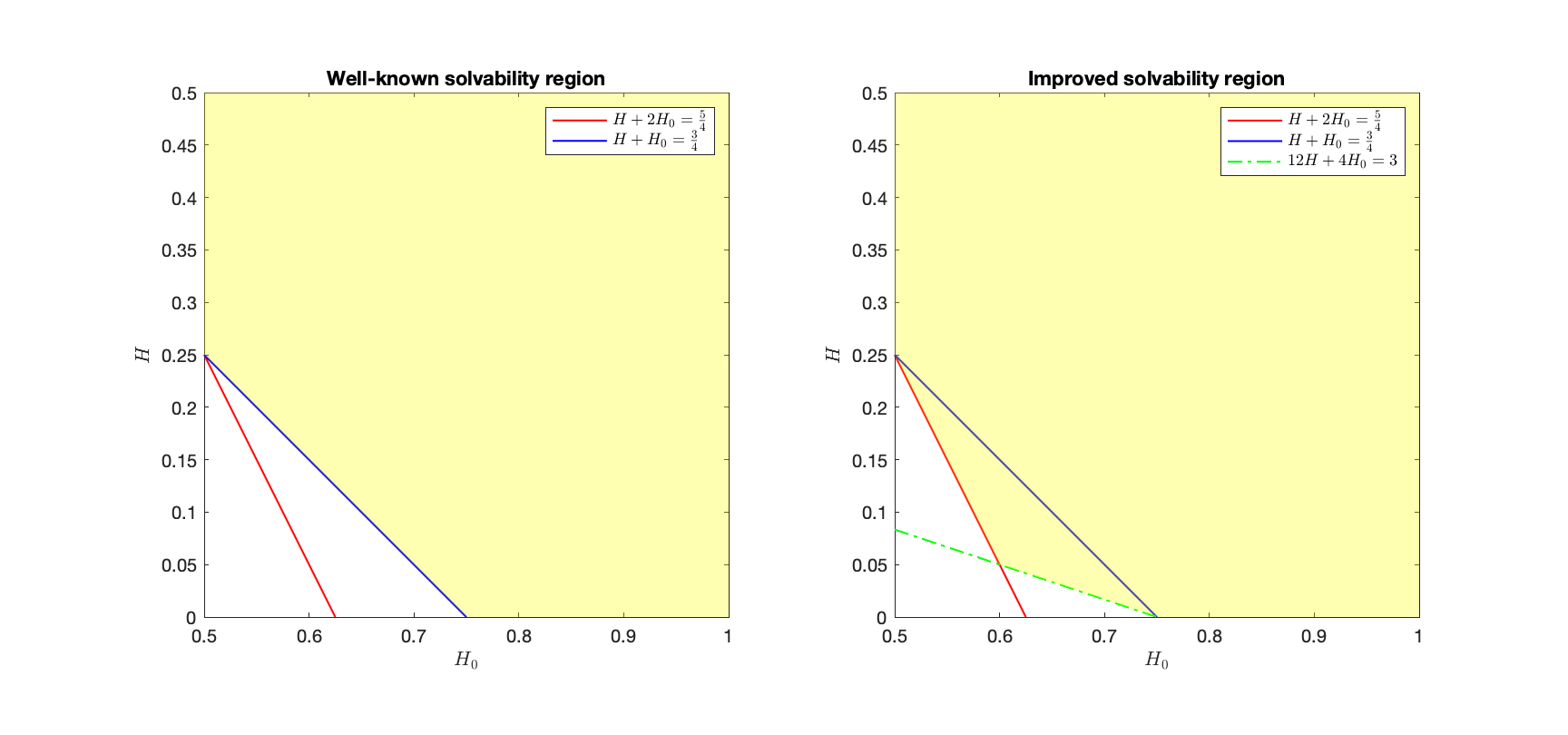}
	\caption{Solvability regions}
	\label{Fig.Improve}
\end{figure}
\begin{proof}[Proof of Theorem \ref{thm_conver}]
For any    permutation   $\sigma$   of $\{1, \cdots, n \}$, recall the notation  $\bT_1^{\sigma}$ given by \eqref{Tsr_sig} and
\begin{align}
	\widehat{f}_{n,\sigma}^{(t,x)}(\vbfs,\bm{\eta}):=\widehat f_{n,\sigma}^{(t,x)}(s_1,\xi_1,\cdots,s_n,\xi_n) &=\prod_{i=1}^n e^{-\frac 12(s_{\sigma(i+1)}-s_{\sigma(i)})|\eta_i^{\sigma}|^2}e^{-\iota x \eta_n^{\sigma}} \1_{\sigma}(\vbfs) \,, \label{Def:hatf_n}
\end{align}
with the notations $\bm{\eta}:=(\eta_1,\cdots,\eta_n)$, $\1_{\sigma}(\vbfs):={\1_{\{\vbfs\in\bT_1^{\sigma}\}}}=\1_{\{0<s_{\sigma(1)}<\cdots<s_{\sigma(n)}<1\}}(\vbfs)$, $\eta_i^{\sigma}:=\xi_{\sigma(1)}+\cdots+\xi_{\sigma(i)}$. {And we set $\widehat f_{n}^{(t,x)}(\cdot):=\widehat f_{n,\sigma}^{(t,x)}(\cdot)$ if $\sigma$ is the natural permutation.} The notations involved  with the variable $\vbfr$ are defined similarly.

{ By applying change of variables $\eta_j=\xi_1+\cdots+\xi_{j}$ and triangle inequality, we have by expanding the   product
\begin{align*}
	\prod_{j=1}^n|\xi_j|^{1-2H}&=\prod_{j=1}^n|\eta_j-\eta_{j-1}|^{1-2H} \\
	&\leq \prod_{j=1}^n (|\eta_j|+|\eta_{j-1}|)^{1-2H} \leq C_H \sum_{\bm{\pi}} \prod_{j=1}^n|\eta_j|^{(1-2H)\pi_j}\,, 
\end{align*}
where $\bm{\pi}=(\pi_1,\cdots,\pi_n)$ with $\pi_j\in\{0,1,2\}$ and $|\bm{\pi}|:=\sum_{i=1}^n \pi_i=n-1$ can be thought as one dimensional  simplified version of \eqref{Def:al_i}}. Then, it is not hard  to see that by the change of variables described before
\begin{align*}
I_{n }:= 	&\|\frac{1}{n!} \sum_{\sigma} f_{n,\sigma}^{(t, x)}(\cdot)\|^2_{\HH^{\otimes n}}\nonumber \\
	=&\frac{1}{(n!)^2} \int_{\RR_+^{2n}} \int_{\RR^{n}} \sum_{\sigma} \widehat{f}_{n,\sigma}^{(t,x)}(\vbfr,\bm{\xi})\cdot \sum_{\varsigma}\overline{\widehat{f}_{n,\varsigma}^{(t,x)}(\vbfs,\bm{\xi})} \prod_{j=1}^n|\xi_j|^{1-2H} \prod_{i=1}^n|s_i-r_i|^{-\gamma_0} d\bm{\xi}d\vec{\bfs} d\vec{\bfr} \nonumber\\
	\leq& C_H \sum_{\bm{\pi}} \int_{\RR_+^{2n}} \int_{\RR^{n}} \frac{1}{n!}\sum_{\sigma} \widehat{f}_{n,\sigma}^{(t,x)}(\vbfr,\bm{\eta})\cdot \frac{1}{n!}\sum_{\varsigma}\overline{\widehat{f}_{n,\varsigma}^{(t,x)}(\vbfs,\bm{\eta})} \\
	&\qquad\qquad\qquad\qquad\qquad\qquad\qquad\qquad \times \prod_{j=1}^n|\eta_j|^{(1-2H)\pi_j} \prod_{i=1}^n|s_i-r_i|^{-\gamma_0} d\bm{\eta}d\vec{\bfs} d\vec{\bfr} \,,
\end{align*}
where $d\bm{\eta}=d\eta_1\cdots d\eta_n$ and $\sigma, \varsigma$ are two permutations. 
Let us set
  \begin{align*}
	\psi_{n } (\vbfr,\vbfs;\bm{\eta})
	:=   \frac{1}{n!} \sum_{\sigma} \widehat{f}_{n,\sigma}^{(t,x)}(\vbfr,\bm{\eta})\cdot \frac{1}{n!}\sum_{\varsigma}\overline{\widehat{f}_{n,\varsigma}^{(t,x)}(\vbfs,\bm{\eta})}
\end{align*}
and
\begin{align*}
	\Psi_{n,\bm{\tau}}(\vbfr,\vbfs):=&\int_{\RR^{n}}
	\psi_{n } (\vbfr,\vbfs;\bm{\eta}) \prod_{j=1}^n|\eta_j|^{\tau_j} d\bm{\eta}
\end{align*}
with indices $\bm{\tau}=(\tau_1,\cdots,\tau_n)$ { depending on $\bm{\pi}$} to be chosen later.    By H\"older's  inequality with $\frac{1}{p}+\frac{1}{q}=1$, $k+k'=1$ and $m+m'=1$, we have
\begin{align}\label{def_j1j2}
I_n 	\leq& C_H\sum_{\bm{\pi}}\bigg|\int_{\RR_+^{2n}}\Psi_{n,\bm{\tau}}(\vbfr,\vbfs)\prod_{i=1}^n|s_i-r_i|^{-pm\gamma_0} d\vec{\bfs} d\vec{\bfr} \bigg|^{\frac{1}{p}}\nonumber\\
&\qquad \times\bigg|\int_{\RR_+^{2n}}\Psi_{n,\bm{\tau'}}(\vbfr,\vbfs)\prod_{i=1}^n|s_i-r_i|^{-qm'\gamma_0} d\vec{\bfs} d\vec{\bfr} \bigg|^{\frac{1}{q}}\nonumber\\
	=:&C_H {\sum_{\bm{\pi}} \cJ_1(p,m,k,\bm{\pi})\times \cJ_2(q,m',k',\bm{\pi})}\,,
\end{align}
with { $\cJ_1, \cJ_2$ depend on the $\bm{\tau}$, which is determined by $\bm{\pi}$ as given by
\begin{align*}
	\tau_{j}=kp(1-2H)\pi_j\quad \text{and}\quad \tau'_{j}=k'q(1-2H)\pi_j
\end{align*}
depending on $p,q,k,k'$ will be chosen later.} 
We will apply the combination of the H\"older-Young-Brascamp-Lieb inequality and the Hardy-Littlewood-Sobolev inequality to {$\cJ_1(p,m,k,\bm{\pi})$} as in Section \ref{sufficient}, and will apply solely the Hardy-Littlewood-Sobolev inequality to {$\cJ_2(q,m',k',\bm{\pi})$}. 
{ As we noticed in Section \ref{sufficient}, the condition of using H\"older-Young-Brascamp-Lieb inequality is weaker than the one of using Hardy-Littlewood-Sobolev. But the bound for $I_n$ obtained by applying the combination of the H\"older-Young-Brascamp-Lieb inequality and the Hardy-Littlewood-Sobolev inequality is 
less precise in terms of $n$, 
which   ensures  the finiteness under more general condition. The application of 
Hardy-Littlewood-Sobolev inequality to the multiple integral, on the other hand,   will produce  a more explicit decay factor of the form  $\frac{1}{(n!)^\kappa}$ with certain $\kappa>0$,   which is key for guaranteeing the convergence of \eqref{Fin_Result}.}

\bigskip
\noindent\textbf{Step 1: Estimate $\cJ_1(p,m,k,\bm{\pi})$.} We shall use the H\"older-Young-Brascamp-Lieb  inequality to show {$\cJ_1:=\cJ_1(p,m,k,\bm{\pi})\leq C^n$}.  {First, we have
\begin{align*}
	\psi_{n }(\vbfr,\vbfs;\bm{\eta})
	\leq& \frac{1}{(n!)^2}\sum_{\sigma,\varsigma} |\widehat{f}_{n,\sigma}^{(t,x)}(\vbfr,\bm{\eta})\cdot\overline{\widehat{f}_{n,\varsigma}^{(t,x)}(\vbfs,\bm{\eta})}|\,,
\end{align*}
for any permutations (of $\{1,\cdots,n\}$) $\sigma, \varsigma$ and consequently we get
\begin{align}
	|\cJ_1|^p =& \int_{\RR_+^{2n}}\Psi_{n,\bm{\tau}}(\vbfr,\vbfs)\prod_{j=1}^n|s_j-r_j|^{-pm\gamma_0} d\vec{\bfs} d\vec{\bfr} \nonumber  \\
	\leq& \frac{1}{(n!)^2}\sum_{\sigma,\varsigma}\int_{\RR_+^{2n}}\int_{\RR^{n}}|\widehat{f}_{n,\sigma}^{(t,x)}(\vbfr,\bm{\eta})\cdot\overline{\widehat{f}_{n,\varsigma}^{(t,x)}(\vbfs,\bm{\eta})}|  \prod_{j=1}^n|\eta_j|^{\tau_{j}} |s_j-r_j|^{-pm\gamma_0}d\bm{\eta} d\vec{\bfs} d\vec{\bfr}\,. \label{Ineq:J1^p}
\end{align}
}
An application  of { the Cauchy-Schwarz inequality} to \eqref{Ineq:J1^p}  yields
\begin{align}
	|\cJ_1|^p \leq& \frac{1}{(n!)^2}\sum_{\sigma,\varsigma} \bigg(\int_{\RR_+^{2n}} \int_{\RR^{n}}  |\widehat{f}_{n,\sigma}^{(t,x)}(\vbfr,\bm{\eta}) \overline{\widehat{f}_{n,\sigma}^{(t,x)}(\vbfs,\bm{\eta})}| \prod_{j=1}^n|\eta_j|^{\tau_{j}} |s_j-r_j|^{-pm\gamma_0} d\bm{\eta}d\vec{\bfs} d\vec{\bfr} \bigg)^{\frac 12} \nonumber\\
	&\qquad\quad  \times \bigg(\int_{\RR_+^{2n}} \int_{\RR^{n}}  |\widehat{f}_{n,\varsigma}^{(t,x)}(\vbfr,\bm{\eta}) \overline{\widehat{f}_{n,\varsigma}^{(t,x)}(\vbfs,\bm{\eta})}| \prod_{j=1}^n|\eta_j|^{\tau_{j}} |s_j-r_j|^{-pm\gamma_0} d\bm{\eta}d\vec{\bfs} d\vec{\bfr} \bigg)^{\frac 12}     \nonumber\\
	=& \frac{1}{(n!)^2}\sum_{\sigma,\varsigma}\int_{\RR_+^{2n}} \int_{\RR^{n}}  |\widehat{f}_{n}^{(t,x)}(\vbfr,\bm{\eta}) \overline{\widehat{f}_{n}^{(t,x)}(\vbfs,\bm{\eta})}| \prod_{j=1}^n|\eta_j|^{\tau_{j}} |s_j-r_j|^{-pm\gamma_0} d\bm{\eta}d\vec{\bfs} d\vec{\bfr} \,,  \nonumber
\end{align}
where the equality holds since the integrals in line 1 and line 2 are actually independent of permutations $\sigma,\varsigma$. {We introduce the following notations to simplify the presentation:
\begin{equation*}
	\varrho_{j}:=\frac{1+\tau_{j}}{2}=\begin{cases}
		\frac 12\,, &\pi_j=0\,;\\
		\frac 12+\frac{kp}{2}(1-2H)\,, &\pi_j=1\,;\\
		\frac 12+kp(1-2H)\,, &\pi_j=2\,.
	\end{cases}
\end{equation*}
Then noticing the definition of $\widehat{f}_{n}^{(t,x)}(\vbfr,\bm{\eta})$ in \eqref{Def:hatf_n} with $\sigma$ being natural permutation, the integration with respect to $\eta$  gives 
\begin{align}
	|\cJ_1|^p \leq&C^n \int_{{\bT_1\times\bT_1}} \int_{\RR^{n}} \prod_{j=1}^n e^{-\frac 12(s_{j+1}-s_{j}+r_{j+1}-r_{j})|\eta_j|^2}\prod_{j=1}^n|\eta_j|^{\tau_{j}} |s_j-r_j|^{-pm\gamma_0} d\bm{\eta}d\vec{\bfs} d\vec{\bfr} \nonumber\\
	\leq&   C^n   \int_{{\bT_1\times\bT_1}} \int_{\RR^{n}} \prod_{j=1}^{n} |s_{j+1}-s_j+r_{j+1}-r_{j}|^{-\frac{\varrho_{j}}{2}}\times|s_j-r_j|^{-pm\gamma_0} d\bm{\eta}d\vec{\bfs} d\vec{\bfr} \nonumber\\
	\leq&C^n \int_{{\bT_1\times\bT_1}}\prod_{j=1}^{n} |s_{j+1}-s_j|^{-\frac{\varrho_{j}}{2}}|r_{j+1}-r_{j}|^{-\frac{\varrho_{j}}{2}}\times |s_j-r_{j}|^{-pm\gamma_0} d\vec{\bfs} d\vec{\bfr}\leq C^n  \,,\label{Ineq:J1^p_2}
\end{align}}
under the \emph{dimension conditions} of Theorem \ref{Thm.HYBL_N}.

Similarly to the argument in the proof of Theorem \ref{main_result1},   the \emph{dimension conditions} will hold if the following conditions are satisfied with $H_0^{(1)}:=1-pm(1-H_0)$:
\begin{align}
	&2\Big[ \frac{1}{4}+\frac{kp}{2}(1-2H) \Big]+2(2-2H^{(1)}_0)<3 \quad \text{and} \quad pm\gamma_0<1 \nonumber
\end{align}
which is equivalent to
\begin{align}\label{condi_p<}
	   1< p<\frac{5}{2k(1-2H)+8m(1-H_0)} \quad \text{and} \quad  1< p<\frac{1}{2m(1-H_0)}\,.
\end{align}


\textbf{Step 2: Estimate $\cJ_2(q,m',k',\bm{\pi})$.} We shall use the Hardy-Littlewood-Sobolev inequality solely to show 
\begin{equation*}
	{\cJ_2:=\cJ_2(q,m',k',\bm{\pi})\leq C^n (n!)^{-\frac 2q+\frac{1}{2q}+\frac{k'}{2}(1-2H)}\,,}
\end{equation*}
with $\cJ_2(q,m',k',\bm{\pi})$ defined by \eqref{def_j1j2}.
By the Cauchy-Schwarz inequality, we have
\begin{equation*}
	\Psi_{n,\bm{\tau'}}(\vbfr,\vbfs)\leq \sqrt{\Psi_{n,\bm{\tau'}}(\vbfr,\vbfr)} \sqrt{\Psi_{n,\bm{\tau'}}(\vbfs,\vbfs)}\,.
\end{equation*}
We substitute this into $\cJ_2$ and then apply   the Hardy-Littlewood-Sobolev inequality to  obtain
\begin{align*}
	\cJ_2=&\bigg(\int_{\RR_+^{2n}}\Psi_{n,\bm{\tau'}}(\vbfr,\vbfs)\prod_{i=1}^n|s_i-r_i|^{-qm'\gamma_0} d\vec{\bfs} d\vec{\bfr} \bigg)^{\frac{1}{q}}\\
	\leq& \bigg(\int_{\RR_+^{2n}}\sqrt{\Psi_{n,\bm{\tau'}}(\vbfr,\vbfr)} \sqrt{\Psi_{n,\bm{\tau'}}(\vbfs,\vbfs)}\prod_{i=1}^n|s_i-r_i|^{-qm'\gamma_0} d\vec{\bfs} d\vec{\bfr} \bigg)^{\frac{1}{q}}\\
	\leq& C^n \bigg(\int_{\RR_+^{n}} \Big| \Psi_{n,\bm{\tau'}}(\vbfs,\vbfs)\Big|^{\frac{1}{2H^{(2)}_0}} d\vec{\bfs} \bigg)^{\frac{2H^{(2)}_0}{q}}\,,
\end{align*}
where $H_0^{(2)}:=1-qm'(1-H_0)$ and we should require $qm'\gamma_0<1$.  Recall that $\ga_0=2-2H_0$, then
\begin{equation}\label{Rmk.condi_q<}
	qm'\gamma_0<1\Leftrightarrow 1<  q<\frac{1}{2m'(1-H_0)}\,.
\end{equation}
Notice that
\begin{align}\label{est_phi1}
	\Psi_{n,\bm{\tau'}}(\vbfs,\vbfs)=&\int_{\RR^{n}} \bigg| \frac{1}{n!} \sum_{\sigma} \widehat{f}_{n,\sigma}^{(t,x)}(\vbfs,\bm{\eta})\cdot \frac{1}{n!}\sum_{\varsigma}\overline{\widehat{f}_{n,\varsigma}^{(t,x)}(\vbfs,\bm{\eta})}\bigg| \prod_{j=1}^n|\eta_j|^{\tau'_{j}} d\bm{\eta}\nonumber \\
	=&\int_{\RR^{n}} \bigg|\frac{1}{(n!)^2}\sum_{\sigma,\varsigma}\widehat{f}_{n,\sigma}^{(t,x)}(\vbfs,\bm{\eta})\cdot \overline{\widehat{f}_{n,\varsigma}^{(t,x)}(\vbfs,\bm{\eta})}\bigg| \prod_{j=1}^n|\eta_j|^{\tau'_{j}} d\bm{\eta} \nonumber\\
	=&\int_{\RR^{n}} \frac{1}{(n!)^2}\sum_{\sigma}|\widehat{f}_{n,\sigma}^{(t,x)}(\vbfs,\bm{\eta})|^2 \prod_{j=1}^n|\eta_j|^{\tau'_{j}} d\bm{\eta}\,,
\end{align}
where the last equality follows from
\begin{equation*}
	\1_{\sigma}(\vbfs)\cdot \1_{\varsigma}(\vbfs)=\begin{cases}
		0\,, &\sigma\neq\varsigma\,;\\
		\1_{\sigma}(\vbfs)\,, &\sigma=\varsigma\,.
	\end{cases}
\end{equation*}
Hence, we have by \eqref{est_phi1}
\begin{align}
	\int_{\RR_+^{n}} \Big| \Psi_{n,\bm{\tau'}}(\vbfs,\vbfs)\Big|^{\frac{1}{2H^{(2)}_0}} d\vec{\bfs}
	=&\frac{1}{(n!)^{\frac{1}{H^{(2)}_0}}}\int_{\RR_+^{n}}\bigg| \int_{\RR^{n}} \sum_{\sigma}|\widehat{f}_{n,\sigma}^{(t,x)}(\vbfs,\bm{\eta})|^{2} \prod_{j=1}^n|\eta_j|^{\tau'_{j}} d\bm{\eta} \bigg|^{\frac{1}{2H^{(2)}_0}} d\vec{\bfs}\nonumber \\
	=&\frac{1}{(n!)^{\frac{1}{H^{(2)}_0}}}\sum_{\sigma'}\int_{{\bT_1^{\sigma' }}}  \bigg|\int_{\RR^{n}} \sum_{\sigma}|\widehat{f}_{n,\sigma}^{(t,x)}(\vbfs,\bm{\eta})|^{2} \prod_{j=1}^n|\eta_j|^{\tau'_{j}} d\bm{\eta} \bigg|^{\frac{1}{2H^{(2)}_0}} d\vec{\bfs}\nonumber \\
	=&\frac{1}{(n!)^{\frac{1}{H^{(2)}_0}}}\sum_{\sigma}\int_{{\bT_1^{\sigma }}} \bigg| \int_{\RR^{n}} |\widehat{f}_{n,\sigma}^{(t,x)}(\vbfs,\bm{\eta})|^{2} \prod_{j=1}^n|\eta_j|^{\tau'_{j}} d\bm{\eta} \bigg|^{\frac{1}{2H^{(2)}_0}} d\vec{\bfs} \nonumber\,.
\end{align}   
Similar to \eqref{Ineq:J1^p_2}, we get
\begin{align}\label{eq.530}
	\int_{\RR_+^{n}} \Big| \Psi_{n,\bm{\tau'}}(\vbfs,\vbfs)\Big|^{\frac{1}{2H^{(2)}_0}} d\vec{\bfs} \leq \frac{C^n}{(n!)^{\frac{1}{H^{(2)}_0}}}\sum_{\sigma}\int_{{\bT_1^{\sigma }}} \prod_{i=1}^{n} |s_{\sigma(i+1)}-s_{\sigma(i)}|^{-\frac{\varrho'_{j}}{2H^{(2)}_0}}d\vec{\bfs}
\end{align}
with $\varrho'_{j}$ being defined by
\begin{equation*}
	\varrho'_{j}=\frac{1+\tau'_{j}}{2}=\begin{cases}
		\frac 12\,, &\pi_j=0\,;\\
		\frac 12+\frac{k'q}{2}(1-2H)\,, &\pi_j=1\,;\\
		\frac 12+k'q(1-2H)\,, &\pi_j=2\,.
	\end{cases}
\end{equation*}
The finiteness of \eqref{eq.530} requires that for $j=1, \cdots, n$,   $\frac{\varrho'_{j}}{2H^{(2)}_0}$ satisfies the following \emph{Hardy-Littlewood-Sobolev conditions}:
\begin{align}\label{Rmk.condi_p>}
	\frac{\varrho'_{j}}{2H^{(2)}_0}<1 \Leftrightarrow&\ \frac{\frac 12+k'q(1-2H)}{2H^{(2)}_0}<1 \nonumber \\
	\Leftrightarrow&\ \   1<q<\frac{3}{2k'(1-2H)+4m'(1-H_0)}\,.
\end{align}
Under the above condition  we have from \eqref{eq.530}
\begin{align}
	\bigg[\int_{\RR_+^{n}} \Big| \Psi_{n,\bm{\tau'}}^{(q)}(\vbfs,\vbfs)\Big|^{\frac{1}{2H^{(2)}_0}} d\vec{\bfs} \bigg]^{\frac{2H^{(2)}_0}{q}} 
	\leq& \frac{C^n}{(n!)^{\frac 2q}}\bigg[n!\cdot \frac{\prod_{i=1}^n\Gamma(1-\varrho'_{j}/2H^{(2)}_0)}{\Gamma(n-\sum_{j=1}^n \varrho'_{j}/2H^{(2)}_0 +1)} \bigg]^{\frac{2H^{(2)}_0}{q}}\nonumber \\
	\leq& C^n (n!)^{-\frac 2q+[\frac{1}{2q}+\frac{k'}{2}(1-2H)]}\,,  \label{Ineq:Step2} 
\end{align}
where the last inequality  follows from  the Stirling  formula and the fact that
\begin{equation*}
	\sum_{j=1}^n \varrho'_{j}= n\Big[\frac 12+\frac{k'q}{2}(1-2H)\Big] \,.
\end{equation*}
Incorporating \eqref{Ineq:J1^p_2} from Step 1 and \eqref{Ineq:Step2} from Step 2 into \eqref{def_j1j2}, we obtain {for any $(t,x)\in\RR_+\times\RR^d$}
\begin{align*}
	\EE [ u(t, x)^2]\leq& \sum_{n=0}^{\infty} n!\sum_{\bm{\pi}} \cJ_1\times \cJ_2 
	\leq \sum_{n=0}^{\infty} C^n (n!)^{1-\frac 2q+[\frac{1}{2q}+\frac{k'}{2}(1-2H)]}\,.
\end{align*}
This proves   $\EE [ u(t, x)^2]<\infty$ if
\begin{align}\label{condi_q<}
	1-\frac 2q+\Big[\frac{1}{2q}+\frac{k'}{2}(1-2H)\Big]< 0 \Leftrightarrow   1<q<\frac{3}{k'(1-2H)+2}  \,.
\end{align}

\textbf{Step 3: Constraints on $p$ ($q$), $m$ ($m'$), and $k$ ($k'$).}
The conditions \eqref{Rmk.condi_q<} and \eqref{Rmk.condi_p>} are summarized as
 \begin{align}
	1<q<\frac{3}{2k'(1-2H)+4m'(1-H_0)} \quad \text{and} \quad 1<q<\frac{1}{2m'(1-H_0)}\,. \label{Rmk.condi_q}	
\end{align}
Remember that in \textbf{Step 2} and \textbf{Step 3}, in order to ensure $\EE[u(t,x)^2]<+\infty$, it is necessary to have
\begin{align}
	&1<p<\frac{5}{2k(1-2H)+8m(1-H_0)} \wedge \frac{1}{2m(1-H_0)}\,;\label{Rmk.condi_p} \\
	&1<q<\frac{3}{2k'(1-2H)+4m'(1-H_0)} \wedge \frac{1}{2m'(1-H_0)}\,;\label{Rmk.condi_q} \\
	&1<q<\frac{3}{k'(1-2H)+2} \,. \label{Rmk.condi_qq}
\end{align}
To determine the conditions on $H$ and $H_0$ under which \eqref{Rmk.condi_p}-\eqref{Rmk.condi_qq} are met, we initially get rid of  the wedge symbol ``$\wedge$" in the above first two inequalities \eqref{Rmk.condi_p} and \eqref{Rmk.condi_q}. To this end,    we   consider the following four cases:
\begin{equation*}
	\begin{cases}
\hbox{Case 1}: \   		\frac{5}{2k(1-2H)+8m(1-H_0)} < \frac{1}{2m(1-H_0)}\,,\quad \frac{3}{2k'(1-2H)+4m'(1-H_0)} < \frac{1}{2m'(1-H_0)}\,; \\
\hbox{Case 2}: \ 		\frac{5}{2k(1-2H)+8m(1-H_0)} < \frac{1}{2m(1-H_0)}\,,\quad \frac{3}{2k'(1-2H)+4m'(1-H_0)} > \frac{1}{2m'(1-H_0)}\,; \\
\hbox{Case 3}: \ 		\frac{5}{2k(1-2H)+8m(1-H_0)} > \frac{1}{2m(1-H_0)}\,,\quad \frac{3}{2k'(1-2H)+4m'(1-H_0)} < \frac{1}{2m'(1-H_0)}\,; \\
\hbox{Case 4}: \ 		\frac{5}{2k(1-2H)+8m(1-H_0)} > \frac{1}{2m(1-H_0)}\,,\quad \frac{3}{2k'(1-2H)+4m'(1-H_0)} > \frac{1}{2m'(1-H_0)}\,. \\
	\end{cases}
\end{equation*}
We will provide a detailed treatment of Case 1, and the approach for the other cases is analogous.


\noindent \emph{Case 1:} In this case,   the constraints \eqref{Rmk.condi_p}-\eqref{Rmk.condi_qq} on $p,\ q$  become the following conditions:
\begin{equation}\label{condi_case1}
\begin{cases}
&1<q<\frac{3}{k'(1-2H)+2}\,;\\
&1<q<\frac{3}{2k'(1-2H)+4m'(1-H_0)} < \frac{1}{2m'(1-H_0)}\,; \\
&1<p<\frac{5}{2k(1-2H)+8m(1-H_0)} < \frac{1}{2m(1-H_0)}\, .
\end{cases}
\end{equation}
Denote the region
\begin{align}\label{defi_s}
\mathcal{S}:=&\lt\{(H,H_0)\in(0,\frac14)\times(\frac12,\frac34): \  H+2H_0>\frac 54,\ H+H_0\leq \frac34,\ 12H+4H_0>3 \rt\}.
\end{align}
\begin{figure}[H]
	\centering
	\includegraphics[width=0.6\textwidth]{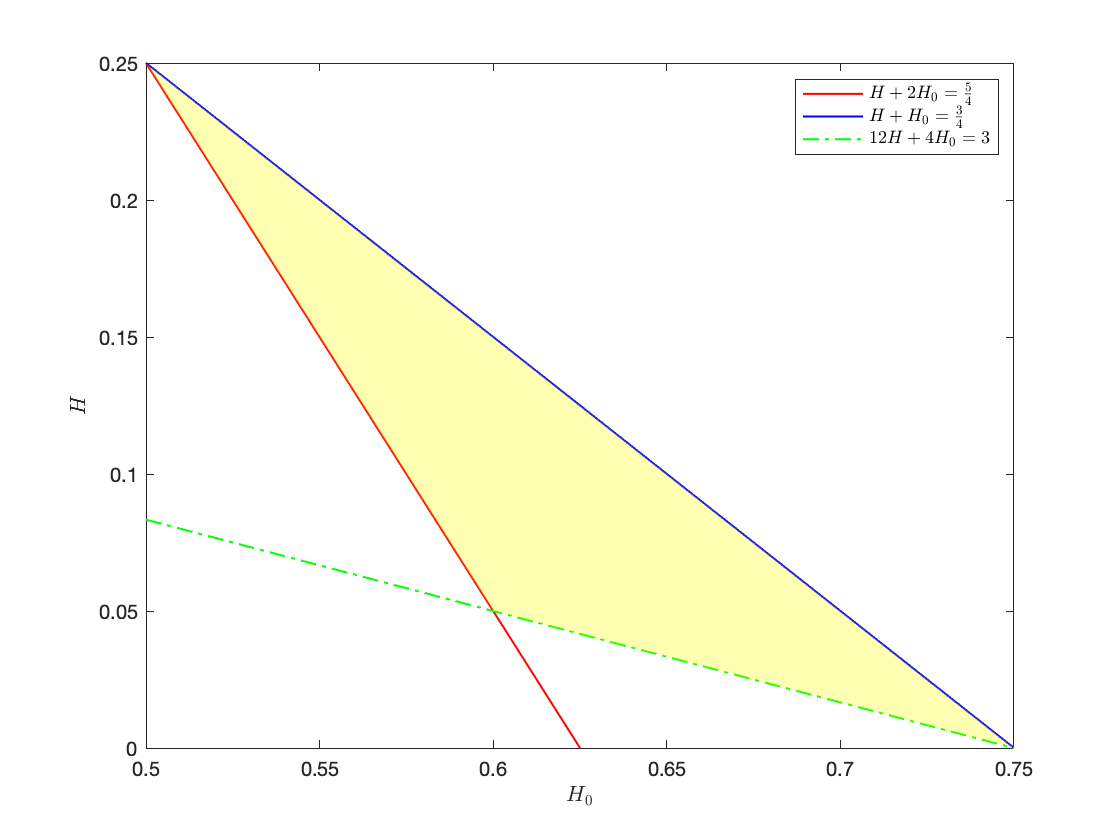}
	\caption{The region $\mathcal{S}$ in \eqref{defi_s}}
	\label{Fig.Case1}
\end{figure}

We shall show  that for any $(H,H_0)\in\mathcal{S}$, there exist $p,q \in[1,\infty)$, and $m, m'\in[0,1],\ k, k'\in[0,1]$ such that conditions in. \eqref{condi_case1} hold.
For any arbitrarily  small $\varepsilon>0$, let us choose
\begin{align}
& q=\frac{3}{k'(1-2H)+2}-\varepsilon,\label{chosen_q}\\
& p=\frac{q}{q-1}=\frac{3-\varepsilon}{1-k'(1-2H)-\varepsilon},\label{chosen_p}\\
& m'=\frac{2-k'(1-2H)}{4(1-H_0)}\label{chosen_m}.
\end{align}
Thus it suffices to show that for any $(H,H_0)\in \cS$, there exists $k'\in[0,1]$ such that the parameters $(p, q, m')$ given by \eqref{chosen_q}-\eqref{chosen_m} satisfy conditions  \eqref{condi_case1}  for  sufficiently small $\varepsilon>0$.

 Firstly, we need to ensure $p,q \geq 1$ and $0\leq m'\leq 1$. Note that
\[q\geq 1\Leftrightarrow k'\leq \frac{1-2\ep}{(1-2H)(1+\ep)}\sim \frac{1}{1-2H} \,,\]
which is implied by $k'\leq 1$  and $H>0$. Now $p\geq 1$ is equivalent to $q\ge 1$. Besides,
\begin{equation}\label{condi_k1}
0\leq m'\leq 1\Leftrightarrow k'\geq\frac{4H_0-2}{1-2H}\,.
\end{equation}
In order to be able to  find a  $k'\in[0,1]$ so that the above inequality holds true we need
$$\frac{4H_0-2}{1-2H}\leq 1\Leftrightarrow H+2H_0\leq\frac32,$$
which holds for $(H,H_0)\in\cS$ since $H+H_0\leq\frac34$ and then $H_0<\frac34$.

 Next, we shall verify the conditions in 
\eqref{condi_case1}  with  the   $p,\ q$ and $m'$ given by  \eqref{chosen_q}-\eqref{chosen_m}. It is clear that we can take $\ep=0$.
\begin{enumerate}
\item[(i)]  The     inequality  in the first line of \eqref{condi_case1} holds obviously with the choice of $q$ in \eqref{chosen_q}.
\item[(ii)] For the  $q$ defined by   \eqref{chosen_q} we now  show the   inequality in the second line  of   \eqref{condi_case1}:
\begin{align*}
	q<\frac{3}{2k'(1-2H)+4m'(1-H_0)}\,,
\end{align*}
which  is equivalent to
\begin{align*}
	 \frac{3}{k'(1-2H)+2}\leq &\frac{3}{2k'(1-2H)+4m'(1-H_0)} \\
	 \Leftrightarrow\quad    m'\leq &\frac{2-k'(1-2H)}{4(1-H_0)}\,.
\end{align*}
This holds with the choice of $m'$ in \eqref{chosen_m}.
\item[(iii)] Recall that the third inequality in the second  line of  \eqref{condi_case1}  is
$$\frac{3}{2k'(1-2H)+4m'(1-H_0)} < \frac{1}{2m'(1-H_0)} $$
which is equivalent to
\begin{equation*}
m'(1-H_0)<k'(1-2H)\,.
\end{equation*}
By using  the definition \eqref{chosen_m} of $m'$, the constraint placed on $k'$ as mentioned above is equivalent to
    \begin{equation}\label{condi_k2}
    \frac{2}{5(1-2H)}<k'<1\,.
    \end{equation}
\item[(iv)] The second inequality of the third  line   of \eqref{condi_case1},  by choice of $p$ in \eqref{chosen_p},  is equivalent to
\begin{align*}
	 & 25-12(H+2H_0)<k'(1-2H)+24m'(1-H_0)\,.
\end{align*}
With  $m'$ defined in \eqref{chosen_m}  this can be rewritten as
    \begin{equation}\label{condi_k3}
    0<k'<\frac{12(H+2H_0)-13}{5(1-2H)}.
    \end{equation}
\item[(v)]  Now we consider the third inequality in the last  line    of \eqref{condi_case1}.
Noticing that $k=1-k'$ and $m=1-m'$ and substituting the value of $m'$,  this yields
    \begin{equation}\label{condi_k4}
    0<k'<\frac{2+4(H_0-2H)}{5(1-2H)}.
    \end{equation}
\end{enumerate} 

In summary, when $k'$ fulfills the requirements specified in \eqref{condi_k2}, \eqref{condi_k3}, and \eqref{condi_k4}, it ensures 
\eqref{condi_case1}. The conditions met by $k'$ can be concisely summarized as follows.
\begin{equation}\label{condi_k'}
 \lt\{\frac{2}{5(1-2H)}\vee \frac{4H_0-2}{1-2H} \rt\}<k'< \lt\{\frac{12(H+2H_0)-13}{5(1-2H)}\wedge \frac{2+4(H_0-2H)}{5(1-2H)}\wedge 1 \rt\}\,.
\end{equation}
Notice  that
\[12(H+2H_0)-13\leq 2+4(H_0-2H)\Leftrightarrow H+H_0\leq \frac34 \] and
\[\frac{12(H+2H_0)-13}{5(1-2H)}\leq 1\Leftrightarrow 11H+12H_0\leq 9,\]
which holds under the condition $H+H_0\leq\frac34$. Thus, on the set $\cS$  the condition  \eqref{condi_k'}  is reduced  to
\begin{equation}\label{condi_k}
\lt\{ \frac{2}{5(1-2H)}\vee \frac{4H_0-2}{1-2H}\rt\} <k'<\frac{12(H+2H_0)-13}{5(1-2H)}.
\end{equation}

Finally, it remains to show that when $(H, H_0)\in \mathcal{S}$, there is $k'$ satisfying \eqref{condi_k} which amounts to
\begin{equation}\label{condi_k1}
 \lt\{\frac{2}{5(1-2H)}\vee \frac{4H_0-2}{1-2H}\rt\} \le \frac{12(H+2H_0)-13}{5(1-2H)}.
\end{equation}
We shall prove the above inequality assuming $H_0\le 3/5$ and $H_0>3/5$ respectively.
First, let us assume $H_0\leq \frac35$. In this case, we have $\frac{2}{5(1-2H)}\leq \frac{4H_0-2}{1-2H}$. Then \eqref{condi_k1} becomes
\[\frac{2}{5(1-2H)}<\frac{12(H+2H_0)-13}{5(1-2H)}\Leftrightarrow  H+2H_0>\frac54\, \]
which corresponds to the red solid line in Figure \ref{Fig.Case1}.
Next, we assume $H_0>\frac35$ which means $\frac{2}{5(1-2H)}>\frac{4H_0-2}{1-2H}$.
In this case, \eqref{condi_k1} becomes
\[\frac{4H_0-2}{1-2H}<\frac{12(H+2H_0)-13}{5(1-2H)}\Leftrightarrow 12H+4H_0>3\,,\]
which aligns with the green dashed line in Figure \ref{Fig.Case1}. 
Therefore, for any given $(H,H_0)\in\mathcal{S}$ (as illustrated in Figure \ref{Fig.Case1}) we can select an appropriate $k'$ such that condition \eqref{condi_k} is satisfied.

\emph{Case 2-Case 4:} These cases can be approached with the same methodology as Case 1. Nevertheless, we find that it does not allow for an extension of the solvability region. Consequently, we will omit the details.

As a consequence, on the region $\cS$ (refer to Figure \ref{Fig.Case1}), we have {for any $(t,x)\in\RR_+\times\RR^d$}
\[\EE[u(t,x)^2]=\sum\limits_{n=1}^{\infty}\EE[u_n(t,x)^2]<+\infty.\]
Since   the above convergence was proved  in \cite{CH2021}   when $H_0+H>3/4$,
we complete the proof.
\end{proof}

\appendix
\section{Some lemmas}\label{appenA}

\begin{lemma}\label{l.3.2}  
Let $0<\al, \be<1$ with $\al+\be
 >1$. Then there is a constant $c\geq 1$ independent of $\vare>0$ such that for all $x\in (0, 3\vare )$,
\begin{equation}\label{e:3.6}
c^{-1}   x^{1-(\al+\be)}  \leq \int_0^\vare u^{-\al} (u+x)^{-\be} du\leq c  x^{1-(\al+\be)}   .
\end{equation}
\end{lemma}
This lemma can be found in \cite{CH2021}, and the following lemma can be found in \cite{HNS}.


\begin{lemma}[Hardy-Littlewood-Sobolev inequality]\label{Hardy-Littlewood-Sobolev}
For any $\varphi\in L^{1/H_0}(\RR^n)$, it holds
\begin{equation}\label{ineq_Hardy-Littlewood-Sobolev}
\int_{\RR^n}\int_{\RR^n}\varphi(\vec{\mathbf{r}})\varphi(\vec{\mathbf{s}})\prod_{i=1}^n|s_i-r_i|^{2H_0-2}d\vec{\mathbf{r}}d\vec{\mathbf{s}}\leq   C_{H_0} ^n   \lt(\int_{\RR^n}|\varphi(\vec{\mathbf{r}})|^{1/H_0}d\vec{\mathbf{r}}\rt)^{2H_0},
\end{equation}
where $C_{H_0}>0$.
\end{lemma}

The H\"older-Young-Brascamp-Lieb  type inequality has been studied and extended by various groups after the seminar work of Brascamp and Lieb \cite{BL1976}.  We refer \cite[Appendix A]{ALNS2015}, \cite{Lehec2014} and the references therein to the readers on this interesting topic. Below, we present the homogeneous and non-homogeneous H\"older-Young-Brascamp-Lieb inequalities that we are going to use in this work.  { Let $\cH=\RR^n$, $\cH_1=\RR^{n_1}$,\dots, $\cH_m=\RR^{n_m}$ be finite dimensional Euclidean spaces equipped with the corresponding Lebesgue measure $\mu(\cdot)$.} We remark that $\cH$, $\cH_1$,\dots, $\cH_m$ can be Hilbert space in general (e.g., \cite{BCCT2008}). For $j=1,\cdots,m$, let $f_j:\cH_j\to \RR_+$ be nonnegative function satisfying  $f_j\in L^{p_j}(\mu)$ for $p_1,\cdots,p_m\in [1,\infty)$  and let $l_j:\cH\to \cH_j$ be surjective linear transformations.

 We recall the (local) multilinear functional
\begin{equation*}
	\Lambda(f_1,\dots,f_m):=\int \prod_{j=1}^m f_j(l_j(\bfx))\prod_{j=1}^n d\mu(x_j)\,,
\end{equation*}
and the H\"older-Young-Brascamp-Lieb  type inequality
\begin{equation}\label{Ineq.HYBL}
	|\Lambda(f_1,\dots,f_m)|\leq K_{\cH} \prod_{j=1}^m \|f_j\|_{L^{p_j}(\mu)}\,,
\end{equation}
for some positive finite  constant $K_{\cH}$. There are many works focusing on the conditions under which
\eqref{Ineq.HYBL} holds true.  {We shall use  a specific condition, the nonhomogeneous H\"older-Young-Brascamp-Lieb inequality, as presented in \cite[Theorem 2.2]{BCCT2010} (also discussed in \cite{BCCT2008}).}

\begin{Thm}[Nonhomogeneous H\"older-Young-Brascamp-Lieb  inequality]\label{Thm.HYBL_N}
	Let $\cH$, $\cH_1$,\dots, $\cH_m$, be finite dimensional Hilbert spaces equipped with { the corresponding Lebesgue measure $\mu(\cdot)$ and $f_j\in L^{p_j}(\mu)$.} Consider a local version of \eqref{Def.Multi}
	\begin{equation}\label{Def.Multi_loc}
	\Lambda_{loc}(f_1,\dots,f_m):=\int_{\{\bfx\in \cH:|\bfx|\leq 1\}} \prod_{j=1}^m f_j(l_j(\bfx))\prod_{j=1}^n d\mu(x_j)\,.
	\end{equation}
	A necessary and sufficient condition that H\"older-Young-Brascamp-Lieb  type inequality \eqref{Ineq.HYBL} holds for  $\Lambda_{loc}$, with $0<K_{\cH}<\infty$ for all nonnegative measurable functions $f_j$, is that every subspace $V\subseteq \cH$ satisfies the following dimension condition:
	\begin{equation}\label{Cond.codim}
		codim_{\cH} (V)\geq \sum_{j}p_j^{-1} codim_{\cH_j} (l_j(V))\,.
	\end{equation}
\end{Thm}

Using the above theorem, we can show an important ingredient in the proof of Theorem \ref{main_result1}.
\begin{lemma}\label{verify_HYBL}
Under the condition  $H_0\geq 1/2$ and the conditions \eqref{cond.0<H<1_a}-\eqref{cond.0<H<1_d} of Theorem \ref{main_result1}, the following dimensional conditions hold when $co\dim_{\cH} (V)=1,2,3,4$:
\begin{equation*}
	co\dim_{\cH} (V)\geq \sum\limits_{j=1}^6  z_j\cdot  co\dim_{\cH_j} (l_j(V))\,.
\end{equation*}
\end{lemma}

\begin{proof}
We show the  lemma  for $co\dim=1 ,2, 3, 4$ separately.

\noindent\textbf{Case 1:}  $co\dim_\cH (V)=1$, i.e. $\dim_\cH (V)=3$. In this case, we only need to consider $\rank([l_j]_{j\in J})=1$, i.e. $|J|=1$. {From \eqref{Cond_z's}, $z_j\leq 1$ ($j=1,\cdots,6$) under the condition $H_0\geq 1/2$ and the condition \eqref{cond.0<H<1_a}. Thus, the dimension condition \eqref{e.dimension} holds.}

\noindent\textbf{Case 2:}  $co\dim_\cH (V)=2$, i.e. $\dim_\cH (V)=2$. We shall choose $J$ such that $\rank([l_j]_{j\in J})=2$.

Note that we only have to deal with $|J|=3$ since the dimension condition \eqref{e.dimension} automatically holds for $|J|=2$ and it is impossible to have  $\rank([l_j]_{j\in J})=3$ for $|J|=4,5,6$. The only possible case is $\rank([l_3,l_4,l_6])=2$. Then, the condition \eqref{e.dimension} in this case is equivalent to
\begin{align}\label{condi_case2}
	2\, \geq& z_3+z_4+z_6> \rho_2+\rho_2+\gamma_0 \nonumber\\
	=&(\frac32 d+\frac12\be^*-2|H|-\al_2)+2-2H_0 \nonumber\\
	\Leftrightarrow\qquad & |H|+2H_0>d+(\frac12d-|H|+\frac12\be^*)-\al_2\,.
\end{align}
Since $\al_2\in\{\frac12d-|H|+\frac12\be^*, d-2|H|+\be^*\}$, it is obvious that \eqref{condi_case2} holds under the condition \eqref{cond.0<H<1_b}:
\begin{align}\label{condi_2}
|H|+2H_0>d.
\end{align}


\noindent\textbf{Case 3:}  $co\dim_\cH (V)=3$, i.e. $\dim_\cH (V)=1$. We shall choose $J$ such that $\rank([l_j]_{j\in J})=3$.

For $|J|=3$, the dimension condition \eqref{e.dimension} holds automatically since $z_i\leq1$ for $i=1,\cdots,6$. Moreover, it is impossible to get  $\rank([l_j]_{j\in J})=3$ for $|J|=5,6$ by computation of the relevant ranks. Thus, we only need to consider $|J|=4$. By simple analysis, we know that there are four cases:
\begin{align*}
	&\rank([l_1,l_2,l_5,l_6])= \rank([l_1,l_3,l_4,l_6]) \\
	&=\rank([l_2,l_3,l_4,l_6])=\rank([l_3,l_4,l_6,l_5])=3\,.
\end{align*}
We treat $J=\{1,2,5,6\}$ as an example. The rest is similar and  we omit the details.

Now, we have $\rank([l_1,l_2,l_5,l_6])=3$, $\dim(\cap_{j=1,2,5,6} \ker (l_j))=1$ and $co\dim_\cH (V)=3$. In fact, we can obtain that $V=span\{(1,1,1,1)\}=\cap_{j=1,2,5,6} \ker (l_j)$ and
\begin{align*}
	&co\dim_{\cH_{1}} (l_1(V))=1\,,co\dim_{\cH_{3}} (l_3(V))=0\,,co\dim_{\cH_{5}} (l_5(V))=1\,,\\
	&co\dim_{\cH_{2}} (l_2(V))=1\,,co\dim_{\cH_{4}} (l_4(V))=0\,,co\dim_{\cH_{6}} (l_6(V))=1\,.
\end{align*}
Then, in this case the condition \eqref{e.dimension}  is analogous to
\begin{align}\label{condi_case3}
	3\, \geq& z_1+z_2+z_5+z_6\nonumber\\
	>& \rho_1+\rho_1+\gamma_0+\gamma_0 \nonumber\\
	=&(\frac32d_*-2H_*)+4-4H_0\,.
\end{align}
It is easy to see that  \eqref{condi_case3} holds under the third condition \eqref{cond.0<H<1_c} of \eqref{cond.0<H<1}:
\begin{align}\label{condi_3}
H_*+2H_0>\frac34d_*+\frac12.
\end{align}
\\

\noindent\textbf{Case 4:}  $co\dim_\cH (V)=4$, i.e. $\dim_\cH (V)=0$. We shall choose $J$ such that $\rank([l_j]_{j\in J})=4$.

Obviously, we only need to consider $|J|=4,5,6$. We first deal with $|J|=6$. This means
\[
 V=\bigcap_{j=1}^6 \ker (l_j)= \{0,0,0,0\}\,.
\]
Thus, the dimensional condition \eqref{e.dimension} in this case  is equivalent to the following:
 \begin{align}
	4\, \geq& z_1+z_2+z_3+z_4+z_5+z_6\nonumber \\
	>& \rho_1+\rho_1+\rho_2+\rho_2+\gamma_0+\gamma_0 \nonumber \\
	=&(\frac32d_*-2H_*)+2(2-2H_0) +(\frac32d+\frac12\be^*-2|H|-\al_2) \nonumber
\end{align}
{which amounts to
\begin{equation}
	|H|+2H_*+4H_0>d+\frac32d_* +(\frac12d-|H|+\frac12\be^*-\al_2). \label{Ineq.case1}
\end{equation}}
Since $\al_2\in\{\frac12d-|H|+\frac12\be^*, d-2|H|+\be^*\}$,
 we know \eqref{Ineq.case1} holds under the condition \eqref{cond.0<H<1_d}:
 \begin{align}\label{condi_4}
 |H|+2H_*+4H_0>d+\frac32d_*.
 \end{align}

 When $|J|=4,5$, the corresponding dimension conditions \eqref{e.dimension} can be implied by \eqref{Ineq.case1}. 
Thus, the proof of the lemma is complete.
\end{proof}



Next, we give some estimations for the alternative proof of the sufficiency of Theorem \ref{main_result1}.
\begin{lemma}\label{l.3.1}  For any $\al, \be, \ga\in (0, 1)$ and $a, b\in \RR$
denote
\begin{empheq}[left=\empheqlbrace]{align}
 \mathfrak{B}_1(a,b):=\int_0^{\frac12 } u^{-\beta}(1-u)^{-\al}
 |au-b|^{-\ga}du\,; \label{e.3.9}  \\
 \fB_2(a,b):=\int_ {\frac12 }^1 u^{-\be}(1-u)^{-\al}
 |au-b|^{-\ga}du. \label{e.3.10}
\end{empheq}
Then
\begin{equation}
\fB_1(a,b)
 \ls
 \begin{cases}
 |a|^{-q}|b|^{-\ga +q} \qquad\hbox{for any $q\in [0,  1-\be]  $ } \quad \hbox{if $\be+\ga>1$};\\
(|a|+|b|)^{-\ga} \ls
 |a|^{-q_1}  | b-a|^{- q_2}  |b |^{-\ga +q_1+q_2}\\
  \qquad\qquad \qquad\quad \hbox{for any $q_1, q_2, q_1+q_2\in [0, \ga]$}      \quad \hbox{if $\be+\ga<1$},\\
 \end{cases}\label{e.3.11}
 \end{equation}
 and
 \begin{equation}
\fB_2(a,b)
 \ls
 \begin{cases}
 |a|^{-q}|b-a|^{-\ga +q} \qquad \hbox{for any $q\in [0,  1-\al] $ } \quad \hbox{if $\al+\ga>1$};\\
(|a|+|b-a|)^{-\ga} \ls
 |a|^{-q_1}  | b-a|^{- q_2}  |b |^{-\ga +q_1+q_2}\\
  \qquad\qquad \qquad \qquad\quad\hbox{for any $q_1, q_2, q_1+q_2\in [0, \ga]$}      \quad \hbox{if $\al +\ga<1$}.\\
 \end{cases}\label{e.3.12}
 \end{equation}
\end{lemma}
\begin{proof}  First let $\be+\ga>1$. Since for $x\geq 0,$ $\be\in (0, 1)$, $\ga\in (0, 1)$    it holds that
\begin{align}
\int_0^x|\xi|^{-\be}|1-\xi|^{-\ga}d\xi\ \ls |x|^{-\be+1} \wedge 1\,. \label{Ineq:A.16}
\end{align}
Then for any $\rho_1\in[0,1]$,
\begin{align}\label{0-1/2dw}
\int_0^{\frac12}  w ^{-\be}(1-w)  ^{-\al}| {a}w-b |^{-\gamma}dw 
&\ls\int_0^{\frac12}w^{-\be }| aw-b |^{-\gamma}dw\nonumber\\
 &\ls |b|^{-\ga}\int_0^{\frac12}w^{-\be }\lt|\frac{|a|}{|b|}w-1\rt|^{-\ga}dw\nonumber\\
&\ls|b  |^{-\gamma-\be +1}\cdot| {a}|^{\be-1}\times\int_0^{\frac{| {a}|}{2|b  |}}\xi^{-\be }|1-\xi|^{-\gamma}d\xi\nonumber\,.
\end{align}
{We can apply \eqref{Ineq:A.16} in the last integral to get
\begin{align}
\int_0^{\frac12}  w ^{-\be}(1-w)  ^{-\al}| {a}w-b |^{-\gamma}dw &\ls|b |^{-\gamma-\be +1}\cdot| {a}|^{\be -1}\cdot \lt(\frac{| {a}|}{|b |}\wedge 1\rt)^{ -\be +1  }  \nonumber\\
&\ls|b |^{-\gamma-\be +1}\cdot| {a}|^{\be -1}\cdot \lt(\frac{| {a}|}{|b |}\rt)^{ (-\be +1)\rho_1  }  \nonumber\\
&=| {a}|^{-q_1}|b |^{-\gamma+q_1 }\,,
\end{align}}
where $q_1=(1-\be )(1-\rho_1)\in[0,1-\be ]$ and we make substitution  $\xi=\frac{|{a}|}{|b|}w$ in the third inequality.
Here we need that   $\be  ,\ \gamma\in (0, 1),\ \be +\gamma>1$.

Make substitution  $\widetilde{w}=1-w$. Then for $0\leq\rho_2\leq1$,
\begin{align}\label{1/2-1dw}
\int_{\frac12}^1|1-w|^{-\al}w^{-\be }| w-b_v|^{-\gamma}dw=\int_0^{\frac12}\widetilde{w}^{-\al}|1-\widetilde{w}|^{-\be }| {a}-b - {a}\widetilde{w}|^{-\gamma}d\widetilde{w}\,.
\end{align}
This can be used to show the first inequality in \eqref{e.3.12}.

Now if $\be\in (0, 1)$, $\ga\in (0, 1)$ and $\be+\ga<1$, then
\begin{align*}
\int_0^x|\xi|^{-\be}|1-\xi|^{-\ga}d\xi\ \ls \frac{|x|^{-\be+1} (1+|x|^{1-\be-\ga} )}{1+
|x|^{-\be+1}} \qquad \forall \ x>0 \,.
\end{align*}
With  the same argument as above we have
\begin{align}\label{0-1/2dw}
\int_0^{\frac12}&  w ^{-\be }(1-w)  ^{-\beta}| {a}w-b |^{-\gamma}dw\nonumber\\
&\ls|b  |^{-\gamma-\be +1}\cdot| {a}|^{\be-1}\times\int_0^{\frac{| {a}|}{2|b  |}}\xi^{-\be }|1-\xi|^{-\gamma}d\xi\nonumber\\
&\ls|b |^{-\gamma-\be +1}\cdot| {a}|^{\be -1}\cdot
\frac{|\frac ab |^{-\be+1} (1+|\frac ab|^{1-\be-\ga} )}{1+
|\frac ab|^{-\be+1}}\,.
\nonumber\\
&\ls (|a|+|b|)^{-\ga} \,.
\end{align}
This shows the first part of the
 second  inequality  in  \eqref{e.3.11}.
Since  $|a|+|b |\ge |a|\vee |b-a|\vee |b|$ implies the following inequalities:
\begin{align*}
 (|a|+|b |)^{-\ga}
\le &\left[ |a| \vee |b-a|\vee |b| \right]^{-
q_1}\left[ |a| \vee |b-a|\vee |b| \right]^{-
q_2}\left[ |a| \vee |b-a|\vee |b| \right]^{-\ga+
q_1+q_2}\\
\le & |a|  ^{-
q_1}  |b-a| ^{-
q_2}  |b|  ^{-\ga+
q_1+q_2}\,.
\end{align*}
This shows the second part of the
 second  inequality  in  \eqref{e.3.11}.
     The second  inequality in \eqref{e.3.12} can be proved similarly since    $|a|+|b-a|\ge |a|\vee |b-a|\vee |b|$.
This completes the proof of the lemma.
\end{proof}

We immediately have the following Corollary.
\begin{corollary}\label{c.3.2}
  The following conclusion holds true.
\begin{equation}
\begin{cases}
\fB_1(a,b)
 \ls
 |a|^{-q}|b|^{-\ga +q} & \qquad \hbox{for any $q\in [0, (1-\be)\wedge \ga]$};
 \\
\fB_2(a,b)
 \ls
 |a|^{-q}|b-a|^{-\ga +q}& \qquad
 \hbox{for any $q\in [0, (1-\al)\wedge \ga]$}.
 \end{cases}
 \end{equation}
\end{corollary}

Now we give a lemma that indicates the upper bound of the integrations with respect to $\vbfs$ and $\vbfr$ in Proposition \ref{est_un}.
\begin{lemma}\label{lem_mainterm}
Let $0\le a<A<\infty$,  $0\le b<B<\infty$
and let $\al, \be, \ga\in [0, 1)$, $ \ga<  2- \be-(\be\vee \al) $.
Denote
\begin{equation}\label{mainterm}
\cI_{\alpha,\beta,\gamma}(A,a;B,b):=\int_{{a<s<A \atop b<r<B}} |A-s|^{-\alpha}|s-a|^{-\beta}|s-r|^{-\gamma}|B-r|^{-\alpha}|r-b|^{-\beta}dsdr\,.
\end{equation}
Then  for any $q_2\in (0\vee(\be+\ga-1), (1-\al)\wedge \frac{\ga}{2}]$,
\begin{align}\label{main_est}
\cI_{\alpha,\beta,\gamma}(A,a;B,b)\ls   &  |A-a|^{-\al-\be+1-  \frac{\ga}{2}}|B-b|^{-\alpha-\be+1- \frac{\ga}{2} } \\
   &\qquad +   |A-a|^{-\al-\be+1-q_2}|B-b|^{-\alpha-\be+1- q_2 }  |A-B |^{-\ga+2q_2}   \,.
\end{align}
\end{lemma}

\begin{proof}
Making substitution  $s=a+(A-a)w,\  r=b+(B-b)v$
and denoting   $\bar{a}:=A-a,\ b_v:=(B-b)v-(a-b)$ and $M:=|A-a|^{-\alpha-\beta+1}\cdot|B-b|^{-\alpha-\beta+1}$,  we have
\begin{align}\label{term_wv}
\cI_{\alpha,\beta,\gamma}(A,a;B,b)=
&M
\times\int_0^1\int_0^1 |1-w|^{-\alpha}w^{-\beta}
|1-v|^{-\alpha}v^{-\beta}|\bar{a}w-b_v|^{-\gamma} dwdv
\nonumber\\
=& M
\times\int_0^1
|1-v|^{-\alpha}v^{-\beta}  \tilde  \cI (v)  dv\,.
\end{align}
where by Corollary  \ref{c.3.2} $\tilde  \cI (v)$ is defined and bounded as follows.
\begin{align*}
\tilde  \cI (v)
 =&(\int_0^{\frac12}+\int_{\frac12}^1)\lt[|1-w|^{-\alpha}w^{-\beta}|\bar{a}w-b_v|^{-\gamma}\rt]dw\\
\le& \bar a^{-q_1}  |b_v|^{-\ga+q_1}+\bar a^{-q_2}|b_v-\bar a|^{-\ga+q_2} \,,
\end{align*}
where $q_1\in   (0, (1-\be)\wedge \ga)$ and $q_2\in(0, (1-\al)\wedge \ga)$.
Thus, we have
\begin{align}
\cI_{\alpha,\beta,\gamma}(A,a;B,b)
=  M\bar a^{-q_1}
\times&  \int_0^1
|1-v|^{-\alpha}v^{-\beta}     |b_v|^{-\ga+q_1}  dv\nonumber\\
+&M\bar a^{-q_2} \int_0^1
|1-v|^{-\alpha}v^{-\beta}     |b_v-\bar a|^{-\ga+q_2}  dv \nonumber\\
  =: M\bar a^{-q_1}
\times&   \fB_1+M\bar a^{-q_2} \fB_2    \,,\label{e.3.22}
\end{align}
where $\fB_1$ and $\fB_2$ are defined and dealt with as follows.
%
By Lemma \ref{l.3.1} we  obtain
\begin{align}
\fB_1
 :=&\int_0^1|1-v|^{-\alpha}|v|^{-\beta}\cdot |(B-b)v-(a-b)|^{-\gamma+q_1}dv\nonumber\\
 \ls &\int_0^{1/2}|v|^{-\beta}|(B-b)v-(a-b)|^{-\gamma+q_1}dv\nonumber\\
 &\qquad\qquad +\int_0^{1/2}|v|^{-\alpha}|(B-b)v-(B-a)|^{-\gamma+q_1}dv\nonumber\\
  \ls & |B-b|^ {-\gamma+q_1 } +(|B-b|+|B-a|)^{-\gamma+q_1}\nonumber\\
 \ls &  
 |B-b|^ {-\gamma+q_1 }  \,,\label{e.3.24}
\end{align}
where in obtaining the bound for the above first integral we need
\[
\begin{cases}
q_1>\ga+\be -1 \quad \hbox{ and hence}\quad
\  \ga+\be -1< q_1< (1-\be)\wedge \ga\\
q_1>\ga+\al-1\quad \hbox{ and hence}\quad
\  \ga+\al -1< q_1< (1-\be)\wedge \ga\\
\end{cases}
\]
 which is possible   by the assumption  $\ga<2-  \be
 - (\be \vee \al)$.
Now we  turn to bound $\fB_2$.
\begin{align}\label{cI_2}
 \fB_2&:=\int_0^1|1-v|^{-\alpha}|v|^{-\beta}|(B-b)v-(A-b)|^{-\gamma+q_2}dv\nonumber\\
&\ls \int_0^{1/2}|v|^{-\beta}|(B-b)v-(A-b)|^{-\gamma+q_2}dv\nonumber\\
& \qquad +\int_0^{1/2}|v|^{-\alpha}|(B-b)v-(B-A)|^{-\gamma+q_2}dv\nonumber\\
&\ls (|B-b|+|A-b|)^{-\gamma+q_2}+ |B-b|^{-q_3} |B-A|^{-\gamma+q_2+q_3 }\nonumber\\
&\ls  |B-b|^{-q_3} |B-A|^{-\gamma+q_2+q_3 }\,,
\end{align}
where $\ga+\be-1< q_2\le(1-\al)\wedge \ga$, $q_3\in [0, (1-\al)\wedge (\ga-q_2)]$.  Substituting the bounds \eqref{e.3.24}-\eqref{cI_2}
for $\fB_1$ and $\fB_2$ into \eqref{e.3.22}, we have
\begin{align*}
\cI_{\alpha,\beta,\gamma}(A,a;B,b)
\ls & M\bar a^{-q_1} |B-b|^{-\ga+ q_1}
 \nonumber\\
 &  +M\bar a^{-q_2}  |B-b|^{-q_3} |B-A|^{-\gamma+q_2+q_3 }  \nonumber\\
\ls & |A-a|^{-\al-\be+1-q_1}|B-b|^{-\alpha-\be+1-\ga+ q_1}  \nonumber\\
 &  + |A-a|^{-\al-\be+1-q_2}|B-b|^{-\alpha-\be+1 -q_3 }  |A-B |^{-\ga+ q_2+q_3}.
\end{align*}
If we integrate $r$ first, then we will have
\begin{align*}
\cI_{\alpha,\beta,\gamma}(A,a;B,b)
\ls & |B-b|^{-\al-\be+1-q_1}|A-a|^{-\alpha-\be+1-\ga+ q_1}  \nonumber\\
 &  + |B-b|^{-\al-\be+1-q_2}|A-a|^{-\alpha-\be+1-q_3  }  |A-B |^{-\ga+ q_2+q_3}   \,.
\end{align*}
Thus, by choosing $q_2=q_3\in (0\vee(\ga+\be-1), (1-\al)\wedge \frac{\ga}{2}]$, it holds
\begin{align*}
\begin{cases}
\cI_{\alpha,\beta,\gamma}(A,a;B,b)
\ls & |B-b|^{-\al-\be+1-q_1}|A-a|^{-\alpha-\be+1-\ga+ q_1}  \nonumber\\
 &  + |B-b|^{-\al-\be+1-q_2}|A-a|^{-\alpha-\be+1-q_2 }  |A-B |^{-\ga+ 2q_2 } \,;  \quad \hbox{and} \\
\cI_{\alpha,\beta,\gamma}(A,a;B,b)
\ls & |A-a|^{-\al-\be+1-q_1}|B-b|^{-\alpha-\be+1-\ga+ q_1}  \nonumber\\
 &  + |A-a|^{-\al-\be+1-q_2}|B-b|^{-\alpha-\be+1 -q_2 }  |A-B |^{-\ga+ 2q_2}\,.
 \end{cases}
\end{align*}

Denote
\begin{align*}
\mu:=
&\min (|A-a|^{-\al-\be+1-q_1}|B-b|^{-\alpha-\be+1-\ga+ q_1} \,,\\
&\qquad
|B-b|^{-\al-\be+1-q_1}|A-a|^{-\alpha-\be+1-\ga+ q_1} ) \\
=&|A-a|^{-\al-\be+1 }|B-b|^{-\alpha-\be+1 } \\
&\qquad \min (|A-a|^{ -q_1}|B-b|^{  -\ga+ q_1} \,, |B-b|^{ -q_1}|A-a|^{ -\ga+ q_1} ) \\
\le&  |A-a|^{-\al-\be+1 }|B-b|^{-\alpha-\be+1 } \\
&\qquad \left(|A-a|^{ -q_1}|B-b|^{  -\ga+ q_1}\right)^{1/2}
\left( |B-b|^{ -q_1}|A-a|^{ -\ga+ q_1} \right)^{1/2}  \\ 
=& |A-a|^{-\al-\be+1-\frac{\ga}{2}}|B-b|^{-\alpha-\be+1- \frac{\ga}{2} }\,.
\end{align*}
Then
\begin{align*}
\cI_{\alpha,\beta,\gamma}(A,a;B,b)
\ls &\mu   + |A-a|^{-\al-\be+1-q_2}|B-b|^{-\alpha-\be+1- q_2 }  |A-B |^{-\ga+2q_2}\\
   \le &  |A-a|^{-\al-\be+1-  \frac{\ga}{2}}|B-b|^{-\alpha-\be+1- \frac{\ga}{2} } \\
   &\qquad +   |A-a|^{-\al-\be+1-q_2}|B-b|^{-\alpha-\be+1- q_2 }  |A-B |^{-\ga+2q_2}   \,.
\end{align*}
This completes the proof of the lemma.
\end{proof}

\section*{Acknowledgement}
We would like to express our sincere gratitude to the anonymous reviewer and the editor of this paper for their invaluable contributions and constructive feedback, which have significantly enhanced the quality and clarity of our work. 
The authors thank Prof. Xia Chen   for bringing his work  \cite{Ch2}
to our attention.

\bibliographystyle{amsplain}
\bibliography{Ref_NS_cond}

\end{document}